\documentclass[11pt]{preprint}
\usepackage{difftrees} 
\usepackage[full]{textcomp}
\usepackage[osf]{newtxtext} 
\usepackage[cal=boondoxo]{mathalfa}
\usepackage{colortbl}

\usepackage{tikz-cd} 
\usetikzlibrary{cd} 
\usepackage[all,cmtip]{xy}
\usepackage{comment}

\usepackage{amssymb}
\usepackage{mathtools}
\usepackage{hyperref}
\usepackage{breakurl}
\usepackage{mhenvs}
\usepackage{mhequ} 
\usepackage{mhsymb}
\usepackage{booktabs}
\usepackage{tikz}
\usepackage{tcolorbox}
\usepackage{mathrsfs}
\usepackage[utf8]{inputenc}
\usepackage{longtable}
\usepackage{wrapfig}
\usepackage{rotating} 
\usepackage{subcaption}
\usepackage{mathrsfs}
\usepackage{epsfig}
\usepackage{microtype}
\usepackage{comment}
\usepackage{wasysym}
\usepackage{centernot}
\usepackage{enumitem}
\usepackage{bm}
\usepackage{stackrel}
\usepackage{graphicx}
\usepackage{axodraw}
\usepackage{xspace}
\usepackage{subcaption} 
\usepackage{epsfig} 
\usepackage{axodraw} 
\usepackage{xspace} 
\usepackage[toc,page]{appendix} 
\usepackage[all,cmtip]{xy} 
\usepackage{relsize} 
\usepackage{shuffle} 
\makeatletter
\newcommand{\globalcolor}[1]{%
  \color{#1}\global\let\default@color\current@color
}
\makeatother

\usetikzlibrary{calc}
\usetikzlibrary{decorations}
\usetikzlibrary{positioning}
\usetikzlibrary{shapes}
\usetikzlibrary{external}

\definecolor{blush}{rgb}{0.87, 0.36, 0.51}
	\definecolor{brightcerulean}{rgb}{0.11, 0.67, 0.84}
	\definecolor{greenryb}{rgb}{0.4, 0.69, 0.2}

\newif\ifdark
\darkfalse

\ifdark
\definecolor{darkred}{rgb}{0.9,0.2,0.2}
\definecolor{darkblue}{rgb}{0.7,0.3,1}
\definecolor{darkgreen}{rgb}{0.1,0.9,0.1}
\definecolor{franck}{rgb}{0,0.8,1}
\definecolor{pagebackground}{rgb}{.15,.21,.18}
\definecolor{pageforeground}{rgb}{.84,.84,.85}
\pagecolor{pagebackground}
\AtBeginDocument{\globalcolor{pageforeground}}
\definecolor{symbols}{rgb}{0,0.7,1}
\colorlet{connection}{red!80!black}
\colorlet{boxcolor}{blue!50}

\else

\definecolor{darkred}{rgb}{0.7,0.1,0.1}
\definecolor{darkblue}{rgb}{0.4,0.1,0.8}
\definecolor{darkgreen}{rgb}{0.1,0.7,0.1}
\definecolor{franck}{rgb}{0,0,1}
\definecolor{pagebackground}{rgb}{1,1,1}
\definecolor{pageforeground}{rgb}{0,0,0}
\colorlet{symbols}{blue!90!black}
\colorlet{connection}{red!30!black}
\colorlet{boxcolor}{blue!50!black}

\fi

\def\slash{\leavevmode\unskip\kern0.18em/\penalty\exhyphenpenalty\kern0.18em}
\def\dash{\leavevmode\unskip\kern0.18em--\penalty\exhyphenpenalty\kern0.18em}

\DeclareMathAlphabet{\mathbbm}{U}{bbm}{m}{n}

\DeclareFontFamily{U}{BOONDOX-calo}{\skewchar\font=45 }
\DeclareFontShape{U}{BOONDOX-calo}{m}{n}{
  <-> s*[1.05] BOONDOX-r-calo}{}
\DeclareFontShape{U}{BOONDOX-calo}{b}{n}{
  <-> s*[1.05] BOONDOX-b-calo}{}
\DeclareMathAlphabet{\mcb}{U}{BOONDOX-calo}{m}{n}
\SetMathAlphabet{\mcb}{bold}{U}{BOONDOX-calo}{b}{n}

\setlist{noitemsep,topsep=4pt,leftmargin=1.5em}

\DeclareMathAlphabet{\mathbbm}{U}{bbm}{m}{n}

\DeclareMathAlphabet{\mcb}{U}{BOONDOX-calo}{m}{n}
\SetMathAlphabet{\mcb}{bold}{U}{BOONDOX-calo}{b}{n}
\DeclareFontFamily{U}{mathx}{\hyphenchar\font45}
\DeclareFontShape{U}{mathx}{m}{n}{
      <5> <6> <7> <8> <9> <10>
      <10.95> <12> <14.4> <17.28> <20.74> <24.88>
      mathx10
      }{}
\DeclareSymbolFont{mathx}{U}{mathx}{m}{n}
\DeclareMathSymbol{\bigtimes}{1}{mathx}{"91}

\providecommand{\figures}{false}
{ \ifthenelse{\equal{\figures}{false}} {#1}{\[ {\rm Figure \ missing !} \]} }{}
\def\id{\mathrm{id}}

\def\CH{\mathcal{H}}

\def\CT{\mathcal{T}}

\tikzstyle{tinydots}=[dash pattern=on \pgflinewidth off \pgflinewidth]
\tikzstyle{superdense}=[dash pattern=on 4pt off 1pt]

\newcommand{\mcT}{\mathcal{T}}

\newcommand{\beq}{\begin{equation}}
\newcommand{\eeq}{\end{equation}}

\usepackage{empheq}

\newcommand{\mfe}{\mathfrak{e}}

\newcommand{\mfL}{\mathfrak{L}}

\newcommand{\mft}{\mathfrak{t}}

\newcommand{\mfp}{\mathfrak{p}}

\newcommand{\mff}{\mathfrak{f}}

\def\Labp{\mathfrak{p}}

\def\Labhom{\mathfrak{t}}

\def\Lab{\mathfrak{L}}

\def\${|\!|\!|}

\newcounter{theorem}

\newtheorem{ex}[theorem]{Example}

\newenvironment{DIFnomarkup}{}{} 

\newtheorem{assumption}{Assumption}
 
\theorembodyfont{\rmfamily}

\newcommand{\rrightarrow}{{\to\hskip -4.9mm\raise 1pt\hbox{$\to$}}}

\newfont{\indic}{bbmss12}

\def\Nabla_#1{\nabla_{\!#1}}

\makeatletter
\pgfdeclareshape{crosscircle}
{
  \inheritsavedanchors[from=circle] 
  \inheritanchorborder[from=circle]
  \inheritanchor[from=circle]{north}
  \inheritanchor[from=circle]{north west}
  \inheritanchor[from=circle]{north east}
  \inheritanchor[from=circle]{center}
  \inheritanchor[from=circle]{west}
  \inheritanchor[from=circle]{east}
  \inheritanchor[from=circle]{mid}
  \inheritanchor[from=circle]{mid west}
  \inheritanchor[from=circle]{mid east}
  \inheritanchor[from=circle]{base}
  \inheritanchor[from=circle]{base west}
  \inheritanchor[from=circle]{base east}
  \inheritanchor[from=circle]{south}
  \inheritanchor[from=circle]{south west}
  \inheritanchor[from=circle]{south east}
  \inheritbackgroundpath[from=circle]
  \foregroundpath{
    \centerpoint%
    \pgf@xc=\pgf@x%
    \pgf@yc=\pgf@y%
    \pgfutil@tempdima=\radius%
    \pgfmathsetlength{\pgf@xb}{\pgfkeysvalueof{/pgf/outer xsep}}%
    \pgfmathsetlength{\pgf@yb}{\pgfkeysvalueof{/pgf/outer ysep}}%
    \ifdim\pgf@xb<\pgf@yb%
      \advance\pgfutil@tempdima by-\pgf@yb%
    \else%
      \advance\pgfutil@tempdima by-\pgf@xb%
    \fi%
    \pgfpathmoveto{\pgfpointadd{\pgfqpoint{\pgf@xc}{\pgf@yc}}{\pgfqpoint{-0.707107\pgfutil@tempdima}{-0.707107\pgfutil@tempdima}}}
    \pgfpathlineto{\pgfpointadd{\pgfqpoint{\pgf@xc}{\pgf@yc}}{\pgfqpoint{0.707107\pgfutil@tempdima}{0.707107\pgfutil@tempdima}}}
    \pgfpathmoveto{\pgfpointadd{\pgfqpoint{\pgf@xc}{\pgf@yc}}{\pgfqpoint{-0.707107\pgfutil@tempdima}{0.707107\pgfutil@tempdima}}}
    \pgfpathlineto{\pgfpointadd{\pgfqpoint{\pgf@xc}{\pgf@yc}}{\pgfqpoint{0.707107\pgfutil@tempdima}{-0.707107\pgfutil@tempdima}}}
  }
}
\makeatother

\def\symbol#1{\textcolor{symbols}{#1}}

\def\decorate#1#2{
        \ifnum#2>0
    		\foreach \count in {1,...,#2}{
	       	let
				\p1 = (sourcenode.center),
                \p2 = (sourcenode.east),
				\n1 = {\x2-\x1},
				\n2 = {1mm},
				\n3 = {(1.3+0.6*(\count-1))*\n1},
				\n4 = {0.7*\n1}
			in 
        		node[rectangle,fill=symbols,rotate=30,inner sep=0pt,minimum width=0.2*\n2,minimum height=\n2] at ($(sourcenode.center) + (\n3,\n4)$) {}
				}
		\fi
        \ifnum#1>0
    		\foreach \count in {1,...,#1}{
	       	let
				\p1 = (sourcenode.center),
                \p2 = (sourcenode.east),
				\n1 = {\x2-\x1},
				\n2 = {1mm},
				\n3 = {(1.3+0.6*(\count-1))*\n1},
				\n4 = {0.7*\n1}
			in 
        		node[rectangle,fill=symbols,rotate=-30,inner sep=0pt,minimum width=0.2*\n2,minimum height=\n2] at ($(sourcenode.center) + (-\n3,\n4)$) {}
				}
		\fi
}

\tikzset{
    dectriangle/.style 2 args={
        triangle,
        alias=sourcenode,
        append after command={\decorate{#1}{#2}}
    },
    dectriangle/.default={0}{0},
}

\tikzset{
	cross/.style={path picture={ 
  		\draw[symbols]
			(path picture bounding box.south east) -- (path picture bounding box.north west) (path picture bounding box.south west) -- (path picture bounding box.north east);
		}},
root/.style={circle,fill=green!50!black,inner sep=0pt, minimum size=1.2mm},
        dot/.style={circle,fill=pageforeground,inner sep=0pt, minimum size=1mm},
        dotred/.style={circle,fill=pageforeground!50!pagebackground,inner sep=0pt, minimum size=2mm},
        var/.style={circle,fill=pageforeground!10!pagebackground,draw=pageforeground,inner sep=0pt, minimum size=3mm},
        kernel/.style={semithick,shorten >=2pt,shorten <=2pt},
        kernels/.style={snake=zigzag,shorten >=2pt,shorten <=2pt,segment amplitude=1pt,segment length=4pt,line before snake=2pt,line after snake=5pt,},
        rho/.style={densely dashed,semithick,shorten >=2pt,shorten <=2pt},
           testfcn/.style={dotted,semithick,shorten >=2pt,shorten <=2pt},
        renorm/.style={shape=circle,fill=pagebackground,inner sep=1pt},
        labl/.style={shape=rectangle,fill=pagebackground,inner sep=1pt},
        xic/.style={very thin,circle,draw=symbols,fill=symbols,inner sep=0pt,minimum size=1.2mm},
        g/.style={very thin,rectangle,draw=symbols,fill=symbols!10!pagebackground,inner sep=0pt,minimum width=2.5mm,minimum height=1.2mm},
        xi/.style={very thin,circle,draw=symbols,fill=symbols!10!pagebackground,inner sep=0pt,minimum size=1.2mm},
	xies/.style={very thin,rectangle,fill=green!50!black!25,draw=symbols,inner sep=0pt,minimum size=1.1mm},
	xiesf/.style={very thin,rectangle,fill=green!50!black,draw=symbols,inner sep=0pt,minimum size=1.1mm},
        xix/.style={very thin,crosscircle,fill=symbols!10!pagebackground,draw=symbols,inner sep=0pt,minimum size=1.2mm},
        X/.style={very thin,cross,rectangle,fill=pagebackground,draw=symbols,inner sep=0pt,minimum size=1.2mm},
	xib/.style={thin,circle,fill=symbols!10!pagebackground,draw=symbols,inner sep=0pt,minimum size=1.6mm},
	xie/.style={thin,circle,fill=green!50!black,draw=symbols,inner sep=0pt,minimum size=1.6mm},
	xid/.style={thin,circle,fill=symbols,draw=symbols,inner sep=0pt,minimum size=1.6mm},
	xibx/.style={thin,crosscircle,fill=symbols!10!pagebackground,draw=symbols,inner sep=0pt,minimum size=1.6mm},
	kernels2/.style={very thick,draw=connection,segment length=12pt},
	keps/.style={thin,draw=symbols,->},
	kepspr/.style={thick,draw=connection,->},
	krho/.style={thin,draw=symbols,superdense,->},
	krhopr/.style={thick,draw=connection,superdense},
	triangle/.style = { regular polygon, regular polygon sides=3},
	not/.style={thin,circle,draw=connection,fill=connection,inner sep=0pt,minimum size=0.5mm},
	diff/.style = {very thin,draw=symbols,triangle,fill=red!50!black,inner sep=0pt,minimum size=1.6mm},
	diff1/.style = {very thin,dectriangle={1}{0},fill=red!50!black,draw=symbols,inner sep=0pt,minimum size=1.6mm},
	diff2/.style = {very thin,dectriangle={1}{1},fill=red!50!black,draw=symbols,inner sep=0pt,minimum size=1.6mm},
		diffmini/.style = {very thin,rectangle,fill=black,draw=black,inner sep=0pt,minimum size=0.75mm},
	 kernelsmod/.style={very thick,draw=connection,segment length=12pt},
	 rec/.style = {very thin,rectangle,fill=black,draw=black,inner sep=0pt,minimum size=2mm},
	cerc/.style={very thin,circle,draw=black,fill=symbols,inner sep=0pt,minimum size=2mm},
	stars/.style={very thin,star,star points=6,star point ratio=0.5, draw=black,fill=red,inner sep=0pt,minimum size=0.7mm},
	>=stealth,
        }
        \tikzset{
root/.style={circle,fill=black!50,inner sep=0pt, minimum size=3mm},
        circ/.style={circle,fill=white,draw=black,very thin,inner sep=.5pt, minimum size=1.2mm},
        round1/.style={fill=white,outer sep = 0,inner sep=2pt,rounded corners=1mm,draw,text=black,thin,minimum size=1.2mm},
          circ1/.style={circle,fill=red!10,draw=red,very thin,inner sep=.5pt, minimum size=1.2mm},
        rect/.style={fill=white,outer sep = 0,inner sep=2pt,rectangle,draw,text=black,thin,minimum size=1.2mm},
        rect1/.style={fill=white,outer sep = 0,inner sep=2pt,rectangle,draw,text=black,thin,minimum size=1.2mm},
        round2/.style={fill=red!10,outer sep = 0,inner sep=2pt,rounded corners=1mm,draw,text=black,thin,minimum size=1.2mm},
       round3/.style={fill=blue!10,outer sep = 0,inner sep=2pt,rounded corners=1mm,draw,text=black,thin,minimum size=1.2mm}, 
        rect2/.style={fill=black!10,outer sep = 0,inner sep=2pt,rectangle,draw,text=black,thin,minimum size=1.2mm},
        dot/.style={circle,fill=black,inner sep=0pt, minimum size=1.2mm},
        dotred/.style={circle,fill=black!50,inner sep=0pt, minimum size=2mm},
        var/.style={circle,fill=black!10,draw=black,inner sep=0pt, minimum size=3mm},
        kernel/.style={semithick,shorten >=2pt,shorten <=2pt},
         diag/.style={thin,shorten >=4pt,shorten <=4pt},
        kernel1/.style={thick},
        kernels/.style={snake=zigzag,shorten >=2pt,shorten <=2pt,segment amplitude=1pt,segment length=4pt,line before snake=2pt,line after snake=5pt,},
		kernels1/.style={snake=zigzag,segment amplitude=0.5pt,segment length=2pt},
		rho1/.style={densely dotted,semithick},
        rho/.style={densely dashed,semithick,shorten >=2pt,shorten <=2pt},
           testfcn/.style={dotted,semithick,shorten >=2pt,shorten <=2pt},
           visible/.style={draw, circle, fill, inner sep=0.25ex},
        renorm/.style={shape=circle,fill=white,inner sep=1pt},
        labl/.style={shape=rectangle,fill=white,inner sep=1pt},
        xic/.style={very thin,circle,fill=symbols,draw=black,inner sep=0pt,minimum size=1.2mm},
        xi/.style={very thin,circle,fill=blue!10,draw=black,inner sep=0pt,minimum size=1.2mm},
	xib/.style={very thin,circle,fill=blue!10,draw=black,inner sep=0pt,minimum size=1.6mm},
	xie/.style={very thin,circle,fill=green!50!black,draw=black,inner sep=0pt,minimum size=1mm},
	xid/.style={very thin,circle,fill=symbols,draw=black,inner sep=0pt,minimum size=1.6mm},
	edgetype/.style={very thin,circle,draw=black,inner sep=0pt,minimum size=5mm},
	nodetype/.style={very thick,circle,draw=black,inner sep=0pt,minimum size=5mm},
	kernels2/.style={very thick,draw=connection,segment length=12pt},
clean/.style={thin,circle,fill=black,inner sep=0pt,minimum size=1mm},	not/.style={thin,circle,fill=symbols,draw=connection,fill=connection,inner sep=0pt,minimum size=0.8mm},
	>=stealth,
        }

\makeatletter
\def\DeclareSymbol#1#2#3{%
	\expandafter\gdef\csname MH@symb@#1\endcsname{\tikzsetnextfilename{symbol#1}%
	\tikz[baseline=#2,scale=0.15,draw=symbols,line join=round]{#3}}%
	\expandafter\gdef\csname MH@symb@#1s\endcsname{\scalebox{0.75}{\tikzsetnextfilename{symbol#1}%
	\tikz[baseline=#2,scale=0.15,draw=symbols,line join=round]{#3}}}%
	\expandafter\gdef\csname MH@symb@#1ss\endcsname{\scalebox{0.65}{\tikzsetnextfilename{symbol#1}%
	\tikz[baseline=#2,scale=0.15,draw=symbols,line join=round]{#3}}}%
	}
\def\<#1>{\ifthenelse{\boolean{mmode}}{\mathchoice{\csname MH@symb@#1\endcsname}{\csname MH@symb@#1\endcsname}{\csname MH@symb@#1s\endcsname}{\csname MH@symb@#1ss\endcsname}}{\csname MH@symb@#1\endcsname}}
\makeatother

\DeclareSymbol{Xi22}{0.5}{\draw (0,0) node[xi] {} -- (-1,1) node[not] {} -- (0,2) node[xi] {};} 
\DeclareSymbol{Xi2}{-2}{\draw (-1,-0.25) node[xi] {} -- (0,1) node[xi] {};} 
\DeclareSymbol{Xi2b}{-2}{\draw (-1,-0.25) node[xic] {} -- (0,1) node[xic] {};} 
\DeclareSymbol{Xi2g}{-2}{\draw (-1,-0.25) node[xies] {} -- (0,1) node[xi] {};} 
\DeclareSymbol{Xi2g2}{-2}{\draw (-1,-0.25) node[xi] {} -- (0,1) node[xies] {};} 
\DeclareSymbol{cXi2}{-2}{\draw (0,-0.25) node[xi] {} -- (-1,1) node[xic] {};}
\DeclareSymbol{Xi3}{0}{\draw (0,0) node[xi] {} -- (-1,1) node[xi] {} -- (0,2) node[xi] {};}
\DeclareSymbol{XiIIXi}{0}{\draw (0,0) node[xi] {} -- (-1,1); \draw[kernels2] (-1,1) node[not] {} -- (0,2) node[xi] {};}

\DeclareSymbol{Xi4}{2}{\draw (0,0) node[xi] {} -- (-1,1) node[xi] {} -- (0,2) node[xi] {} -- (-1,3) node[xi] {};}
\DeclareSymbol{Xi4_1}{2}{\draw (0,0) node[xic] {} -- (-1,1) node[xic] {} -- (0,2) node[xi] {} -- (-1,3) node[xi] {};}
\DeclareSymbol{Xi4_2}{2}{\draw (0,0) node[xic] {} -- (-1,1) node[xi] {} -- (0,2) node[xi] {} -- (-1,3) node[xic] {};}
\DeclareSymbol{Xi2X}{-2}{\draw (0,-0.25) node[xi] {} -- (-1,1) node[xix] {};}
\DeclareSymbol{XXi2}{-2}{\draw (0,-0.25) node[xix] {} -- (-1,1) node[xi] {};}
\DeclareSymbol{IIXi}{0}{\draw (0,-0.25) node[not] {} -- (-1,1) node[xi] {} -- (0,2) node[xi] {};}
\DeclareSymbol{IXi^2}{-1}{\draw (-1,1) node[xi] {} -- (0,0) node[not] {} -- (1,1) node[xi] {};}
\DeclareSymbol{IIXi^2}{-4}{\draw (0,-1.5) node[not] {} -- (0,0);
\draw[kernels2] (-1,1) node[xi] {} -- (0,0) node[not] {} -- (1,1) node[xi] {};}
\DeclareSymbol{XiX}{-2.8}{\node[xibx] {};}
\DeclareSymbol{tauX}{-2.8}{ \node[X] {};}
\DeclareSymbol{Xi}{-2.8}{\node[xib] {};}

\DeclareSymbol{IXiX}{-1}{\draw (0,-0.25) node[not] {} -- (-1,1) node[xix] {};}
\DeclareSymbol{IXi3}{2}{\draw (0,-0.25) node[not] {} -- (-1,1) node[xi] {} -- (0,2) node[xi] {} -- (-1,3) node[xi] {};}
\DeclareSymbol{IXi}{-2}{\draw (0,-0.25) node[not] {} -- (-1,1) node[xi] {};}
\DeclareSymbol{XiI}{-2}{\draw (0,-0.25) node[xi] {} -- (-1,1) node[not] {};}

\DeclareSymbol{Xi4b}{0}{\draw(0,1.5) node[xi] {} -- (0,0); \draw (-1,1) node[xi] {} -- (0,0) node[xi] {} -- (1,1) node[xi] {};}
\DeclareSymbol{Xi4b'}{0}{\draw(0,1.5) node[xi] {} -- (0,-0.2); \draw (-1,1) node[xi] {} -- (0,-0.2) node[not] {} -- (1,1) node[xi] {};}
\DeclareSymbol{Xi4c}{0}{\draw (0,1) -- (0.8,2.2) node[xi] {};\draw (0,-0.25) node[xi] {} -- (0,1) node[xi] {} -- (-0.8,2.2) node[xi] {};}
\DeclareSymbol{Xi4d}{-4.5}{\draw (0,-1.5) node[not] {} -- (0,0); \draw (-1,1) node[xi] {} -- (0,0) node[xi] {} -- (1,1) node[xi] {};}
\DeclareSymbol{Xi4e}{0}{\draw (0,2) node[xi] {} -- (-1,1) node[xi] {} -- (0,0) node[xi] {} -- (1,1) node[xi] {};}
\DeclareSymbol{Xi4e'}{0}{\draw (0,2) node[xi] {} -- (-1,1) node[xi] {} -- (0,-0.2) node[not] {} -- (1,1) node[xi] {};}

\DeclareSymbol{Xitwo}
{0}{\draw[kernels2] (0,0) node[not] {} -- (-1,1) node[not] {}
-- (-2,2) node[not]{} -- (-3,3) node[xi]  {};
\draw[kernels2] (0,0) -- (1,1) node[xi] {};
\draw[kernels2] (-1,1) -- (0,2) node[xi] {};
\draw[kernels2] (-2,2) -- (-1,3) node[xi] {};}

\DeclareSymbol{IXitwo}
{0}{\draw (-.7,1.2) node[xi] {} -- (0,-0.2) -- (.7,1.2) node[xi] {};}
\DeclareSymbol{I1Xitwo}
{0}{\draw[kernels2] (0,0) node[not] {} -- (-1,1) node[xi] {};
\draw[kernels2] (0,0) -- (1,1) node[xi] {};}

\DeclareSymbol{I1Xitwobis}
{0}{\draw[kernels2] (0,0) node[not] {} -- (-1,1) node[xies] {};
\draw[kernels2] (0,0) -- (1,1) node[xies] {};}

\DeclareSymbol{I1Xitwog}
{0}{\draw[kernels2] (0,0) node[not] {} -- (-1,1) node[xies] {};
\draw[kernels2] (0,0) -- (1,1) node[xi] {};}

\DeclareSymbol{cI1Xitwo}
{0}{\draw[kernels2] (0,0) node[not] {} -- (-1,1) node[xic] {};
\draw[kernels2] (0,0) -- (1,1) node[xi] {};}

\DeclareSymbol{I1IXi3}{0}{\draw (0,0) node[xi] {} -- (-1,1) ; 
\draw[kernels2] (-1,1) node[not] {} -- (0,2) node[xi] {};
\draw[kernels2] (-1,1) node[not] {} -- (-2,2) node[xi] {};}

\DeclareSymbol{I1Xi3c}{-1}{\draw[kernels2](0,1.5) node[xi] {} -- (0,0) node[not] {}; \draw (-1,1) node[xi] {} -- (0,0) ; \draw[kernels2] (0,0) -- (1,1) node[xi] {};}

\DeclareSymbol{I1Xi3cbis}{-1}{\draw[kernels2](0,1.5) node[xies] {} -- (0,0) node[not] {}; \draw (-1,1) node[xies] {} -- (0,0) ; \draw[kernels2] (0,0) -- (1,1) node[xies] {};}

\DeclareSymbol{I1IXi3b}{0}{\draw[kernels2] (0,0) node[not] {} -- (-1,1) ; \draw[kernels2] (0,0)   -- (1,1) node[xi] {} ;
\draw (-1,1) node[xi] {} -- (0,2) node[xi] {};
}

\DeclareSymbol{I1IXi3c}{0}{\draw[kernels2] (0,0) node[not] {} -- (-1,1) ; \draw[kernels2] (0,0)   -- (1,1) node[xi] {} ;
\draw[kernels2] (-1,1) node[not] {} -- (0,2) node[xi] {};
\draw[kernels2] (-1,1) node[not] {} -- (-2,2) node[xi] {};}

\DeclareSymbol{I1IXi3cbis}{0}{\draw[kernels2] (0,0) node[not] {} -- (-1,1) ; \draw[kernels2] (0,0)   -- (1,1) node[xies] {} ;
\draw[kernels2] (-1,1) node[not] {} -- (0,2) node[xies] {};
\draw[kernels2] (-1,1) node[not] {} -- (-2,2) node[xies] {};}

\DeclareSymbol{I1Xi}{0}{\draw[kernels2] (0,0) node[not] {} -- (-1,1)  node[xi] {} ;}

\DeclareSymbol{I1Xi4a}{2}{\draw[kernels2] (0,0) node[not] {} -- (-1,1) ; \draw[kernels2] (0,0) node[not] {} -- (1,1) node[xi] {} ;
\draw (-1,1) node[xi] {} -- (0,2) node[xi] {} -- (-1,3) node[xi] {};}

\DeclareSymbol{cI1Xi4a}{2}{\draw[kernels2] (0,0) node[not] {} -- (-1,1) ; \draw[kernels2] (0,0) node[not] {} -- (1,1) node[xic] {} ;
\draw (-1,1) node[xic] {} -- (0,2) node[xi] {} -- (-1,3) node[xi] {};}

\DeclareSymbol{I1Xi4b}{2}{\draw (0,0) node[xi] {} -- (-1,1) node[xi] {} -- (0,2) ; \draw[kernels2] (0,2) node[not] {} -- (-1,3) node[xi] {};\draw[kernels2] (0,2)  -- (1,3) node[xi] {};
}

\DeclareSymbol{cI1Xi4b}{2}{\draw (0,0) node[xic] {} -- (-1,1) node[xic] {} -- (0,2) ; \draw[kernels2] (0,2) node[not] {} -- (-1,3) node[xi] {};\draw[kernels2] (0,2)  -- (1,3) node[xi] {};
}

\DeclareSymbol{I1Xi4c}{2}{\draw (0,0) node[xi] {} -- (-1,1) node[not] {}; \draw[kernels2] (-1,1) -- (0,2) ; 
\draw[kernels2] (-1,1) -- (-2,2) node[xi] {} ;
\draw (0,2) node[xi] {} -- (-1,3) node[xi] {};}

\DeclareSymbol{cI1Xi4c}{2}{\draw (0,0) node[xic] {} -- (-1,1) node[not] {}; \draw[kernels2] (-1,1) -- (0,2) ; 
\draw[kernels2] (-1,1) -- (-2,2) node[xic] {} ;
\draw (0,2) node[xi] {} -- (-1,3) node[xi] {};}

\DeclareSymbol{I1Xi4ab}{2}{\draw[kernels2] (0,0) node[not] {} -- (-1,1) ; \draw[kernels2] (0,0) node[not] {} -- (1,1) node[xi] {};\draw (-1,1) node[xi] {} -- (0,2) ; \draw[kernels2] (0,2) node[not] {} -- (-1,3) node[xi] {};\draw[kernels2] (0,2)  -- (1,3) node[xi] {}; }

\DeclareSymbol{cI1Xi4ab}{2}{\draw[kernels2] (0,0) node[not] {} -- (-1,1) ; \draw[kernels2] (0,0) node[not] {} -- (1,1) node[xic] {};\draw (-1,1) node[xic] {} -- (0,2) ; \draw[kernels2] (0,2) node[not] {} -- (-1,3) node[xi] {};\draw[kernels2] (0,2)  -- (1,3) node[xi] {}; }

\DeclareSymbol{I1Xi4bc}{2}{\draw (0,0) node[xi] {} -- (-1,1) node[not] {}; \draw[kernels2] (-1,1) -- (0,2) ; 
\draw[kernels2] (-1,1) -- (-2,2) node[xi] {} ; \draw[kernels2] (0,2) node[not] {} -- (-1,3) node[xi] {};\draw[kernels2] (0,2)  -- (1,3) node[xi] {};
}

\DeclareSymbol{cI1Xi4bc}{2}{\draw (0,0) node[xic] {} -- (-1,1) node[not] {}; \draw[kernels2] (-1,1) -- (0,2) ; 
\draw[kernels2] (-1,1) -- (-2,2) node[xic] {} ; \draw[kernels2] (0,2) node[not] {} -- (-1,3) node[xi] {};\draw[kernels2] (0,2)  -- (1,3) node[xi] {};
}

\DeclareSymbol{I1Xi4abcc1}{2}{\draw[kernels2] (0,0) node[not] {} -- (-1,1) node[not] {}
-- (-2,2) node[not]{} -- (-3,3) node[xic]  {};
\draw[kernels2] (0,0) -- (1,1) node[xic] {};
\draw[kernels2] (-1,1) -- (0,2) node[xi] {};
\draw[kernels2] (-2,2) -- (-1,3) node[xi] {};
}

\DeclareSymbol{I1Xi4abcc1b}{2}{\draw[kernels2] (0,0) node[not] {} -- (-1,1) node[not] {}
-- (-2,2) node[not]{} -- (-3,3) node[xi]  {};
\draw[kernels2] (0,0) -- (1,1) node[xic] {};
\draw[kernels2] (-1,1) -- (0,2) node[xic] {};
\draw[kernels2] (-2,2) -- (-1,3) node[xi] {};
}

\DeclareSymbol{I1Xi4abcc2}{2}{\draw[kernels2] (0,0) node[not] {} -- (-1,1) node[not] {}
-- (-2,2) node[not]{} -- (-3,3) node[xic]  {};
\draw[kernels2] (0,0) -- (1,1) node[xi] {};
\draw[kernels2] (-1,1) -- (0,2) node[xi] {};
\draw[kernels2] (-2,2) -- (-1,3) node[xic] {};
}

\DeclareSymbol{I1Xi4ac}{2}{\draw[kernels2] (0,0) node[not] {} -- (-1,1) ; \draw[kernels2] (0,0) node[not] {} -- (1,1) node[xi] {}; 
\draw[kernels2] (-1,1) node[not] {} -- (0,2) ; 
\draw[kernels2] (-1,1) -- (-2,2) node[xi] {} ;
\draw (0,2) node[xi] {} -- (-1,3) node[xi] {};}

\DeclareSymbol{cI1Xi4ac}{2}{\draw[kernels2] (0,0) node[not] {} -- (-1,1) ; \draw[kernels2] (0,0) node[not] {} -- (1,1) node[xic] {}; 
\draw[kernels2] (-1,1) node[not] {} -- (0,2) ; 
\draw[kernels2] (-1,1) -- (-2,2) node[xic] {} ;
\draw (0,2) node[xi] {} -- (-1,3) node[xi] {};}

\DeclareSymbol{I1Xi4acc1}{2}{\draw[kernels2] (0,0) node[not] {} -- (-1,1) ; \draw[kernels2] (0,0) node[not] {} -- (1,1) node[xic] {}; 
\draw[kernels2] (-1,1) node[not] {} -- (0,2) ; 
\draw[kernels2] (-1,1) -- (-2,2) node[xi] {} ;
\draw (0,2) node[xic] {} -- (-1,3) node[xi] {};}

\DeclareSymbol{I1Xi4acc2}{2}{\draw[kernels2] (0,0) node[not] {} -- (-1,1) ; \draw[kernels2] (0,0) node[not] {} -- (1,1) node[xic] {}; 
\draw[kernels2] (-1,1) node[not] {} -- (0,2) ; 
\draw[kernels2] (-1,1) -- (-2,2) node[xi] {} ;
\draw (0,2) node[xi] {} -- (-1,3) node[xic] {};}

\DeclareSymbol{2I1Xi4}{2}{\draw[kernels2] (0,0) node[not] {} -- (-1,1) node[not] {};
\draw[kernels2] (0,0) -- (1,1) node[not] {};
\draw[kernels2] (-1,1) -- (-1.5,2.5) node[xi] {};
\draw[kernels2] (-1,1) -- (-0.5,2.5) node[xi] {};
\draw[kernels2] (1,1) -- (0.5,2.5) node[xi] {};
\draw[kernels2] (1,1) -- (1.5,2.5) node[xi] {};
}

\DeclareSymbol{2I1Xi4dis}{2}{\draw[kernels2] (0,0) node[not] {} -- (-1,1) node[not] {};
\draw[kernels2] (0,0) -- (1,1) node[not] {};
\draw[kernels2] (-1,1) -- (-1.5,2.5) node[xies] {};
\draw[kernels2] (-1,1) -- (-0.5,2.5) node[xies] {};
\draw[kernels2] (1,1) -- (0.5,2.5) node[xies] {};
\draw[kernels2] (1,1) -- (1.5,2.5) node[xies] {};
}

\DeclareSymbol{2I1Xi4c1}{2}{\draw[kernels2] (0,0) node[not] {} -- (-1,1) node[not] {};
\draw[kernels2] (0,0) -- (1,1) node[not] {};
\draw[kernels2] (-1,1) -- (-1.5,2.5) node[xic] {};
\draw[kernels2] (-1,1) -- (-0.5,2.5) node[xi] {};
\draw[kernels2] (1,1) -- (0.5,2.5) node[xic] {};
\draw[kernels2] (1,1) -- (1.5,2.5) node[xi] {};
}

\DeclareSymbol{2I1Xi4c2}{2}{\draw[kernels2] (0,0) node[not] {} -- (-1,1) node[not] {};
\draw[kernels2] (0,0) -- (1,1) node[not] {};
\draw[kernels2] (-1,1) -- (-1.5,2.5) node[xic] {};
\draw[kernels2] (-1,1) -- (-0.5,2.5) node[xic] {};
\draw[kernels2] (1,1) -- (0.5,2.5) node[xi] {};
\draw[kernels2] (1,1) -- (1.5,2.5) node[xi] {};
}

\DeclareSymbol{2I1Xi4b}{2}{\draw[kernels2] (0,0) node[not] {} -- (-1,1) ;
\draw[kernels2] (0,0) -- (1,1);
\draw (-1,1) node[xi] {} -- (-1,2.5) node[xi] {};
\draw (1,1)  node[xi] {} -- (1,2.5) node[xi] {};
}

\DeclareSymbol{2I1Xi4bb}{2}{\draw[kernels2] (0,0) node[not] {} -- (-1,1) ;
\draw[kernels2] (0,0) -- (1,1);
\draw (-1,1) node[xi] {} -- (-1,2.5) node[xiesf] {};
\draw (1,1)  node[xi] {} -- (1,2.5) node[xic] {};
}

\DeclareSymbol{2I1Xi4c}{2}{\draw[kernels2] (0,0) node[not] {} -- (-1,1);
\draw[kernels2] (0,0) -- (1,1) node[not] {};
\draw (-1,1)  node[xi] {} -- (-1,2.5) node[xi] {};
\draw[kernels2] (1,1) -- (0.4,2.5) node[xi] {};
\draw[kernels2] (1,1) -- (1.6,2.5) node[xi] {};
}

\DeclareSymbol{2I1Xi4cc1}{2}{\draw[kernels2] (0,0) node[not] {} -- (-1,1);
\draw[kernels2] (0,0) -- (1,1) node[not] {};
\draw (-1,1)  node[xic] {} -- (-1,2.5) node[xi] {};
\draw[kernels2] (1,1) -- (0.4,2.5) node[xic] {};
\draw[kernels2] (1,1) -- (1.6,2.5) node[xi] {};
}

\DeclareSymbol{2I1Xi4cc2}{2}{\draw[kernels2] (0,0) node[not] {} -- (-1,1);
\draw[kernels2] (0,0) -- (1,1) node[not] {};
\draw (-1,1)  node[xic] {} -- (-1,2.5) node[xic] {};
\draw[kernels2] (1,1) -- (0.4,2.5) node[xi] {};
\draw[kernels2] (1,1) -- (1.6,2.5) node[xi] {};
}

\DeclareSymbol{Xi4ba}{0}{\draw(-0.5,1.5) node[xi] {} -- (0,0); \draw (-1.5,1) node[xi] {} -- (0,0) node[not] {}; \draw[kernels2] (0,0) -- (1.5,1) node[xi] {};
\draw[kernels2] (0,0) -- (0.5,1.5) node[xi] {} ;}

\DeclareSymbol{Xi4badis}{0}{\draw(-0.5,1.5) node[xies] {} -- (0,0); \draw (-1.5,1) node[xies] {} -- (0,0) node[not] {}; \draw[kernels2] (0,0) -- (1.5,1) node[xies] {};
\draw[kernels2] (0,0) -- (0.5,1.5) node[xies] {} ;}

\DeclareSymbol{Xi4ba1}{0}{\draw(-0.5,1.5) node[xi] {} -- (0,0); \draw (-1.5,1) node[xi] {} -- (0,0) node[not] {}; \draw[kernels2] (0,0) -- (1.5,1) node[xic] {};
\draw[kernels2] (0,0) -- (0.5,1.5) node[xic] {} ;}

\DeclareSymbol{Xi4ba1b}{0}{\draw(-0.5,1.5) node[xic] {} -- (0,0); \draw (-1.5,1) node[xic] {} -- (0,0) node[not] {}; \draw[kernels2] (0,0) -- (1.5,1) node[xi] {};
\draw[kernels2] (0,0) -- (0.5,1.5) node[xi] {} ;}

\DeclareSymbol{Xi4ba1bdiff}{0}{\draw(-0.5,1.5) node[xic] {} -- (0,0); \draw (-1.5,1) node[xic] {} -- (0,0) node[not] {}; \draw (0,0) -- (1.5,1) node[xi] {};
\draw (0,0) -- (0.5,1.5) node[xi] {};
\draw(0,0) node[diff] {};}

\DeclareSymbol{Xi4ba1bb}{0}{\draw(-0.5,1.5) node[xic] {} -- (0,0); \draw (-1.5,1) node[xiesf] {} -- (0,0) node[not] {}; \draw[kernels2] (0,0) -- (1.5,1) node[xi] {};
\draw[kernels2] (0,0) -- (0.5,1.5) node[xi] {} ;}

\DeclareSymbol{Xi4ba2}{0}{\draw(-0.5,1.5) node[xi] {} -- (0,0); \draw (-1.5,1) node[xic] {} -- (0,0) node[not] {}; \draw[kernels2] (0,0) -- (1.5,1) node[xi] {};
\draw[kernels2] (0,0) -- (0.5,1.5) node[xic] {} ;}

\DeclareSymbol{Xi4ba2b}{0}{\draw(-0.5,1.5) node[xi] {} -- (0,0); \draw (-1.5,1) node[xic] {} -- (0,0) node[not] {}; \draw[kernels2] (0,0) -- (1.5,1) node[xi] {};
\draw[kernels2] (0,0) -- (0.5,1.5) node[xiesf] {} ;}

\DeclareSymbol{Xi4ca}{0}{\draw (0,1) -- (-1,2.2) node[xi] {};\draw (0,-0.25) node[xi] {} -- (0,1) ; \draw[kernels2] (0,1) node[not] {} -- (1,2.2) node[xi] {};
\draw[kernels2] (0,1) {} -- (0,2.7) node[xi] {};
}

\DeclareSymbol{Xi4cb}{0}{\draw (-1,1) -- (-2,2) node[xi] {};\draw[kernels2] (0,0)  -- (-1,1) node[xi] {} ; \draw[kernels2] (0,0) node[not] {} -- (1,1) node[xi] {} ; 
\draw (-1,1) node[xi] {} -- (0,2) node[xi] {};}

\DeclareSymbol{Xi4cbb}{0}{\draw (-1,1) -- (-2,2) node[xiesf] {};\draw[kernels2] (0,0)  -- (-1,1) node[xi] {} ; \draw[kernels2] (0,0) node[not] {} -- (1,1) node[xi] {} ; 
\draw (-1,1) node[xi] {} -- (0,2) node[xic] {};}

\DeclareSymbol{Xi4cbc1}{0}{\draw (-1,1) -- (-2,2) node[xic] {};\draw[kernels2] (0,0)  -- (-1,1) node[xic] {} ; \draw[kernels2] (0,0) node[not] {} -- (1,1) node[xi] {} ; 
\draw (-1,1) node[xic] {} -- (0,2) node[xi] {};}

\DeclareSymbol{Xi4cbc2}{0}{\draw (-1,1) -- (-2,2) node[xi] {};\draw[kernels2] (0,0)  -- (-1,1) node[xi] {} ; \draw[kernels2] (0,0) node[not] {} -- (1,1) node[xic] {} ; 
\draw (-1,1) node[xic] {} -- (0,2) node[xi] {};}

\DeclareSymbol{Xi4cab}{0}{\draw (-1,1) -- (-2,2) node[xi] {};\draw[kernels2] (0,0)  -- (-1,1); \draw[kernels2] (0,0) node[not] {} -- (1,1) node[xi] {} ; 
\draw[kernels2] (-1,1)  {} -- (0,2) node[xi] {};
\draw[kernels2] (-1,1) node[not] {} -- (-1,2.5) node[xi] {};
}

\DeclareSymbol{Xi4cabdis}{0}{\draw (-1,1) -- (-2,2) node[xies] {};\draw[kernels2] (0,0)  -- (-1,1); \draw[kernels2] (0,0) node[not] {} -- (1,1) node[xies] {} ; 
\draw[kernels2] (-1,1)  {} -- (0,2) node[xies] {};
\draw[kernels2] (-1,1) node[not] {} -- (-1,2.5) node[xies] {};
}

\DeclareSymbol{Xi4cabc1}{0}{\draw (-1,1) -- (-2,2) node[xi] {};\draw[kernels2] (0,0)  -- (-1,1); \draw[kernels2] (0,0) node[not] {} -- (1,1) node[xic] {} ; 
\draw[kernels2] (-1,1)  {} -- (0,2) node[xic] {};
\draw[kernels2] (-1,1) node[not] {} -- (-1,2.5) node[xi] {};
}

\DeclareSymbol{Xi4cabc2}{0}{\draw (-1,1) -- (-2,2) node[xic] {};\draw[kernels2] (0,0)  -- (-1,1); \draw[kernels2] (0,0) node[not] {} -- (1,1) node[xic] {} ; 
\draw[kernels2] (-1,1)  {} -- (0,2) node[xi] {};
\draw[kernels2] (-1,1) node[not] {} -- (-1,2.5) node[xi] {};
}

\DeclareSymbol{Xi4ea}{1.5}{\draw (-1,2.5) node[xi] {} -- (-1,1) node[xi] {} -- (0,0); 
 \draw[kernels2] (0,0)  -- (1,1) node[xi] {};
\draw[kernels2] (0,0) node[not] {} -- (0,1.5) node[xi] {}; }

\DeclareSymbol{Xi4eac1}{1.5}{\draw (-1,2.5) node[xic] {} -- (-1,1) node[xi] {} -- (0,0); 
 \draw[kernels2] (0,0)  -- (1,1) node[xic] {};
\draw[kernels2] (0,0) node[not] {} -- (0,1.5) node[xi] {}; }

\DeclareSymbol{Xi4eac1b}{1.5}{\draw (-1,2.5) node[xic] {} -- (-1,1) node[xi] {} -- (0,0); 
 \draw[kernels2] (0,0)  -- (1,1) node[xiesf] {};
\draw[kernels2] (0,0) node[not] {} -- (0,1.5) node[xi] {}; }

\DeclareSymbol{Xi4eac2}{1.5}{\draw (-1,2.5) node[xic] {} -- (-1,1) node[xic] {} -- (0,0); 
 \draw[kernels2] (0,0)  -- (1,1) node[xi] {};
\draw[kernels2] (0,0) node[not] {} -- (0,1.5) node[xi] {}; }

\DeclareSymbol{Xi4eact1}{1.5}{\draw (-1,2.5) node[xic] {} -- (-1,1) node[xi] {} -- (0,0); 
 \draw (0,0)  -- (1,1) node[xic] {};
\draw[rho] (0,0) node[not] {} -- (0,1.5) node[xi] {}; }

\DeclareSymbol{Xi4eact2}{1.5}{\draw[rho] (-1,2.5) node[xic] {} -- (-1,1) node[xi] {} -- (0,0); 
 \draw (0,0)  -- (1,1) node[xic] {};
\draw (0,0) node[not] {} -- (0,1.5) node[xi] {}; }

\DeclareSymbol{Xi4eabis}{1.5}{\draw (-1,2.5) node[xi] {} -- (-1,1) ; \draw[kernels2] (-1,1) node[xi] {} -- (0,0); 
 \draw (0,0)  -- (1,1) node[xi] {};
\draw[kernels2] (0,0) node[not] {} -- (0,1.5) node[xi] {}; }

\DeclareSymbol{Xi4eabisc1}{1.5}{\draw (-1,2.5) node[xic] {} -- (-1,1) ; \draw[kernels2] (-1,1) node[xi] {} -- (0,0); 
 \draw (0,0)  -- (1,1) node[xi] {};
\draw[kernels2] (0,0) node[not] {} -- (0,1.5) node[xic] {}; }

\DeclareSymbol{Xi4eabisc1b}{1.5}{\draw (-1,2.5) node[xic] {} -- (-1,1) ; \draw[kernels2] (-1,1) node[xi] {} -- (0,0); 
 \draw (0,0)  -- (1,1) node[xi] {};
\draw[kernels2] (0,0) node[not] {} -- (0,1.5) node[xiesf] {}; }

\DeclareSymbol{Xi4eabisc1bis}{1.5}{\draw (-1,2.5) node[xi] {} -- (-1,1) ; \draw[kernels2] (-1,1) node[xi] {} -- (0,0); 
 \draw (0,0)  -- (1,1) node[xi] {};
\draw[kernels2] (0,0) node[not] {} -- (0,1.5) node[xi] {};
\draw (-2,1) node[] {\tiny{$i$}};
\draw (-2,2.5) node[] {\tiny{$\ell$}};
\draw (2,1) node[] {\tiny{$k$}};
\draw (0,2.5) node[] {\tiny{$j$}};
 }

\DeclareSymbol{Xi4eabisc1tris}{1.5}{\draw (-1,2.5) node[xi] {} -- (-1,1) ; \draw[kernels2] (-1,1) node[xi] {} -- (0,0); 
 \draw (0,0)  -- (1,1) node[xi] {};
\draw[kernels2] (0,0) node[not] {} -- (0,1.5) node[xi] {};
\draw (-2,1) node[] {\tiny{i}};
\draw (-2,2.5) node[] {\tiny{j}};
\draw (2,1) node[] {\tiny{j}};
\draw (0,2.5) node[] {\tiny{i}};
 }

\DeclareSymbol{Xi4eabisc1quater}{1.5}{\draw (-1,2.5) node[xic] {} -- (-1,1) ; \draw[kernels2] (-1,1) node[xi] {} -- (0,0); 
 \draw (0,0)  -- (1,1) node[xic] {};
\draw[kernels2] (0,0) node[not] {} -- (0,1.5) node[xi] {};
 }

\DeclareSymbol{Xi4eabisc2}{1.5}{\draw (-1,2.5) node[xic] {} -- (-1,1) ; \draw[kernels2] (-1,1) node[xi] {} -- (0,0); 
 \draw (0,0)  -- (1,1) node[xic] {};
\draw[kernels2] (0,0) node[not] {} -- (0,1.5) node[xi] {}; }

\DeclareSymbol{Xi4eabisc2l}{1.5}{\draw (-1,2.5) node[xiesf] {} -- (-1,1) ; \draw[kernels2] (-1,1) node[xi] {} -- (0,0); 
 \draw (0,0)  -- (1,1) node[xic] {};
\draw[kernels2] (0,0) node[not] {} -- (0,1.5) node[xi] {}; }

\DeclareSymbol{Xi4eabisc2r}{1.5}{\draw (-1,2.5) node[xic] {} -- (-1,1) ; \draw[kernels2] (-1,1) node[xi] {} -- (0,0); 
 \draw (0,0)  -- (1,1) node[xiesf] {};
\draw[kernels2] (0,0) node[not] {} -- (0,1.5) node[xi] {}; }

\DeclareSymbol{Xi4eabisc3}{1.5}{\draw (-1,2.5) node[xic] {} -- (-1,1) ; \draw[kernels2] (-1,1) node[xic] {} -- (0,0); 
 \draw (0,0)  -- (1,1) node[xi] {};
\draw[kernels2] (0,0) node[not] {} -- (0,1.5) node[xi] {}; }

\DeclareSymbol{Xi4eb}{0}{
\draw[kernels2] (0,2) node[xi] {} -- (-1,1) ; \draw[kernels2] (-2,2)  node[xi] {} -- (-1,1) ; \draw (-1,1)  node[not] {} -- (0,0); 
 \draw (0,0) node[xi] {}  -- (1,1) node[xi] {};
}

\DeclareSymbol{Xi4eab}{1.5}{\draw[kernels2] (-1,2.5) node[xi] {} -- (-1,1) ; \draw[kernels2] (-2,2)  node[xi] {} -- (-1,1) ; \draw (-1,1)  node[not] {} -- (0,0); 
 \draw[kernels2] (0,0)  -- (1,1) node[xi] {};
\draw[kernels2] (0,0) node[not] {} -- (0,1.5) node[xi] {}; 
}

\DeclareSymbol{Xi4eabdis}{1.5}{\draw[kernels2] (-1,2.5) node[xies] {} -- (-1,1) ; \draw[kernels2] (-2,2)  node[xies] {} -- (-1,1) ; \draw (-1,1)  node[not] {} -- (0,0); 
 \draw[kernels2] (0,0)  -- (1,1) node[xies] {};
\draw[kernels2] (0,0) node[not] {} -- (0,1.5) node[xies] {}; 
}

\DeclareSymbol{Xi4eabc1}{1.5}{\draw[kernels2] (-1,2.5) node[xic] {} -- (-1,1) ; \draw[kernels2] (-2,2)  node[xi] {} -- (-1,1) ; \draw (-1,1)  node[not] {} -- (0,0); 
 \draw[kernels2] (0,0)  -- (1,1) node[xic] {};
\draw[kernels2] (0,0) node[not] {} -- (0,1.5) node[xi] {}; 
}

\DeclareSymbol{Xi4eabc2}{1.5}{\draw[kernels2] (-1,2.5) node[xi] {} -- (-1,1) ; \draw[kernels2] (-2,2)  node[xi] {} -- (-1,1) ; \draw (-1,1)  node[not] {} -- (0,0); 
 \draw[kernels2] (0,0)  -- (1,1) node[xic] {};
\draw[kernels2] (0,0) node[not] {} -- (0,1.5) node[xic] {}; 
}

\DeclareSymbol{Xi4eabbis}{1.5}{\draw[kernels2] (-1,2.5) node[xi] {} -- (-1,1) ; \draw[kernels2] (-2,2)  node[xi] {} -- (-1,1) ; \draw[kernels2] (-1,1)  node[not] {} -- (0,0); 
 \draw (0,0)  -- (1,1) node[xi] {};
\draw[kernels2] (0,0) node[not] {} -- (0,1.5) node[xi] {}; 
}

\DeclareSymbol{Xi4eabbisc1}{1.5}{\draw[kernels2] (-1,2.5) node[xic] {} -- (-1,1) ; \draw[kernels2] (-2,2)  node[xi] {} -- (-1,1) ; \draw[kernels2] (-1,1)  node[not] {} -- (0,0); 
 \draw (0,0)  -- (1,1) node[xic] {};
\draw[kernels2] (0,0) node[not] {} -- (0,1.5) node[xi] {}; 
}

\DeclareSymbol{Xi4eabbisc1perm}{1.5}{\draw[kernels2] (-1,2.5) node[xi] {} -- (-1,1) ; \draw[kernels2] (-2,2)  node[xic] {} -- (-1,1) ; \draw[kernels2] (-1,1)  node[not] {} -- (0,0); 
 \draw (0,0)  -- (1,1) node[xic] {};
\draw[kernels2] (0,0) node[not] {} -- (0,1.5) node[xi] {}; 
}

\DeclareSymbol{Xi4eabbisc2}{1.5}{\draw[kernels2] (-1,2.5) node[xi] {} -- (-1,1) ; \draw[kernels2] (-2,2)  node[xi] {} -- (-1,1) ; \draw[kernels2] (-1,1)  node[not] {} -- (0,0); 
 \draw (0,0)  -- (1,1) node[xic] {};
\draw[kernels2] (0,0) node[not] {} -- (0,1.5) node[xic] {}; 
}

\DeclareSymbol{Xi2cbis}{0}{\draw[kernels2] (0,1) -- (0.8,2.2) node[xi] {};\draw[kernels2] (0,-0.25) node[not] {} -- (0,1); \draw[kernels2] (0,1) node[not] {} -- (-0.8,2.2) node[xi] {};}

\DeclareSymbol{Xi2cbis1}{0}{\draw (0,1) -- (-0.8,2.2) node[xi] {};\draw[kernels2] (0,-0.25) node[not] {} -- (0,1) node[xi] {}; }

\DeclareSymbol{Xi2Xbis}{-2}{\draw[kernels2] (0,-0.25)  -- (-1,1) ; \draw (-1,1) node[xix] {};
\draw[kernels2] (0,-0.25) node[not] {} -- (1,1) node[xi] {};}

\DeclareSymbol{XXi2bis}{-2}{\draw[kernels2] (0,-0.25) -- (-1,1) node[xi] {};
\draw[kernels2] (0,-0.25) node[X] {} -- (1,1) node[xi] {};}

\DeclareSymbol{I1XiIXi}{0}{\draw[kernels2] (0,-0.25) -- (1,1) node[xi] {};
\draw (0,-0.25) node[not] {} -- (-1,1) node[xi] {};}

\DeclareSymbol{I1XiIXib}{0}{\draw  (0,-0.25) node[xi] {} -- (0,1) node[not] {};
\draw[kernels2] (0,1) -- (0,2.25) ; \draw (0,2.25) node[xi]{}; }

\DeclareSymbol{I1XiIXic}{0}{
\draw[kernels2] (0,0) -- (1,1) node[xi] {} ; 
\draw[kernels2] (0,0) node[not] {}  -- (-1,1) node[not] {} -- (0,2) node[xi] {};
}

\DeclareSymbol{thin}{1.4}{\draw[pagebackground] (-0.3,0) -- (0.3,0); \draw  (0,0) -- (0,2);}
\DeclareSymbol{thin2}{1.4}{\draw[pagebackground] (-0.3,0) -- (0.3,0); \draw[tinydots]  (0,0) -- (0,2);}

\DeclareSymbol{thick}{1.4}{\draw[pagebackground] (-0.3,0) -- (0.3,0); \draw[kernels2]  (0,0) -- (0,2);}

\DeclareSymbol{thick2}{1.4}{\draw[pagebackground] (-0.3,0) -- (0.3,0); \draw[kernels2,tinydots]  (0,0) -- (0,2);}

\DeclareSymbol{Xi4ind}{2}{\draw (0,0) node[xi,label={[label distance=-0.2em]right: \scriptsize  $ i $}]  { } -- (-1,1) node[xi,label={[label distance=-0.2em]left: \scriptsize  $ j $}] {} -- (0,2) node[xi,label={[label distance=-0.2em]right: \scriptsize  $ k $}] {} -- (-1,3) node[xi,label={[label distance=-0.2em]left: \scriptsize  $ \ell $}] {};}

\DeclareSymbol{Xi4c1}{2}{\draw (0,0) node[xic] {} -- (-1,1) node[xi] {} -- (0,2) node[xic] {} -- (-1,3) node[xi] {};} 
\DeclareSymbol{IXi2ex}{0}{\draw (0,-0.25) node[xie] {} -- (-1,1) node[xi] {} ; \draw (0,-0.25)-- (1,1) node[xi] {};}
\DeclareSymbol{IXi2ex1}{0}{\draw (0,-0.25) node[xie] {} -- (-1,1) node[xi] {} -- (0,2) node[xi] {};}

\DeclareSymbol{Xi4b1}{0}{\draw(0,1.5) node[xic] {} -- (0,0); \draw (-1,1) node[xic] {} -- (0,0) node[xi] {} -- (1,1) node[xi] {};}

\DeclareSymbol{Xi4ec1}{0}{\draw (0,2) node[xi] {} -- (-1,1) node[xic] {} -- (0,0) node[xic] {} -- (1,1) node[xi] {};}
\DeclareSymbol{Xi4ec2}{0}{\draw (0,2) node[xic] {} -- (-1,1) node[xi] {} -- (0,0) node[xic] {} -- (1,1) node[xi] {};}
\DeclareSymbol{Xi4ec3}{0}{\draw (0,2) node[xic] {} -- (-1,1) node[xic] {} -- (0,0) node[xi] {} -- (1,1) node[xi] {};}

\DeclareSymbol{I1Xi4ac1}{2}{\draw[kernels2] (0,0) node[not] {} -- (-1,1) ; \draw[kernels2] (0,0) node[not] {} -- (1,1) node[xic] {} ;
\draw (-1,1) node[xi] {} -- (0,2) node[xic] {} -- (-1,3) node[xi] {};}

\DeclareSymbol{I1Xi4ac2}{2}{\draw[kernels2] (0,0) node[not] {} -- (-1,1) ; \draw[kernels2] (0,0) node[not] {} -- (1,1) node[xic] {} ;
\draw (-1,1) node[xi] {} -- (0,2) node[xi] {} -- (-1,3) node[xic] {};}

\DeclareSymbol{I1Xi4bp}{2}{\draw (0,0) node[not] {} -- (-1,1) node[xi] {} -- (0,2) ; \draw[kernels2] (0,2) node[not] {} -- (-1,3) node[xi] {};\draw[kernels2] (0,2)  -- (1,3) node[xi] {};
}

\DeclareSymbol{I1Xi4bc1}{2}{\draw (0,0) node[xic] {} -- (-1,1) node[xi] {} -- (0,2) ; \draw[kernels2] (0,2) node[not] {} -- (-1,3) node[xi] {};\draw[kernels2] (0,2)  -- (1,3) node[xic] {};
}

\DeclareSymbol{I1Xi4bc2}{2}{\draw (0,0) node[xic] {} -- (-1,1) node[xi] {} -- (0,2) ; \draw[kernels2] (0,2) node[not] {} -- (-1,3) node[xic] {};\draw[kernels2] (0,2)  -- (1,3) node[xi] {};
}

\DeclareSymbol{I1Xi4cp}{2}{\draw (0,0) node[not] {} -- (-1,1) node[not] {}; \draw[kernels2] (-1,1) -- (0,2) ; 
\draw[kernels2] (-1,1) -- (-2,2) node[xi] {} ;
\draw (0,2) node[xi] {} -- (-1,3) node[xi] {};}

\DeclareSymbol{I1Xi4cc1}{2}{\draw (0,0) node[xic] {} -- (-1,1) node[not] {}; \draw[kernels2] (-1,1) -- (0,2) ; 
\draw[kernels2] (-1,1) -- (-2,2) node[xi] {} ;
\draw (0,2) node[xic] {} -- (-1,3) node[xi] {};}

\DeclareSymbol{I1Xi4cc2}{2}{\draw (0,0) node[xic] {} -- (-1,1) node[not] {}; \draw[kernels2] (-1,1) -- (0,2) ; 
\draw[kernels2] (-1,1) -- (-2,2) node[xi] {} ;
\draw (0,2) node[xi] {} -- (-1,3) node[xic] {};}

\DeclareSymbol{I1Xi4abc1}{2}{\draw[kernels2] (0,0) node[not] {} -- (-1,1) ; \draw[kernels2] (0,0) node[not] {} -- (1,1) node[xic] {};\draw (-1,1) node[xi] {} -- (0,2) ; \draw[kernels2] (0,2) node[not] {} -- (-1,3) node[xic] {};\draw[kernels2] (0,2)  -- (1,3) node[xi] {}; }

\DeclareSymbol{I1Xi4abc2}{2}{\draw[kernels2] (0,0) node[not] {} -- (-1,1) ; \draw[kernels2] (0,0) node[not] {} -- (1,1) node[xic] {};\draw (-1,1) node[xi] {} -- (0,2) ; \draw[kernels2] (0,2) node[not] {} -- (-1,3) node[xi] {};\draw[kernels2] (0,2)  -- (1,3) node[xic] {}; }

\DeclareSymbol{R1}{0}{\draw (-1,1) node[xi] {} -- (0,0) node[not] {};
\draw[kernels2] (0,1.5) node[xic] {} -- (0,0) -- (1,1) node[xic] {};}
\DeclareSymbol{R2}{0}{\draw (-1,1) node[xic] {} -- (0,0) node[not] {};
\draw[kernels2] (0,1.5)  {} -- (0,0) -- (1,1)  {};
\draw (0,1.5) node[xi] {};
\draw (1,1) node[xic] {};
}
\DeclareSymbol{R3}{1}{\draw[kernels2] (-1,1.5)  {} -- (0,0) node[not] {} -- (1,1.5);
\draw (-1,1.5) node[xi] {};
\draw[kernels2] (0,3) {} -- (1,1.5) -- (2,3)  {};
\draw  (0,3) node[xic] {} ;
\draw (2,3) node[xic] {};}
\DeclareSymbol{R4}{1}{\draw[kernels2] (-1,1.5) node[xic] {} -- (0,0) node[not] {} -- (1,1.5);
\draw[kernels2] (0,3) {} -- (1,1.5) -- (2,3) node[xic] {};
\draw (0,3) node[xi] {};}

\DeclareSymbol{I1Xi4bcp}{2}{\draw (0,0) node[not] {} -- (-1,1) node[not] {}; \draw[kernels2] (-1,1) -- (0,2) ; 
\draw[kernels2] (-1,1) -- (-2,2) node[xi] {} ; \draw[kernels2] (0,2) node[not] {} -- (-1,3) node[xi] {};\draw[kernels2] (0,2)  -- (1,3) node[xi] {};
}

\DeclareSymbol{I1Xi4bcc1}{2}{\draw (0,0) node[xic] {} -- (-1,1) node[not] {}; \draw[kernels2] (-1,1) -- (0,2) ; 
\draw[kernels2] (-1,1) -- (-2,2) node[xi] {} ; \draw[kernels2] (0,2) node[not] {} -- (-1,3) node[xi] {};\draw[kernels2] (0,2)  -- (1,3) node[xic] {};
}

\DeclareSymbol{I1Xi4bcc2}{2}{\draw (0,0) node[xic] {} -- (-1,1) node[not] {}; \draw[kernels2] (-1,1) -- (0,2) ; 
\draw[kernels2] (-1,1) -- (-2,2) node[xi] {} ; \draw[kernels2] (0,2) node[not] {} -- (-1,3) node[xic] {};\draw[kernels2] (0,2)  -- (1,3) node[xi] {};
} 

\DeclareSymbol{2I1Xi4bc1}{2}{\draw[kernels2] (0,0) node[not] {} -- (-1,1) ;
\draw[kernels2] (0,0) -- (1,1);
\draw (-1,1) node[xic] {} -- (-1,2.5) node[xi] {};
\draw (1,1)  node[xic] {} -- (1,2.5) node[xi] {};
}

\DeclareSymbol{2I1Xi4bc2}{2}{\draw[kernels2] (0,0) node[not] {} -- (-1,1) ;
\draw[kernels2] (0,0) -- (1,1);
\draw (-1,1) node[xi] {} -- (-1,2.5) node[xic] {};
\draw (1,1)  node[xic] {} -- (1,2.5) node[xi] {};
}

\DeclareSymbol{diff2I1Xi4bc2}{2}{\draw (0,0) node[diff] {} -- (-1,1) ;
\draw (0,0) -- (1,1);
\draw (-1,1) node[xi] {} -- (-1,2.5) node[xic] {};
\draw (1,1)  node[xic] {} -- (1,2.5) node[xi] {};
}

\DeclareSymbol{2I1Xi4bc3}{2}{\draw[kernels2] (0,0) node[not] {} -- (-1,1) ;
\draw[kernels2] (0,0) -- (1,1);
\draw (-1,1) node[xic] {} -- (-1,2.5) node[xic] {};
\draw (1,1)  node[xi] {} -- (1,2.5) node[xi] {};
}

\DeclareSymbol{Xi41}{0}{\draw (0,1) -- (0.8,2.2) node[xic] {};\draw (0,-0.25) node[xi] {} -- (0,1) node[xi] {} -- (-0.8,2.2) node[xic] {};} 

\DeclareSymbol{Xi42}{0}{\draw (0,1) -- (0.8,2.2) node[xi] {};\draw (0,-0.25) node[xic] {} -- (0,1) node[xi] {} -- (-0.8,2.2) node[xic] {};}

\DeclareSymbol{Xi4ca1}{0}{\draw (0,1) -- (-1,2.2) node[xic] {};\draw (0,-0.25) node[xi] {} -- (0,1) ; \draw[kernels2] (0,1) node[not] {} -- (1,2.2) node[xic] {};
\draw[kernels2] (0,1) {} -- (0,2.7) node[xi] {};
}

\DeclareSymbol{Xi4ca2}{0}{\draw (0,1) -- (-1,2.2) node[xi] {};\draw (0,-0.25) node[xi] {} -- (0,1) ; \draw[kernels2] (0,1) node[not] {} -- (1,2.2) node[xic] {};
\draw[kernels2] (0,1) {} -- (0,2.7) node[xic] {};
}

\DeclareSymbol{Xi4cap}{0}{\draw (0,1) -- (-1,2.2) node[xi] {};\draw (0,-0.25) node[not] {} -- (0,1) ; \draw[kernels2] (0,1) node[not] {} -- (1,2.2) node[xi] {};
\draw[kernels2] (0,1) {} -- (0,2.7) node[xi] {};
}

\DeclareSymbol{Xi3a}{0}{
 \draw (-1,1)  node[xi] {} -- (0,0); 
 \draw (0,0) node[xi] {}  -- (1,1) node[xi] {};
}

\DeclareSymbol{Xi4ebc1}{0}{
\draw[kernels2] (0,2) node[xi] {} -- (-1,1) ; \draw[kernels2] (-2,2)  node[xic] {} -- (-1,1) ; \draw (-1,1)  node[not] {} -- (0,0); 
 \draw (0,0) node[xic] {}  -- (1,1) node[xi] {};
}

\DeclareSymbol{Xi4ebc2}{0}{
\draw[kernels2] (0,2) node[xi] {} -- (-1,1) ; \draw[kernels2] (-2,2)  node[xi] {} -- (-1,1) ; \draw (-1,1)  node[not] {} -- (0,0); 
 \draw (0,0) node[xic] {}  -- (1,1) node[xic] {};
}

\DeclareSymbol{Xi2cbispex}{0}{\draw[kernels2] (0,1) -- (0.8,2.2) node[xi] {};\draw (0,-0.25) node[xie] {} -- (0,1); \draw[kernels2] (0,1) node[not] {} -- (-0.8,2.2) node[xi] {};}

\DeclareSymbol{Xi2cbis1p}{0}{\draw (0,1) -- (-0.8,2.2) node[xi] {};\draw (0,-0.25) node[not] {} -- (0,1) node[xi] {}; }

\DeclareSymbol{Xi2Xp}{-2}{\draw (0,-0.25) node[not] {} -- (-1,1) node[xix] {};} 

\DeclareSymbol{I1XiIXib}{0}{\draw  (0,-0.25) node[xi] {} -- (0,1) node[not] {};
\draw[kernels2] (0,1) -- (0,2.25) ; \draw (0,2.25) node[xi]{}; }

\DeclareSymbol{IXi2b}{0}{\draw  (0,-0.25) node[xi] {} -- (0,1) node[not] {};
\draw (0,1) -- (0,2.25) ; \draw (0,2.25) node[xi]{}; }

\DeclareSymbol{IXi2bex}{0}{\draw  (0,-0.25) node[xi] {} -- (0,1) node[xie] {};
\draw (0,1) -- (0,2.25) ; \draw (0,2.25) node[xi]{}; }

 \def\1{\mathbf{\symbol{1}}}

\def\one{\mathbf{1}}

\DeclareSymbol{diff}{0}{
\draw (0,0.5) node[diff] {};
}

\DeclareSymbol{diff1}{0}{
\draw (0,0.5) node[diff1] {};
}

\DeclareSymbol{diff2}{0}{
\draw (0,0.5) node[diff2] {};
}

\DeclareSymbol{geo}{0}{
\draw (0,0) node[diff] {};
\draw (0.3,0) node[diff] {};
}

\DeclareSymbol{generic}{0}{
\draw (0,0.6) node[xi] {};
}

\DeclareSymbol{g}{0}{
\draw (0,0.6) node[g] {};
}

\DeclareSymbol{Ito}{0}{
\draw (0,0.6) node[xies] {};
}

\DeclareSymbol{Itob}{0}{
\draw (0,0.6) node[xiesf] {};
}

\DeclareSymbol{greycirc}{0}{
\draw (0,0.3) node[xi] {};
}

\DeclareSymbol{not}{0}{
\draw (0,0.6) node[not] {};
\draw[tinydots] (0,0.6) circle (0.8);
}

\DeclareSymbol{genericb}{0}{
\draw (0,0.6) node[xic] {};
}

\DeclareSymbol{bluecirc}{0}{
\draw (0,0.3) node[xic] {};
}

\DeclareSymbol{genericxix}{0}{
\draw (0,0.6) node[xix] {};
}

\DeclareSymbol{genericX}{0}{
\draw (0,0.6) node[X] {};
}

\DeclareSymbol{diffIto}{1}{
\draw  (0,2.5) -- (0,0) ;
\draw (0,-0.1) node[diff] {};
\draw (0,2.5) node[xies] {};
}
\DeclareSymbol{Itodiff}{2}{
\draw(0,2.9) -- (0,-0.2);
\draw (0,2.9) node[diff] {};
\draw (0,-0.1) node[xies] {};
}

\DeclareSymbol{diffgeneric}{1}{
\draw  (0,2.5) -- (0,0) ;
\draw (0,-0.1) node[diff] {};
\draw (0,2.5) node[xi] {};
}

\DeclareSymbol{genericdiff}{2}{
\draw(0,2.9) -- (0,-0.2);
\draw (0,2.9) node[diff] {};
\draw (0,-0.1) node[xi] {};
}

\DeclareSymbol{diffdot}{2}{
\draw  (0,3) -- (0,-0.1) ;
\draw (0,3) node[not] {};
\draw (0,-0.1) node[diff] {};
}

\DeclareSymbol{diffdotmini}{0}{
\draw  (0,0) -- (0,1.2) ;
\draw (0,1.2) node[not] {};
\draw (0,0) node[diffmini] {};
}

\DeclareSymbol{dotdiff}{2}{
\draw[kernelsmod]  (0,3) -- (0,-0.1) ;
\draw (0,3) node[diff] {};
\draw (0,-0.1) node[not] {};
}

\DeclareSymbol{dotdiff1}{2}{
\draw[kernelsmod]  (0,3) -- (0,-0.1) ;
\draw (0,3) node[diff1] {};
\draw (0,-0.1) node[not] {};
}

\DeclareSymbol{dotdiff1mini}{0}{
\draw[kernelsmod]  (0,1.2) -- (0,0) ;
\draw (0,1.2) node[diffmini] {};
\draw (0,0) node[not] {};
}

\DeclareSymbol{dotdiff2}{2}{
\draw (0,3) -- (0,-0.1) ;
\draw (0,3) node[diff] {};
\draw (0,-0.1) node[not] {};
}

\DeclareSymbol{dotdiff2mini}{0}{
\draw (0,1.2) -- (0,0) ;
\draw (0,1.2) node[diffmini] {};
\draw (0,0) node[not] {};
}

\DeclareSymbol{dotdiffstraight}{0}{
\draw  (0,3) -- (0,-0.1) ;
\draw (0,3) node[diff] {};
\draw (0,-0.1) node[not] {};
}

\DeclareSymbol{arbre1}{0}{
\draw  (0,0) -- (1.5,1.5) ;
\draw (1.5,1.5) node[not] {};
\draw (0,0) node[not] {};
}

\DeclareSymbol{arbre2}{0}{
\draw  (0,0) -- (1.5,1.5) ;
\draw[kernelsmod] (0,0) -- (-1.5,1.5);
\draw (1.5,1.5) node[not] {};
\draw (0,0) node[not] {};
\draw (-1.5,1.5) node[xi] {};
}

\DeclareSymbol{arbre3}{0}{
\draw  (0,0) -- (1.5,1.5) ;
\draw[kernelsmod] (1.5,1.5) -- (0,3);
\draw (0,0) node[not] {};
\draw (1.5,1.5) node[not] {};
\draw (0,3) node[xi] {};
}

\DeclareSymbol{treeeval}{0}{
\draw (0,0) -- (1,1);
\draw (0,0) node[xi] {};
\draw (1.25,1.25) node[xi] {};
\draw (-0.6,0.6) node[]{\tiny{$i$}};
\draw (0.65,1.85) node[]{\tiny{$j$}};
}

\DeclareSymbol{testeval}{0}{
\draw (0,0) -- (1,1);
\draw (0,0) -- (-1,1);
\draw (0,0) node[xi] {};
\draw (1.25,1.25) node[xi] {};
\draw (-1.25,1.25) node[xi] {};
\draw (-0.6,-0.6) node[]{\tiny{$i$}};
\draw (0.65,1.85) node[]{\tiny{$j$}};
\draw (-1.95,1.85) node[]{\tiny{$k$}};
}

\DeclareSymbol{treeeval2}{0}{
\draw[kernelsmod] (-0.25,-1) -- (1,0.5) ;
\draw[kernelsmod] (1,0.5) -- (-0.25,2);
\draw (1,0.5) node[diff2] {};
\draw (-0.25,-1) node[not] {};
\draw (-0.25,2) node[xi] {};
\draw (-0.6,1.2) node[]{\tiny{1}};
}

\DeclareSymbol{arbreact}{1}{
\draw (0,0) node[not] {};
\draw[kernelsmod] (0,0) -- (1,1);
\draw[kernelsmod] (0,0) -- (-1,1);
\draw (-1,1) node[xic] {};
\draw  (0,2) -- (1,1) ;
\draw (0,2) node[xic] {};
\draw (1,1) node[xi] {};
}

\DeclareSymbol{arbreact1}{0}{
\draw (0,-1.5) -- (0,0);
\draw[kernelsmod] (0,0) -- (1,1);
\draw[kernelsmod] (0,0) -- (-1,1);
\draw  (0,2) -- (1,1) ;
\draw (0,-1.5) node[diff] {};
\draw (0,0) node[not] {};
\draw (-1,1) node[xic] {};
\draw (0,2) node[xic] {};
\draw (1,1) node[xi] {};
}

\DeclareSymbol{arbreact2}{0}{
\draw (0,-0.75) -- (-1,0.5); 
\draw (0,-0.75) -- (1,0.5);
\draw (0,1.5) -- (1,0.5);
\draw (0,1.5) node[xic] {};
\draw (1,0.5) node[xi] {};
\draw (-1,0.5) node[xic] {};
\draw (0,-0.75) node[diff] {};
}

\DeclareSymbol{arbreact3}{0}{
\draw[kernelsmod] (0,-0.75) -- (-1,0.5); 
\draw[kernelsmod] (0,-0.75) -- (1,0.5);
\draw (0,1.75) -- (1,0.5);
\draw (2,1.75) -- (1,0.5);
\draw (0,1.75) node[xic] {};
\draw (1,0.5) node[diff] {};
\draw (-1,0.5) node[xic] {};
\draw (2,1.75) node[xi] {};
\draw (0,-0.75) node[not] {};
}

\DeclareSymbol{pre_im_I1Xitwo}{0}{
\draw[kernels2] (0,-0.3) node[not] {} -- (-0.6,0.7) ;
\draw[kernels2] (0,-0.3) -- (0.6,0.7);
\draw (0,0.9) node[g] {};
}

\DeclareSymbol{pre_im_cI1Xi4ab}{2}{
\draw[kernels2] (0,-1) node[not] {} -- (-0.6,0) ;
\draw[kernels2] (0,-1) -- (0.6,0);
\draw (0,0.2) node[g] {};
\draw (0,0.6) -- (0,1.5);
\draw[kernels2] (0,1.5) node[not] {} -- (-0.6,2.5) ;
\draw[kernels2] (0,1.5) -- (0.6,2.5);
\draw (0,2.7) node[g] {};
}

\DeclareSymbol{pre_im_I1Xi4acc2}{0}{
\draw[kernels2] (-1,-0.5) node[not] {} -- (-1.6,0.5) ;
\draw[kernels2] (-1,-0.5) -- (-0.4,0.5);
\draw[kernels2] (-1,-0.5) -- (0.2,-1.5) node[not] {} ;
\draw (-1,1.1) -- (-1,2);
\draw[kernels2] (0.2,-1.5) -- (0.2,2);
\draw (-1,0.7) node[g] {};
\draw (-0.3,2.2) node[g] {};
}

\DeclareSymbol{pre_im_I1Xi4abcc2}{2}{
\draw[kernels2] (0,-1) node[not] {} -- (-1,0) node[not] {};
\draw[kernels2] (-1,1.2) node[not] {} -- (-1,0);
\draw[kernels2] (-1,1.2) -- (-1.5,2.5);
\draw[kernels2] (-1,1.2) -- (-0.5,2.5);
\draw[kernels2] (-1,0) -- (0.7,2.5);
\draw[kernels2] (0,-1) -- (1.5,2.5);
\draw (-1,2.7) node[g] {};
\draw (1,2.7) node[g] {};
}

\DeclareSymbol{pre_im_2I1Xi4c1}{2}{
\draw[kernels2] (0,-0.5) node[not] {} -- (-1,0.5) node[not] {};
\draw[kernels2] (0,-0.5) -- (1,0.5) node[not] {};
\draw[kernels2] (-1,0.5) node[not] {}-- (-1.7,2);
\draw[kernels2]  (-1,2) -- (1,0.5);
\draw[kernels2] (-1,0.5) -- (1,2);
\draw[kernels2] (1,0.5) -- (1.7,2);
\draw (-1.2,2.2) node[g] {};
\draw (1.2,2.2) node[g] {};
}

\DeclareSymbol{pre_im_Xi4eabisc2}{2}{
\draw[kernels2] (1.2,-0.5) node[not] {} -- (-0.7,0.8) ;
\draw[kernels2] (1.2,-0.5) -- (0.4,0.8);
\draw (0,1.4)  -- (0,2.2);
\draw (1.2,2.2) -- (1.2,-0.6);
\draw (0,1) node[g] {};
\draw (0.6,2.4) node[g] {};
}

\DeclareSymbol{pre_im_Xi4eabisc22}{2}{
\draw (1.2,-0.5) node[not] {} -- (-0.7,0.8) ;
\draw[kernels2] (1.2,-0.5) -- (0.4,0.8);
\draw (0,1.4)  -- (0,2.2);
\draw[kernels2] (1.2,2.2) -- (1.2,-0.6);
\draw (0,1) node[g] {};
\draw (0.6,2.4) node[g] {};
}

\DeclareSymbol{pre_im_Xi4eabisc222}{2}{
\draw[kernels2] (0.4,-0.5) node[not] {} -- (-0.6,1) ;
\draw[kernels2] (1.2,0) -- (0.3,1);
\draw (0,1.1)  -- (0,2.5);
\draw[kernels2] (1.2,2.5) -- (1.2,0) node[not] {} -- (0.4,-0.6);
\draw (0,1.2) node[g] {};
\draw (0.6,2.5) node[g] {};
}

\DeclareSymbol{pre_im_Xi4eabc2}{2}{
\draw (0,-0.5) node[not] {} -- (-1,0.5) node[not] {};
\draw[kernels2] (-1,0.5) -- (-1.5,2);
\draw[kernels2] (0,-0.5)  -- (0.7,2);
\draw[kernels2] (-1,0.5) -- (-0.5,2);
\draw[kernels2] (0,-0.5) -- (1.5,2);
\draw (-1,2.2) node[g] {};
\draw (1,2.2) node[g] {};
}

\DeclareSymbol{pre_im_Xi4eabbisc2}{2}{
\draw[kernels2] (0,-0.5) node[not] {} -- (-1,0.5) node[not] {};
\draw[kernels2] (-1,0.5) -- (-1.5,2);
\draw[kernels2] (0,-0.5)  -- (0.7,2);
\draw[kernels2] (-1,0.5) -- (-0.5,2);
\draw (0,-0.5) -- (1.5,2);
\draw (-1,2.2) node[g] {};
\draw (1,2.2) node[g] {};
}

\DeclareSymbol{pre_im_I1Xi4abcc1}{2}{
\draw[kernels2] (0,-1) node[not] {} -- (-1,0) node[not] {};
\draw[kernels2] (0,1.1) node[not] {} -- (-1,0);
\draw[kernels2] (-1,0) -- (-1.5,2.5);
\draw[kernels2] (0,1.1) node[not] {} -- (-0.5,2.5);
\draw[kernels2] (0,1.1) -- (0.5,2.5);
\draw[kernels2] (0,-1) -- (1.5,2.5);
\draw (-1,2.7) node[g] {};
\draw (1,2.7) node[g] {};
}

\DeclareSymbol{pre_im_Xi4eabc1}{2}{
\draw (0,-0.5) node[not] {} -- (-1,0.5) node[not] {};
\draw[kernels2] (-1,0.5) -- (-1.5,2);
\draw[kernels2] (0,-0.5)  -- (-0.5,2);
\draw[kernels2] (-1,0.5) -- (0.8,2);
\draw[kernels2] (0,-0.5) -- (1.5,2);
\draw (-1,2.2) node[g] {};
\draw (1,2.2) node[g] {};
}

\DeclareSymbol{pre_im_Xi4ba1b}{2}{
\draw[kernels2] (0,0) node[not] {}  -- (1.8,1.5);
\draw[kernels2] (0,0) -- (0.8,1.5);
\draw (0,-0.1) -- (-1.8,1.5);
\draw (0,-0.1) -- (-0.8,1.5);
\draw (-1,1.7) node[g] {};
\draw (1,1.7) node[g] {};
}

\DeclareSymbol{pre_im_Xi4ba2}{2}{
\draw (0,-0.1) node[not] {}  -- (1.8,1.5);
\draw[kernels2] (0,0) -- (0.8,1.5);
\draw (0,-0.1) -- (-1.8,1.5);
\draw[kernels2] (0,0) -- (-0.8,1.5);
\draw (-1,1.7) node[g] {};
\draw (1,1.7) node[g] {};
}

\DeclareSymbol{pre_im_Xi4cabc2}{2}{
\draw[kernels2] (0,-0.5) node[not] {} -- (-1,0.5) node[not] {};
\draw[kernels2] (-1,0.5) -- (-1.5,2);
\draw (-1,0.5)  -- (0.7,2);
\draw[kernels2] (-1,0.5) -- (-0.5,2);
\draw[kernels2] (0,-0.5) -- (1.5,2);
\draw (-1,2.2) node[g] {};
\draw (1,2.2) node[g] {};
}

\DeclareSymbol{pre_im_Xi4cabc1}{2}{
\draw[kernels2] (0,-0.5) node[not] {} -- (-1,0.5) node[not] {};
\draw (-1,0.5) -- (-1.5,2);
\draw[kernels2] (-1,0.5)  -- (0.7,2);
\draw[kernels2] (-1,0.5) -- (-0.5,2);
\draw[kernels2] (0,-0.5) -- (1.5,2);
\draw (-1,2.2) node[g] {};
\draw (1,2.2) node[g] {};
}

\DeclareSymbol{pre_im_Xi4eabbisc1}{2}{
\draw[kernels2] (0,-0.5) node[not] {} -- (-1,0.5) node[not] {};
\draw[kernels2] (-1,0.5) -- (-1.5,2);
\draw[kernels2] (0,-0.5)  -- (-0.5,2);
\draw[kernels2] (-1,0.5) -- (0.8,2);
\draw (0,-0.5) -- (1.5,2);
\draw (-1,2.2) node[g] {};
\draw (1,2.2) node[g] {};
}

\DeclareSymbol{pre_im_1}{0}{
\draw[kernels2] (0,-0.5) node[not] {} -- (-0.6,0.5) ;
\draw[kernels2] (0,-0.5) -- (0.6,0.5);
\draw (0,1.1)  -- (-0.55,2);
\draw (0,1.1)  -- (0.55,2);
\draw (0,0.7) node[g] {};
\draw (0,2.2) node[g] {};
}

\DeclareSymbol{disconnect}{0}{
\draw[kernels2] (0,-0.5) node[not] {} -- (-0.6,0.5) ;
\draw[kernels2] (0,-0.5) -- (0.6,0.5);
\draw (-0.55,1.1)  -- (-0.55,2.3);
\draw (0.55,2.3) -- (0.55,1.5) -- (1.2,1.5) -- (1.2,3.5) -- (0.55,3.5) -- (0.55,2.7);
\draw (0,0.7) node[g] {};
\draw (0,2.5) node[g] {};
}

\DeclareSymbol{pre_im_2}{2}{\draw[kernels2] (0,0) node[not] {} -- (-1,1) node[not] {};
\draw[kernels2] (0,0) -- (1,1) node[not] {};
\draw[kernels2] (-1,1) -- (-1.5,2.5);
\draw[kernels2] (-1,1) -- (-0.5,2.5);
\draw[kernels2] (1,1) -- (0.5,2.5);
\draw[kernels2] (1,1) -- (1.5,2.5);
\draw (-1,2.7) node[g] {};
\draw (1,2.7) node[g] {};
}

\DeclareSymbol{CX_rec}{0}{
\draw [black] (-0.3,1) to (-0.3,-0.3);
\draw [black] (0.3,1) to (0.3,-0.3);
\draw [black] (-0.3,1) to (-0.3,2.3);
\draw [black] (0.3,1) to (0.3,2.3);
\draw (0,1) node[rec] {};
}

\DeclareSymbol{CX_cerc}{0}{
\draw [black] (0,1) to (0,-0.3);
\draw (0,1) node[cerc] {};
}

\DeclareMathAlphabet{\mathpzc}{OT1}{pzc}{m}{it}

\def\eqref#1{(\ref{#1})}

\makeatletter 
\newcommand*{\bigcdot}{}
\DeclareRobustCommand*{\bigcdot}{%
  \mathbin{\mathpalette\bigcdot@{}}%
}
\newcommand*{\bigcdot@scalefactor}{.5}
\newcommand*{\bigcdot@widthfactor}{1.15}
\newcommand*{\bigcdot@}[2]{%
  \sbox0{$#1\vcenter{}$}
  \sbox2{$#1\cdot\m@th$}%
  \hbox to \bigcdot@widthfactor\wd2{%
    \hfil
    \raise\ht0\hbox{%
      \scalebox{\bigcdot@scalefactor}{%
        \lower\ht0\hbox{$#1\bullet\m@th$}%
      }%
    }%
    \hfil
  }%
}
\makeatother

\tcbset
{colframe=boxcolor,colback=symbols!7!pagebackground,coltext=pageforeground,
fonttitle=\bfseries,nobeforeafter,center title,size=fbox,boxsep=1.5pt,
top=0mm,bottom=0mm,boxsep=0mm,tcbox raise base}

\def\two{{\<generic>\kern0.05em\<genericb>}}
\def\twoI{{\<Ito>\kern0.05em\<Itob>}}

\def\mail#1{\burlalt{#1}{mailto:#1}}

\usepackage{thmtools} 

\declaretheorem[style=definition]{example}

\begin{document}

\title{Derivation of resonance-based schemes via normal forms}
\author{Yvain Bruned}
\institute{ 
 Université de Lorraine, CNRS, IECL, F-54000 Nancy, France
  \\
Email:\ \begin{minipage}[t]{\linewidth}
\mail{yvain.bruned@univ-lorraine.fr}.
\end{minipage}}

\maketitle

\begin{abstract}
In this work, we propose a systematic derivation of resonance-based schemes via normal forms. The main idea is to use an arborification map on decorated trees together with a Butcher-Connes-Kreimer type coproduct and lower-dominant parts decompositions of the Fourier operator coming from the nonlinear interactions. This new family of low regularity schemes has explicit  formulae for its coefficients and its local error. Under a mild assumption, one could expect these schemes to have a similar local error as the low regularity schemes proposed in   \cite{BS}. 
\\[.4em]
\end{abstract}

\setcounter{tocdepth}{1}
\tableofcontents

\section{Introduction}

 Resonance-based schemes first proposed in \cite{HS17, OS18} on some dispersive equations, such as the cubic nonlinear Schrödinger equation (NLS) and the Korteweg-de Vries equation (KdV), were able to provide numerical schemes in time requiring less regularity on the solutions than classical integrators, splitting schemes \cite{HLW06,HO10} will need.
 The main idea is to embed the resonance into the discretisation together with a good implementation.
 Such an approach has been generalised with the help of a suitable set of decorated trees together with Hopf algebraic tools for understanding the local error of the scheme in \cite{BS}.
 This formalism draws its inspiration from decorated trees used in singular SPDEs via the theory of Regularity Structures \cite{reg,BHZ,BCCH} where new B-series describe the local expansion of the solutions. Part of the formalism is also coming from the decorated trees that appear for studying analytical properties of these equations, looking at iterated integrals obtained from the Duhamel's formulation (see \cite{C07,Gub06,Gub11,GKO13}). The Hopf algebraic part also relies on variants of the Butcher-Connes-Kreimer coproduct \cite{Butcher,CK1}.
 Let us mention that this general resonance approach has been extended beyond the case with periodic boundary conditions \cite{FS,ABBS22}, it has been applied for Wave turbulence discretising the second moment of the solution \cite{ABBS24} and to the stochastic context \cite{AGB23}. One can also produce schemes that are symmetric \cite{ABBMS23} and symplectic \cite{MS25,AGB24}.
 
 This resonance approach is reminiscent of the well-established normal forms approach for dispersive PDEs where the main normal forms are Birkhoff normal form \cite{Bam03,BG06,BG22} and the Poincaré-Dulac normal form \cite{GKO13,OST18}, which have many applications in the analytical study of dispersive equations. These two normal forms are based on the same idea, that is, eliminating non-resonant monomials when iterating the dynamics. One open question was to understand the connection between these approaches  and resonance-based schemes. The main aim of the present paper is to address it.
 
 In \cite{B24}, a first step has been made toward this direction by clarifying the algebraic tools at play in the derivation of the Poincaré-Dulac normal form in \cite{GKO13}. It turns out that one of the key maps for its derivation is the arborification map coming from mould calculus \cite{Ecalle1,Ecalle2,EV04,Cr09,FM} that allows us to rewrite an expansion in terms of words from rooted trees. 
 Using the arborification approach, combined with additional algebraic arguments based on a forest formula and the resonance decomposition proposed  by Katharina Schratz, one can derive low regularity schemes via this normal form. This is the main result of this paper together with explicit formulae for the local error. 
 	Under  Assumption \ref{Assumption_1}, one expects normal form schemes to have the same local error as schemes described in \cite{BS}.
 	This main result is made precise in Theorem\ref{main_theorem} and Corollary \ref{main_low_schemes} that derive the low regularity schemes with local error.   
 	Let us first explain some of the main ideas that lead to this theorem. One considers a general dispersive equation of the form
 	\begin{equation}\label{dis}
 		\begin{aligned}
 			& i \partial_t u(t,x) +   \mathcal{L}\left(\nabla\right) u(t,x) =\vert \nabla\vert^\alpha p\left(u(t,x), \overline u(t,x)\right), \\
 			& u(0,x) = u_0(x),
 		\end{aligned}
 	\end{equation}
 	equipped with periodic boundary conditions $x\in \mathbb{T}^d$.  We assume that $p$ is a polynomial nonlinearity, and that the structure of \eqref{dis} implies at least local well-posedness of the problem on a finite time interval $]0,\delta]$, $\delta<\infty$, in an appropriate functional space.
 	We write $ \vert \nabla\vert^\alpha(k) $ (resp. $ \mathcal{L}\left(\nabla\right)(k) $) for the operator $ \vert \nabla\vert^\alpha $ (resp. $ \mathcal{L}\left( \nabla \right) $) in Fourier space. The differential operators $\mathcal{L}\left(\nabla \right) $ and $\vert \nabla\vert^\alpha$ shall cast in Fourier  space into the form 
 	\begin{equs}\label{Lldef}
 		\mathcal{L}\left(\nabla \right)(k) = k^\sigma + \sum_{\gamma : |\gamma| < \sigma} a_{\gamma} \prod_{j} k_j^{\gamma_j} ,\qquad  \vert \nabla\vert^\alpha(k) =  \sum_{\gamma : |\gamma| \le \alpha} \prod_{j=1}^{d} k_j^{\gamma_j}
 	\end{equs}
 	for some $ \alpha \in \mathbb{R} $, $\sigma \in \mathbb{N}$, $ \gamma \in \mathbb{Z}^d $ and $ |\gamma| = \sum_i \gamma_i $,
 	where  for $k = (k_1,\ldots,k_d)\in \mathbb{Z}^d$ and $m = (m_1, \ldots, m_d)\in \mathbb{Z}^d$ we set \begin{equs}
 		k^\sigma  = k_1^\sigma + \ldots + k_d^\sigma, \qquad k  m = k_1 m_1 + \ldots + k_d m_d.
 	\end{equs} Then, making the change of variable $ 	v(t) =  e^{-it \mathcal{L}\left(\nabla\right)} {u}(t) $, one can rewrite \eqref{dis} in Fourier space as 
 	\begin{align}
 		\begin{split}
 			\partial_t v_k 
 			& =  - i \vert \nabla\vert^\alpha(k)
 			\sum_{\substack{k = \sum_{j=1}^n b_j k_j } }
 			e^{ it \Phi(\bar{k}) } 
 		\prod_{j=1}^n v_{k_j}^{b_j}
 		\end{split}
 		\label{phase_equation}
 	\end{align}
 	where we have supposed that $ p(u, \bar{u}) = \prod_{j=1}^n u^{b_j}$ with $ b_j \in \lbrace -1,1 \rbrace $, $ u^0 = u, u^{-1} = \bar{u} $ and $k, k_j \in \mathbb{Z}^d$. Here, $ \Phi(\bar{k}) $ is a phase defined by
 	\begin{equs}
 		 \Phi(\bar{k}) = \mathcal{L}\left(\nabla \right)(k)  + \sum_{j=1}^n b_j \mathcal{L}\left(\nabla \right)\left(b_j k_j\right).
 	\end{equs}
 	where the notation $\bar{k}$ is a short hand notation for $ k, (k_j)_{j \in \{ 1,...,n\}}$.
 	 In the Poincaré-Dulac normal form, one considers a domain of the frequencies such that the phase $ \Phi(\bar{k}) $ is non-zero (non-resonant). Then, one notices the following identity
 	\begin{equs} \label{integra_trick}
 e^{ i t \Phi(\bar{k}) } 	= 	\partial_{t} \left(   \frac{e^{i t \Phi(\bar{k})}}{i \Phi(\bar{k})} \right).
 	\end{equs}
 	One can proceed with an integration by parts that boils down to distribute the derivative in time on the $v_{k_i}^{b_i}$. One iterates the decomposition by substituting  $ \partial_t  v_{k_i}^{b_i}$ by the right-hand side of \eqref{phase_equation}, which introduces a new phase added to the phase $ \Phi(\bar{k}) $. This approach is not working directly for writing a numerical scheme. Indeed, one wants to be able to map back to physical space the term $ \frac{1}{\Phi(\bar{k})} $. This is possible for very simple cases and low-order schemes, as for the KdV equation, but in general one cannot use simple differential operators for expressing the previous operator in physical space. This is where Katharina Schratz proposed to do a splitting of the operator in Fourier space:
 	\begin{equs} \label{dom_low_decomp}
\Phi(\bar{k})  =  \Phi_{{\tiny{\text{low}}}}(\bar{k}) +	\Phi_{{\tiny{\text{dom}}}}(\bar{k}) 
 	\end{equs}
 with the property that now $ \frac{1}{\Phi_{{\tiny{\text{dom}}}}(\bar{k}) } $ can be mapped back to physical space with usual differential operators. Then, one has to perform an intermediate Taylor expansion before applying the trick \eqref{integra_trick}:  
 	\begin{equs} \label{Taylor_low_phi}
 		\begin{aligned}
 			e^{ i t \Phi(\bar{k}) } 	\prod_{j=1}^n v_{k_j}^{b_j}
 			& = 	e^{ i t \Phi_{{\tiny{\text{dom}}}}(\bar{k}) }   e^{ i t \Phi_{{\tiny{\text{low}}}}(\bar{k}) }
 			\prod_{j=1}^n v_{k_j}^{b_j}  
 			\\ & = \sum_{m =0}^r   \frac{t^m}{m!} i^m \Phi^m_{{\tiny{\text{low}}}}(\bar{k}) e^{ i t \Phi_{{\tiny{\text{dom}}}}(\bar{k}) } 	\prod_{j=1}^n v_{k_j}^{b_j} 
 			  \\ & + \mathcal{O}( t^{r+1}  \Phi^{r+1}_{{\tiny{\text{low}}}}(\bar{k})	\prod_{j=1}^n v_{k_j}^{b_j} 
 			 )
 		\end{aligned}
 	\end{equs}
 	where the error term corresponds to some regularity, one has to ask on the solution $v$ as it is multiplied by the monomial $ \Phi^{r+1}_{{\tiny{\text{\tiny{low}}}}}(\bar{k}) $. Classical approximation will have instead $  \Phi^{r+1}(\bar{k}) $, which requires, in general, more regularity if one chooses well the decomposition \eqref{dom_low_decomp}. Then, one applies this algorithm not to the equation but to the iterated integrals on the initial data that describe the solution. This gives us the following scheme
 		\begin{equs}\label{genscheme_low_intro}
 			U_{k}^{n,r}(t, u_0) =   \sum_{T \in \CT^{\leq r,k}_{0}} \frac{\Upsilon( T)(u_0)}{S(T)} (\Pi^{n,r}   T )(t)
 		\end{equs}
 		where $ \CT^{\leq r,k}_{0} $ is a set of decorated trees, $S(T)$ is a symmetry factor associated with $T$, $\Upsilon( T)(u_0)$ are elementary differentials depending on the initial value $u_0$. All these objects are introduced in Section \ref{Sec::2}. The term $(\Pi^{n,r}   T )(t) $ is a  low regularity discretisation of the oscillatory integral associated with $T$ that appears in the iteration of Duhamel's formula. It is given in Theorem \ref{main_theorem} (see Identity \eqref{main_formula_theorem}).
 		One has an explicit expression for the local error of this scheme given below
 		\begin{equs} \label{local_low_intro}
 			u_k(t) -	U_{k}^{n,r}(t, u_0)  = \sum_{T \in \CT^{\leq r,k}_{0}}  \mathcal{O}( t^{r+1}E^{n,r}(T) \Upsilon( T)(u_0) ).
 		\end{equs}
 		The definition of $E^{n,r}(T)$ is given in \eqref{local_error} and this local error is computed in \eqref{local_error_scheme} together with Corollary \ref{main_low_schemes}. The local error is written in Fourier space see Remark \ref{local_error_physical} for the physical space.

Part of the tools for deriving these schemes come from \cite{B24} by using a well-chosen arborification map $\mathfrak{a}$ applied to the decorated trees introduced in \cite{BS}. Then, one uses a character already described in mould calculus, which implements the summation of the  phases and their integration via $\eqref{integra_trick}$. The novelty and the difficulty of the present work is to incorporate the Taylor expansion steps \eqref{Taylor_low_phi} into this process. Then, one introduces extra monomials $t^m$ that can be differentiated in the next integration by parts which makes combinatorial factors hard to track. Another main novelty from \cite{B24} is the use of a forest formula for describing the low regularity schemes. In our case, the forest formula appears via an iteration of a Butcher-Connes-Kreimer type coproduct. This is an extra algebraic ingredient absent so far from the normal form derivations. Forest type formula also appears in the BPHZ algorithm \cite{BP57,Hepp,WZ69} for
renormalising Feynman diagrams (see \cite{CK2} for a Hopf algebraic description) and was later used for renormalising singular
SPDEs in \cite{BHZ,CH16}.
 
 After this brief explanation of the main ideas of this derivation, 	let us  make a couple of remarks related to this normal form derivation.
 	
 	\begin{remark} The low regularity schemes derived in the present paper  cover the different cases described in \cite{BS}, where one can fully Taylor-expand if some regularity is assumed a priori on the  solution in order to get simple schemes. 
 	To treat the general case, one has to amend the splitting between lower and dominant parts taking into account extra parameters. For doing so, we provide a new formula for the resonance analysis in Definition \ref{adaptative_splitting} that could simplify the description of the low regularity schemes given in  \cite{BS}.
 		\end{remark}
 		
 		\begin{remark}
 			This new derivation makes appear words for encoding the numerical scheme which is close in spirit to what has been performed in \cite{Murua2006,Murua2017} where words describe numerical integrators for dynamical systems. This is not surprising as the word series in \cite{Murua2017} are connected to mould calculus and normal forms.
 			\end{remark}
 			
 			\begin{remark}
 		The low regularity scheme in 		Theorem \ref{main_theorem} is introduced via a forest formula.
 		An attempt to give a forest formula was performed in \cite{ABBMS23} to characterise symmetric low regularity schemes. Such a formula is not very explicit, but was enough for this purpose.
 		Now, with 	Theorem \ref{main_theorem}, one has a more explicit forest formula with expressions for its coefficients.
 				\end{remark} 	
 	\begin{remark} We derive low regularity schemes via the Poincaré-Dulac normal form.
 		One can wonder if such a result can be obtained via the Birkhoff normal form. It is actually not clear, as obtaining low regularity symplectic schemes is quite challenging and remains open in the general case (see \cite{MS25}). The issue is that the Birkhoff normal form is very rigid because one composes iteratively with symplectic transforms. So the symplectic structure of the decomposition is preserved, which is not the case of the Poincaré-Dulac normal form where one proceeds mainly via integration by parts. 
 		\end{remark}

 		Finally, let us outline the paper by summarising the content of its sections. In Section \ref{Sec::2}, we recall the main combinatorial formalism for the decorated trees used in the description of the iterated integrals obtained from Duhamel's formulation. One important definition is  the map $ \Pi $ given in  \ref{def_Pi} that interprets decorated trees as iterated integrals. The computation of the phase in Definition \ref{dom_freq} is used in the sequel for performing the normal form derivation. The main result recalled in this section is the B-series approximation of the solution in Proposition \ref{tree_series} that reduces the problem to a discretisation of a finite number of iterated integrals for a fixed order.
 		
 		In Section \ref{Sec::3}, we recall the main algebraic tools needed for writing low regularity schemes. We start by presenting a variant of the Butcher-Connes-Kreimer coproduct $ \Delta_{\text{\tiny{BCK}}}  $ (see \ref{BCK_new}) already used in \cite{B24}. Then, we introduce the grafting product $ \curvearrowright $ in \eqref{grafting_a} whose dual map $ \curvearrowright^{*} $ is used to describe the arborification map $ \mathfrak{a}  $ in \eqref{arbo_new}. We finish the section with two technical algebraic results needed in the sequel: Proposition \ref{co-asso} that shows a type of co-associativity between $  \Delta_{\text{\tiny{BCK}}} $ and $ \curvearrowright^{*} $ and Proposition \ref{int_decomp} that provides a factorisation of $ \Pi $ in terms of a root decomposition described by \eqref{star_decomp}.
 		
 		In Section \ref{Sec::4}, we illustrate the Poincaré-Dulac normal form on the NLS equation, recalling the main steps of its derivation. We present the main character $\tilde{\Psi}$ in \eqref{character_main} that encodes  the various phases obtained via  integrations by parts. We illustrate the extra Taylor expansion step and explain the gain obtained in regularity in contrast with the full Taylor expansion. In Proposition \ref{prop_inte_parts}, we show how the presence of a monomial $t^n$ modifies the character $\tilde{\Psi}$ by introducing lower part monomials and also many combinatorial factors. We finish the section by a brief explanation about the necessity to apply this scheme to iterated integrals. 
 		
 		In Section \ref{Sec::5}, we start by recalling the main strategy for performing the decomposition \eqref{dom_low_decomp} via Definition \ref{proj_dom}.  Then, we compute inductively in Definition \ref{def_dom}, the lower and dominant phases used when one performs the normal form approach. This definition has been extended if one wants to consider cases with sufficient regularity that lead to some full Taylor expansion. This is performed in Definition \ref{adaptative_splitting}, where new parameters have to be taken into account.  Definition \ref{def_Phi} introduces the modified $\tilde{\Psi}$ that takes into account the new Taylor expansion steps in the derivation. This is the main new definition for the low regularity schemes.  After a couple of algebraic results Propositions \ref{partial_t_pi} and \ref{partial_derivative}, one is able to state the main important technical result that makes appear inductively the forest formula: Proposition \ref{Taylor_phi}. The main formula for the scheme and its local error are given in Theorem \ref{main_theorem}.
 		The proof of this Theorem relies on Corollary \ref{main_corollary}, a consequence of Proposition \ref{Taylor_phi}.
 		Under Assumption \ref{Assumption_1}, which is proved in Proposition  \ref{proof_dis} for equations of type \eqref{dis}, one expects to have the same local error than the low regularity schemes coming from \cite{BS}. The idea of this assumption is that the order of the Taylor expansion of lower and dominant parts may not matter in the local error.  
	
  \subsection*{Acknowledgements}
 {\small
 	Y. B. gratefully acknowledges funding support from the European Research Council (ERC) through the ERC Starting Grant Low Regularity Dynamics via Decorated Trees (LoRDeT), grant agreement No.\ 101075208.
 	Views and opinions expressed are however those of the author(s) only and do not necessarily reflect those of the European Union or the European Research Council. Neither the European Union nor the granting authority can be held responsible for them. This work has been completed  while the author was in
 	residence at the Simons Laufer Mathematical Sciences Institute in 
 	Berkeley, California, during the Fall 2025 semester. This residence is partially supported by the National Science
 	Foundation under Grant No. DMS-2424139.
 }

\section{Fourier decorated trees}

\label{Sec::2}
In this section, we introduce the necessary decorated trees for encoding the low regularity schemes. We follow the formalism of \cite{BS}. We consider a  finite set of   frequencies $ k_1,...,k_m \in \mathbb{Z}^{d}$ and a finite set $\Lab$ that  parametrises a family of polynomials $ (P_{\Labhom})_{\Labhom \in \Lab} $ in the frequencies that represent operators with constant coefficients in Fourier space.  
	We denote by $\mcT$  the set of decorated trees  as elements of the form $T_\mfe^{\mff}$ where 	
	\begin{itemize}	
		\item $ T $ is a non-planar rooted tree with root node $ \varrho $, edge set $ E_T $ and node set $ N_T $. We consider only planted trees which means that $ T $ has only one edge connecting its root $ \varrho $.		
		\item $ \mfe:E_T\rightarrow \mfL\times \{0,1 \} $ is the set of edge decorations $ \mfe(\cdot) = (\mft(\cdot), \mfp(\cdot)) $ where the first component selects the correct polynomial $ P_{\mft} $ when $ \mft \in \Lab $.		
		\item $ \mff:N_T\setminus\{\varrho\}\rightarrow \mathbb{Z}^d $ are node decorations  satisfying the relation for every inner node $ u $
		\begin{equs} \label{frequencies_identity}
			(-1)^{\mfp(e_u)}\mff(u)=\sum_{e = (u,v)\in E_T}(-1)^{\mfp(e)}\mff(v)
		\end{equs}
		where $ e_u$ is the edge outgoing $u$ of the form $ (w,u) $. The node decorations  $(\mff(u))_{u \in L_T}$ determine the decorations of the inner nodes. We assume
		that the node decorations at the leaves are linear combinations of the $k_i$ with coefficients in $ \lbrace -1,0,1 \rbrace$.
	\end{itemize}

We define $ \CH $ as the linear span of $H$ the forests containing decorated trees in $ \CT $. A forest is unordered collection of trees. The empty forest is denoted by $ \one $.
We use a symbolic notation for denoting these decorated trees. An  edge decorated by $ o = (\mft, \mfp) $ with $ \mft \in \Lab $ is denoted by $ \mathcal{I}_{o} $. The symbol $ \mathcal{I}_{o} (\lambda_{k}
\cdot) : \CH \rightarrow \CH $ is viewed as the operation that merges all the roots of the trees composing the forest into one node decorated by $ k \in \mathbb{Z}^d$. We obtain
a decorated tree which is then grafted onto a new root with no decoration. If the condition \eqref{frequencies_identity} is not
satisfied on the argument, then $ \mathcal{I}_{o} (\lambda_{k}
\cdot) $  gives zero.
We denote the forest product by $\cdot$. In the sequel, we will omit it and make the following abuse of notation:
\begin{equs}
	\mathcal{I}_{(\mathfrak{t}_1,\mathfrak{p}_1)}(\lambda_{k_1} F_1) \cdot \mathcal{I}_{(\mathfrak{t}_2,\mathfrak{p}_2)}(\lambda_{k_2} F_2) = \mathcal{I}_{(\mathfrak{t}_1,\mathfrak{p}_1)}(\lambda_{k_1} F_1)  \mathcal{I}_{(\mathfrak{t}_2,\mathfrak{p}_2)}(\lambda_{k_2} F_2)
\end{equs}
where $F_1, F_2 \in \mathcal{H}$.
Below, we provide an example where we introduce some graphical notations. Let  $ \mfL = \{\mft_1,\mft_2\}$, 
 one has 
\begin{equs} \label{exemple_1}
	T =  \begin{tikzpicture}[scale=0.2,baseline=-5]
		\coordinate (root) at (0,0);
		\coordinate (tri) at (0,-2);
		\coordinate (t1) at (-2,2);
		\coordinate (t2) at (2,2);
		\coordinate (t3) at (0,3);
		\draw[kernels2,tinydots] (t1) -- (root);
		\draw[kernels2] (t2) -- (root);
		\draw[kernels2] (t3) -- (root);
		\draw[symbols] (root) -- (tri);
		\node[not] (rootnode) at (root) {};t
		\node[not] (trinode) at (tri) {};
		\node[var] (rootnode) at (t1) {\tiny{$ k_{\tiny{1}} $}};
		\node[var] (rootnode) at (t3) {\tiny{$ k_{\tiny{2}} $}};
		\node[var] (trinode) at (t2) {\tiny{$ k_3 $}};
	\end{tikzpicture}
=	  \CI_{(\mft_2,0)}(\lambda_{k}\CI_{(\mft_1,1)}(\lambda_{k_1})\CI_{(\mft_1,0)}(\lambda_{k_2})\CI_{(\mft_1,0)}(\lambda_{k_3}))
\end{equs}
where $ k = -k_1 + k_2 + k_3 $ and we have denoted $ \CI_{(\mft_1,1)}(\lambda_{k_1} \one) $ by $ \CI_{(\mft_1,1)}(\lambda_{k_1}) $. A blue edge encodes $ (\mathfrak{t}_2,0) $, a brown edge is used for $ (\mathfrak{t}_1,0) $ and a dashed brown edge is for $ (\mathfrak{t}_2,1) $. 
The frequency decorations appear on the leaves of the previous tree. One does not have to make explicit the frequency decorations  for the inner nodes as they are determined by the ones coming from the leaves.  In the definition below, we define the phase of the oscillatory integrals associated with the decorated trees described above:

\begin{definition} \label{dom_freq} We recursively define $\mathscr{F} :  H \rightarrow \mathbb{R}[\mathbb{Z}^d]$ as:
	\begin{equs}
\,	&	\mathscr{F}(\one)   = 0, \quad
		\mathscr{F}
		(F \cdot \bar F)  =\mathscr{F}(F) + \mathscr{F}(\bar F), \\
	&	\mathscr{F}\left( \CI_{(\Labhom,\Labp)}(  \lambda_{k}F) \right)   =     P_{(\Labhom,\Labp)}(k) +\mathscr{F}(F)  
	\end{equs}
where one has $
	P_{(\Labhom,\Labp)}(k) = (-1)^{\Labp} P_{\Labhom}((-1)^{\Labp}k).$
\end{definition}
We continue the example given in \eqref{exemple_1} and now we set $P_{\mathfrak{t}_1}(\lambda) = -\lambda^2$ and $P_{\mathfrak{t}_2}(\lambda) = \lambda^2$. One has for $ \mathscr{F}(T) $:
\begin{equs}
	\mathscr{F}(T) & =  (-k_1+k_2+k_3)^{2} + (-k_1)^{2} - k_2^{2} - k_3^2.  
\end{equs}
Given a decorated tree  $ T_\mfe^{\mff} $, we define the order of a tree denoted by $ \vert \cdot \vert_{\text{\tiny{ord}}} $
by
\begin{equation*}
	\vert  T_\mfe^{\mff} \vert_{\text{\tiny{ord}}} = \sum_{e \in E_T }   \one_{\lbrace \mathfrak{t}(e) = \mathfrak{t}_2 \rbrace}.
\end{equation*} 
which corresponds to the number of blue edges in a tree. One has $ \vert T \vert_{\text{\tiny{ord}}} = 1$. We also denote by $ o_i $ edge decorations of the form $(\mft_i,p_i)$ for $i \in \{1,2\}$. We will not considered all decorated trees but only those that are associated with a perturbative expansion of dispersive equation with low regularity initial data.
The main example in this paper is the cubic nonlinear  Schrödinger equation given by
\begin{align}
	\begin{cases}
		i \partial_t u + \partial_x^2  u  = | u |^{2} u \\
		u |_{t = 0} = u_0,
	\end{cases}
	\quad (x, t) \in \mathbb{T}^d \times \mathbb{R}.
	\label{NLS}
\end{align}
One can rewrite the previous equation \eqref{NLS} in the Duhamel's form:
\begin{equs}
	u(t) = e^{i t \Delta} u(0) - i e^{it \Delta} \int_0^t e^{- i s \Delta} \left( \bar{u}(s) u(s)^2 \right) ds
\end{equs}
which is given in Fourier space as
\begin{equs} \label{duhamel_Fourier}	
	\begin{aligned}
		u_k(t) & = e^{-i t k^2} u_k(0) \\ &  - i \sum_{k=-k_1 + k_2 + k_3} e^{-it k^2} \int_0^t e^{ i s k^2} \left( \bar{u}_{k_1}(s) u_{k_2}(s) u_{k_3}(s) \right) ds
	\end{aligned}
\end{equs}
where the operator $ e^{it \Delta} $ is sent to $ e^{-it k^2} $ in Fourier space. Also, pointwise products are sent to convolution products on the frequencies. Here compare to \eqref{dis}, one has
\begin{equs}
\mathcal{L}\left(\nabla\right) = \Delta,\quad	p(v,\bar{v}) = v^2 \bar{v}, \quad \alpha =0.
\end{equs}
We start an expansion by:
\begin{equs}
	u_k(t) = e^{-i t k^2} u_k(0) + \mathcal{O}(t)
\end{equs}
and then we obtain
\begin{equs}
	\begin{aligned}
		u_k(t) & = e^{-i t k^2} u_k(0) \\ &  - i \sum_{k=-k_1 + k_2 + k_3} e^{-it k^2} \int_0^t e^{ i s (k^2+k_1^2-k_2^2-k_3^2)} \left( \bar{u}_{k_1}(0) u_{k_2}(0) u_{k_3}(0) \right) ds + \mathcal{O}(t^2)
	\end{aligned}
\end{equs}
The expansion above produces   the oscillatory integrals that are encoded by the previous decorated trees. These trees are generated by specific rules that come from the nonlinearity of the equation considered. The cubic term in \eqref{NLS} generates ternary decorated trees and the conjugate in the product will produce edges decorated by $(\mathfrak{t}_1,1)$. Based on this observation, one constructs recursively the decorated trees associated with \eqref{dis}:
\begin{equs}
	& \CT_0 = \lbrace \CI_{(\Labhom_1,0)}( \lambda_k \CI_{(\Labhom_2,0)}( \lambda_k    \prod_{j=1}^n T_{j}  ) ), \CI_{(\Labhom_1,0)}(\lambda_k) \,   : \,T_j\in \CT_{a_j},  \, k \in \mathbb{Z}^{d}  \rbrace \\
	& \CT_1  = \lbrace \CI_{(\Labhom_1,1)}( \lambda_k \CI_{(\Labhom_2,1)}( \lambda_k   \prod_{j=1}^n T_{j} ) ), \CI_{(\Labhom_1,1)}(\lambda_k) \,   : \, T_j \in \mathcal{T}_{\bar{a}_j}, \, k \in \mathbb{Z}^{d}  \rbrace
\end{equs}
where $  \bar{a}_j = a_j +1 \mod 2 $.
The set $\mathcal{T}_0$ encodes the decorated trees of the Duhamel's formula whereas $\mathcal{T}_1$ is needed for describing its conjugate.
We define $   \CT_j^k  $ for $j \in \lbrace 0,1 \rbrace$ as the subspace of $ \CT_j $ such that the frequency decorations on the nodes connected to the root is $ k $. For $ r\in\mathbb{N} $ and $ j \in \lbrace 0,1 \rbrace $ we set
\begin{equs} \label{space_deco}
	\mcT^{\leq r,k}_j= \cup_{m=0}^r \mcT^{m,k}_j, \quad \mcT^{m,k}_j=\{T_\mfe^\mff\in \mcT^k_j ,|T^\mff_\mfe|_{\text{\tiny{ord}}} =  m \}.
\end{equs}
One can summarise the definition above by saying that   $ \mathcal{T}^{k}_j $ is graded by $ \vert \cdot \vert_{\text{\tiny{ord}}} $.

 Before presenting the B-series type expansion given in \eqref{genscheme}, we need to introduce more definitions. One of them is the symmetry factor associated with these decorated trees.
 Let $ T^{\mff}_{\mfe}\in\mcT $, we consider only the edge decoration, i.e,  $ T_\mfe := \mathcal{P}^{0} T^{\mff}_{\mfe}  $ where $ \mathcal{P}^{0} $ set to zero all the node decorations. Then, one sets $ S(\boldsymbol{1}) = 1 $ and defines inductively
\begin{equs} \label{symmetry_factor}
	S(T):=\prod_{i,j}S(T_{i,j})^{\gamma_{i,j}}\gamma_{i,j}!
\end{equs}
where  $ T_\mfe  =  \prod_{i,j}\CI_{(\mft_{t_i}, p_i)}(T_{i,j})^{\gamma_{i,j}} $ and $T_{i,j} \neq T_{i,\ell}$ for $j \neq \ell$.
We now define the elementary differentials denoted by $ \Upsilon(T)(v) $ from two given polynomials $ p_0(v,\bar{v}) $ $ p_1(v,\bar{v}) $ as 
\begin{equs}
\hat{T}  & = 
 \CI_{(\Labhom_2,a)}( \lambda_k   F  ), \quad 	T   = 
	\CI_{(\Labhom_1,a)}\left( \lambda_k \hat{T} \right) \quad,   \quad a \in \lbrace 0,1 \rbrace  
\end{equs}
with 
\begin{equs}
	F = \prod_{i=1}^n \CI_{(\Labhom_1,0)}( \lambda_{k_i} T_i) \prod_{j=1}^m \CI_{(\Labhom_1,1)}( \lambda_{\tilde k_j} \tilde T_j)  )
\end{equs}
by
\begin{equs} \label{def_Upsilon}
	\begin{aligned}
	\Upsilon(T)(v) = \Upsilon(\hat{T})(v) \, & { :=}  \partial_v^{n} \partial_{\bar v}^{m} p_a(v,\bar v) \prod_{i=1}^n  \Upsilon( \CI_{(\Labhom_1,0)}\left( \lambda_{k_i}  T_i \right) )(v)  \\ & \prod_{j=1}^m \Upsilon( \CI_{(\Labhom_1,1)}( \lambda_{\tilde k_j}\tilde T_j ) )(v)
	\end{aligned}
\end{equs}
where 
\begin{equs} 
	\label{mid_point_rule}
	\begin{aligned}
		\Upsilon(\CI_{(\Labhom_1,0)}( \lambda_{k})  )(v)  \,  { :=}  v_k, \quad 
		\Upsilon(\CI_{(\Labhom_1,1)}( \lambda_{k})  )(v)  \, & { :=}  \bar{v}_k.
	\end{aligned}
\end{equs}
Above, we have used the notation:
\begin{equs}
	p_0(v,\bar{v}) = p(v,\bar{v}), \quad p_{1}(v,\bar{v}) = \overline{p(v, \bar{v})}. 
\end{equs}

The only definition missing is the oscillatory integrals associated with the decorated trees. They are given inductively via a map $ \Pi :  \CH \rightarrow \mathcal{C} $ where $  \mathcal{C} $ is the vector space containing  functions of the form $ z \mapsto \sum_j Q_j(z) e^{i z P_j(k_1,...,k_n)} $  and the $ Q_j(z) $ are polynomials in
$ z $, the $ P_j $ are polynomials in  $ k_1,...,k_n $, and each $ Q_j $ depends implicitly on $k_i$. 
The map $ \Pi $ is defined for every forest $ F $ by
\begin{equs}
	\label{def_Pi}
	\begin{aligned}
		\Pi (\mathcal{I}_{o_1}(\lambda_{k} F))(t) &= e^{i t  P_{o_1}(k)}  (\Pi F)(t), \\
		\Pi (\mathcal{I}_{o_2}(\lambda_{k} F))(t) &= - i |\nabla|^\alpha(k)\int_0^t e^{i s P_{o_2}(k)} \Pi(F)(s)d s,  \\
		\Pi ( F_1 \cdot F_2)(t) & = \Pi ( F_1 )(t) \, \Pi ( F_2)(t).
	\end{aligned}
\end{equs} 

In the next proposition, we recall \cite[Prop. 4.3]{BS} that shows how to encode iterations of Duhamel's formula with the previous definitions
\begin{proposition} \label{tree_series}
	The tree series given by 
	\begin{equs}\label{genscheme}
		U_{k}^{r}(t, u_0) =   \sum_{T \in \CT^{\leq r,k}_{0}} \frac{\Upsilon( T)(u_0)}{S(T)} (\Pi   T )(t)
	\end{equs}
	is the $k$th Fourier coefficient of the solution of $ \eqref{dis} $ up to order $ r $ which means that one has
	\begin{equs}
		u_{k}(t) - 	U_{k}^{r}(t, u_0) = \mathcal{O}(t^{r+1}).
	\end{equs}
\end{proposition}
As an example, we can list the elements of $ \CT_0^{\leq 1,k} $ and $ \CT_0^{\leq 2,k} $ for the equation \eqref{NLS} below:
\begin{equs}
	\CT^{\leq 1,k}_{0} & = \left\lbrace T_0, T_1, \, k_i \in \mathbb{Z}^d \right\rbrace, \quad T_0 =  \begin{tikzpicture}[scale=0.2,baseline=-5]
		\coordinate (root) at (0,1);
		\coordinate (tri) at (0,-1);
		\draw[kernels2] (tri) -- (root);
		\node[var] (rootnode) at (root) {\tiny{$ k $}};
		\node[not] (trinode) at (tri) {};
	\end{tikzpicture} , \quad T_1 =  \begin{tikzpicture}[scale=0.2,baseline=-5]
		\coordinate (root) at (0,2);
		\coordinate (tri) at (0,0);
		\coordinate (trib) at (0,-2);
		\coordinate (t1) at (-2,4);
		\coordinate (t2) at (2,4);
		\coordinate (t3) at (0,5);
		\draw[kernels2,tinydots] (t1) -- (root);
		\draw[kernels2] (t2) -- (root);
		\draw[kernels2] (t3) -- (root);
		\draw[kernels2] (trib) -- (tri);
		\draw[symbols] (root) -- (tri);
		\node[not] (rootnode) at (root) {};
		\node[not] (trinode) at (tri) {};
		\node[var] (rootnode) at (t1) {\tiny{$ k_{\tiny{1}} $}};
		\node[var] (rootnode) at (t3) {\tiny{$ k_{\tiny{2}} $}};
		\node[var] (trinode) at (t2) {\tiny{$ k_{\tiny{3}} $}};
		\node[not] (trinode) at (trib) {};
	\end{tikzpicture}, 
\\
	\CT^{\leq 2,k}_{0} & = \left\lbrace T_0,T_1,T_2,T_3, \, k_i \in \mathbb{Z}^d \right\rbrace, \quad T_2 = \begin{tikzpicture}[scale=0.2,baseline=-5]
		\coordinate (root) at (0,2);
		\coordinate (tri) at (0,0);
		\coordinate (trib) at (0,-2);
		\coordinate (t1) at (-2,4);
		\coordinate (t2) at (2,4);
		\coordinate (t3) at (0,4);
		\coordinate (t4) at (0,6);
		\coordinate (t41) at (-2,8);
		\coordinate (t42) at (2,8);
		\coordinate (t43) at (0,10);
		\draw[kernels2,tinydots] (t1) -- (root);
		\draw[kernels2] (t2) -- (root);
		\draw[kernels2] (t3) -- (root);
		\draw[symbols] (root) -- (tri);
		\draw[symbols] (t3) -- (t4);
		\draw[kernels2,tinydots] (t4) -- (t41);
		\draw[kernels2] (t4) -- (t42);
		\draw[kernels2] (t4) -- (t43);
		\draw[kernels2] (trib) -- (tri);
		\node[not] (trinode) at (trib) {};
		\node[not] (rootnode) at (root) {};
		\node[not] (rootnode) at (t4) {};
		\node[not] (rootnode) at (t3) {};
		\node[not] (trinode) at (tri) {};
		\node[var] (rootnode) at (t1) {\tiny{$ k_{\tiny{4}} $}};
		\node[var] (rootnode) at (t41) {\tiny{$ k_{\tiny{1}} $}};
		\node[var] (rootnode) at (t42) {\tiny{$ k_{\tiny{3}} $}};
		\node[var] (rootnode) at (t43) {\tiny{$ k_{\tiny{2}} $}};
		\node[var] (trinode) at (t2) {\tiny{$ k_5 $}};
	\end{tikzpicture}, \quad T_3 = \begin{tikzpicture}[scale=0.2,baseline=-5]
		\coordinate (root) at (0,2);
		\coordinate (tri) at (0,0);
		\coordinate (trib) at (0,-2);
		\coordinate (t1) at (-2,4);
		\coordinate (t2) at (2,4);
		\coordinate (t3) at (0,4);
		\coordinate (t4) at (0,6);
		\coordinate (t41) at (-2,8);
		\coordinate (t42) at (2,8);
		\coordinate (t43) at (0,10);
		\draw[kernels2] (t1) -- (root);
		\draw[kernels2] (t2) -- (root);
		\draw[kernels2,tinydots] (t3) -- (root);
		\draw[symbols] (root) -- (tri);
		\draw[symbols,tinydots] (t3) -- (t4);
		\draw[kernels2] (t4) -- (t41);
		\draw[kernels2,tinydots] (t4) -- (t42);
		\draw[kernels2,tinydots] (t4) -- (t43);
		\draw[kernels2] (trib) -- (tri);
		\node[not] (trinode) at (trib) {};
		\node[not] (rootnode) at (root) {};
		\node[not] (rootnode) at (t4) {};
		\node[not] (rootnode) at (t3) {};
		\node[not] (trinode) at (tri) {};
		\node[var] (rootnode) at (t1) {\tiny{$ k_{\tiny{4}} $}};
		\node[var] (rootnode) at (t41) {\tiny{$ k_{\tiny{1}} $}};
		\node[var] (rootnode) at (t42) {\tiny{$ k_{\tiny{3}} $}};
		\node[var] (rootnode) at (t43) {\tiny{$ k_{\tiny{2}} $}};
		\node[var] (trinode) at (t2) {\tiny{$ k_5 $}};
	\end{tikzpicture}.
\end{equs}

Then, one can compute the various components of the Butcher-type series given in \eqref{genscheme}. For example, the symmetry factor $S$ is given by:
\begin{equs}
	S(T_0) = 1, \quad S(T_1) = S(T_2) = 2, \quad S(T_3) = 4.
	\end{equs}
Let us stress that the symmetry factor does not take into account the frequency decorations $k_i$ but only the edge decorations. For the elementary differentials, one has
\begin{equs}
	\Upsilon[T_0](v) = v_k, \quad \Upsilon[T_1](v) = 2\bar{v}_{k_1} v_{k_2} v_{k_3}, \quad \Upsilon[T_3](v) = 4 \bar{v}_{k_1} v_{k_2} v_{k_3} \bar{v}_{k_4} v_{k_5}  ,
\end{equs}
where the factors $2$ and $4$ come from the derivation of the monomial $ u^2 $. For the iterated integrals, one has recursively: 
\begin{equs}
	(\Pi  \begin{tikzpicture}[scale=0.2,baseline=-5]
		\coordinate (root) at (0,1);
		\coordinate (tri) at (0,-1);
		\draw[kernels2] (tri) -- (root);
		\node[var] (rootnode) at (root) {\tiny{$ k_2 $}};
		\node[not] (trinode) at (tri) {};
	\end{tikzpicture}) (t) & = e^{-i t k_2^2}, \quad (\Pi  \begin{tikzpicture}[scale=0.2,baseline=-5]
		\coordinate (root) at (0,1);
		\coordinate (tri) at (0,-1);
		\draw[kernels2,tinydots] (tri) -- (root);
		\node[var] (rootnode) at (root) {\tiny{$ k_1 $}};
		\node[not] (trinode) at (tri) {};
	\end{tikzpicture}) (t) = e^{i t k_1^2}, \quad 
	(\Pi \begin{tikzpicture}[scale=0.2,baseline=-5]
		\coordinate (root) at (0,-1);
		\coordinate (t1) at (-2,1);
		\coordinate (t2) at (2,1);
		\coordinate (t3) at (0,2);
		\draw[kernels2,tinydots] (t1) -- (root);
		\draw[kernels2] (t2) -- (root);
		\draw[kernels2] (t3) -- (root);
		\node[not] (rootnode) at (root) {};t
		\node[var] (rootnode) at (t1) {\tiny{$ k_{\tiny{1}} $}};
		\node[var] (rootnode) at (t3) {\tiny{$ k_{\tiny{2}} $}};
		\node[var] (trinode) at (t2) {\tiny{$ k_3 $}};
	\end{tikzpicture}  )(t) = e^{i t (k_1^2 - k_2^2 - k_3^2)},
	\\ ( \Pi \begin{tikzpicture}[scale=0.2,baseline=-5]
		\coordinate (root) at (0,0);
		\coordinate (tri) at (0,-2);
		\coordinate (t1) at (-2,2);
		\coordinate (t2) at (2,2);
		\coordinate (t3) at (0,3);
		\draw[kernels2,tinydots] (t1) -- (root);
		\draw[kernels2] (t2) -- (root);
		\draw[kernels2] (t3) -- (root);
		\draw[symbols] (root) -- (tri);
		\node[not] (rootnode) at (root) {};t
		\node[not] (trinode) at (tri) {};
		\node[var] (rootnode) at (t1) {\tiny{$ k_{\tiny{1}} $}};
		\node[var] (rootnode) at (t3) {\tiny{$ k_{\tiny{2}} $}};
		\node[var] (trinode) at (t2) {\tiny{$ k_3 $}};
	\end{tikzpicture}) (t) & = -i \int^{t}_0 e^{is (-k_1 + k_2 + k_3)^2} e^{i s (k_1^2 - k_2^2 - k_3^2)} ds.
\end{equs}
with these computations, one has
\begin{equs}
	(\Pi T_0)(t) & = e^{-it k^2}, \\	(\Pi T_1)(t) & = e^{-it k^2} (\Pi \begin{tikzpicture}[scale=0.2,baseline=-5]
		\coordinate (root) at (0,0);
		\coordinate (tri) at (0,-2);
		\coordinate (t1) at (-2,2);
		\coordinate (t2) at (2,2);
		\coordinate (t3) at (0,3);
		\draw[kernels2,tinydots] (t1) -- (root);
		\draw[kernels2] (t2) -- (root);
		\draw[kernels2] (t3) -- (root);
		\draw[symbols] (root) -- (tri);
		\node[not] (rootnode) at (root) {};t
		\node[not,label= {[label distance=-0.2em]below: \scriptsize  $  $}] (trinode) at (tri) {};
		\node[var] (rootnode) at (t1) {\tiny{$ k_{\tiny{1}} $}};
		\node[var] (rootnode) at (t3) {\tiny{$ k_{\tiny{2}} $}};
		\node[var] (trinode) at (t2) {\tiny{$ k_3 $}};
	\end{tikzpicture})(t)  =  -i e^{-it k^2} \int^{t}_0 e^{is (-k_1 + k_2 + k_3)^2} e^{i s (k_1^2 - k_2^2 - k_3^2)} ds.
\end{equs}

\section{Butcher-Connes-Kreimer coproduct and Arborification}
\label{Sec::3}
In this section, we introduce the necessary combinatorial maps that will allow us to organise and encode the derivation of the low regularity schemes.
One starts with a variant of the Butcher-Connes-Kreimer coproduct (see \cite{Butcher,CK1} where this coproduct appears for composing B-series and for renormalising Feynman diagrams) on the decorated trees previously introduced. We first consider the space $\CH_2$ as a subspace of $\CH$ containing forests with planted trees of the form $ \CI_{o_2}(\lambda_k F) $. The Butcher-Connes-Kreimer type coproduct $ \Delta_{\text{\tiny{BCK}}} : \CH \rightarrow \CH_2 \otimes \CH $,  is defined recursively by
\begin{equs} \label{BCK_new}
	\begin{aligned}
	\Delta_{\text{\tiny{BCK}}} \CI_{o_1}( \lambda_{k}  F ) & = \left( \id \, \otimes \CI_{o_1}( \lambda_{k}  \cdot )  \right) \Delta_{\text{\tiny{BCK}}} F,  \\ 
	\Delta_{\text{\tiny{BCK}}} \CI_{o_2}( \lambda_{k}  F ) & = \left( \id \, \otimes \CI_{o_2}( \lambda_{k}  \cdot ) \right) \Delta_{\text{\tiny{BCK}}} F +  \CI_{o_2}( \lambda_{k}  F ) \otimes \one.
	\end{aligned}
\end{equs}
and then extended multiplicatively for the forest product which means that one has 
\begin{equs} \label{mult_Delta}
	\Delta_{\text{\tiny{BCK}}} (F_1 \cdot F_2) = \Delta_{\text{\tiny{BCK}}} F_1 \Delta_{\text{\tiny{BCK}}}  F_2   
\end{equs}
where the product above is
\begin{equs}
	(F_1 \otimes F_2) (F_3 \otimes F_4) = F_1 \cdot F_3 \otimes F_2 \cdot F_4.
\end{equs}
From the definition, it is easy to see that one has $ \Delta_{\text{\tiny{BCK}}} : \CH_2 \rightarrow \CH_2 \otimes \CH_2 $.
 This coproduct is a simple version of the one introduced in \cite{BS}, which deals with more decorations on the trees for encoding Taylor expansions in time. A natural interpretation of \eqref{BCK_new} is to use cuts. Edges of the form $ \mathcal{I}_{o_2} $ can be put on the left-hand side of the tensor product with the decorated trees connected to this edge. This corresponds to cutting this edge. One puts all the cuts on the left-hand side of the tensor product with the property that the cut must be admissible. This means that two edges cut cannot lie on the same path to the root.
Below, we provide some examples of computations:
\begin{equs}
	\\ & \Delta_{\text{\tiny{BCK}}} \begin{tikzpicture}[scale=0.2,baseline=-5]
		\coordinate (root) at (0,0);
		\coordinate (tri) at (0,-2);
		\coordinate (t1) at (-2,2);
		\coordinate (t2) at (2,2);
		\coordinate (t3) at (0,3);
		\draw[kernels2,tinydots] (t1) -- (root);
		\draw[kernels2] (t2) -- (root);
		\draw[kernels2] (t3) -- (root);
		\draw[symbols] (root) -- (tri);
		\node[not] (rootnode) at (root) {};t
		\node[not,label= {[label distance=-0.2em]below: \scriptsize  $ $}] (trinode) at (tri) {};
		\node[var] (rootnode) at (t1) {\tiny{$ k_{\tiny{1}} $}};
		\node[var] (rootnode) at (t3) {\tiny{$ k_{\tiny{2}} $}};
		\node[var] (trinode) at (t2) {\tiny{$ k_3 $}};
	\end{tikzpicture}  =   \begin{tikzpicture}[scale=0.2,baseline=-5]
		\coordinate (root) at (0,0);
		\coordinate (tri) at (0,-2);
		\coordinate (t1) at (-2,2);
		\coordinate (t2) at (2,2);
		\coordinate (t3) at (0,3);
		\draw[kernels2,tinydots] (t1) -- (root);
		\draw[kernels2] (t2) -- (root);
		\draw[kernels2] (t3) -- (root);
		\draw[symbols] (root) -- (tri);
		\node[not] (rootnode) at (root) {};t
		\node[not,label= {[label distance=-0.2em]below: \scriptsize  $ $}] (trinode) at (tri) {};
		\node[var] (rootnode) at (t1) {\tiny{$ k_{\tiny{1}} $}};
		\node[var] (rootnode) at (t3) {\tiny{$ k_{\tiny{2}} $}};
		\node[var] (trinode) at (t2) {\tiny{$ k_3 $}};
	\end{tikzpicture} \otimes \one + \one \otimes  \begin{tikzpicture}[scale=0.2,baseline=-5]
		\coordinate (root) at (0,0);
		\coordinate (tri) at (0,-2);
		\coordinate (t1) at (-2,2);
		\coordinate (t2) at (2,2);
		\coordinate (t3) at (0,3);
		\draw[kernels2,tinydots] (t1) -- (root);
		\draw[kernels2] (t2) -- (root);
		\draw[kernels2] (t3) -- (root);
		\draw[symbols] (root) -- (tri);
		\node[not] (rootnode) at (root) {};t
		\node[not,label= {[label distance=-0.2em]below: \scriptsize  $ $}] (trinode) at (tri) {};
		\node[var] (rootnode) at (t1) {\tiny{$ k_{\tiny{1}} $}};
		\node[var] (rootnode) at (t3) {\tiny{$ k_{\tiny{2}} $}};
		\node[var] (trinode) at (t2) {\tiny{$ k_3 $}};
	\end{tikzpicture},
\\
	& \Delta_{\text{\tiny{BCK}}}  \begin{tikzpicture}[scale=0.2,baseline=-5]
		\coordinate (root) at (0,0);
		\coordinate (tri) at (0,-2);
		\coordinate (t1) at (-2,2);
		\coordinate (t2) at (2,2);
		\coordinate (t3) at (0,2);
		\coordinate (t4) at (0,4);
		\coordinate (t41) at (-2,6);
		\coordinate (t42) at (2,6);
		\coordinate (t43) at (0,8);
		\draw[kernels2,tinydots] (t1) -- (root);
		\draw[kernels2] (t2) -- (root);
		\draw[kernels2] (t3) -- (root);
		\draw[symbols] (root) -- (tri);
		\draw[symbols] (t3) -- (t4);
		\draw[kernels2,tinydots] (t4) -- (t41);
		\draw[kernels2] (t4) -- (t42);
		\draw[kernels2] (t4) -- (t43);
		\node[not] (rootnode) at (root) {};
		\node[not] (rootnode) at (t4) {};
		\node[not] (rootnode) at (t3) {};
		\node[not,label= {[label distance=-0.2em]below: \scriptsize  $  $}] (trinode) at (tri) {};
		\node[var] (rootnode) at (t1) {\tiny{$ k_{\tiny{4}} $}};
		\node[var] (rootnode) at (t41) {\tiny{$ k_{\tiny{1}} $}};
		\node[var] (rootnode) at (t42) {\tiny{$ k_{\tiny{3}} $}};
		\node[var] (rootnode) at (t43) {\tiny{$ k_{\tiny{2}} $}};
		\node[var] (trinode) at (t2) {\tiny{$ k_5 $}};
	\end{tikzpicture}  =
	\begin{tikzpicture}[scale=0.2,baseline=-5]
		\coordinate (root) at (0,0);
		\coordinate (tri) at (0,-2) ;
		\coordinate (t1) at (-2,2);
		\coordinate (t2) at (2,2);
		\coordinate (t3) at (0,2);
		\coordinate (t4) at (0,4);
		\coordinate (t41) at (-2,6);
		\coordinate (t42) at (2,6);
		\coordinate (t43) at (0,8);
		\draw[kernels2,tinydots] (t1) -- (root);
		\draw[kernels2] (t2) -- (root);
		\draw[kernels2] (t3) -- (root);
		\draw[symbols] (root) -- (tri);
		\draw[symbols] (t3) -- (t4);
		\draw[kernels2,tinydots] (t4) -- (t41);
		\draw[kernels2] (t4) -- (t42);
		\draw[kernels2] (t4) -- (t43);
		\node[not] (rootnode) at (root) {};
		\node[not] (rootnode) at (t4) {};
		\node[not] (rootnode) at (t3) {};
		\node[not,label= {[label distance=-0.2em]below: \scriptsize  $  $} ] (trinode) at (tri) {};
		\node[var] (rootnode) at (t1) {\tiny{$ k_{\tiny{4}} $}};
		\node[var] (rootnode) at (t41) {\tiny{$ k_{\tiny{1}} $}};
		\node[var] (rootnode) at (t42) {\tiny{$ k_{\tiny{3}} $}};
		\node[var] (rootnode) at (t43) {\tiny{$ k_{\tiny{2}} $}};
		\node[var] (trinode) at (t2) {\tiny{$ k_5 $}};
	\end{tikzpicture} \otimes \one 
	+ \one \otimes \begin{tikzpicture}[scale=0.2,baseline=-5]
		\coordinate (root) at (0,0);
		\coordinate (tri) at (0,-2);
		\coordinate (t1) at (-2,2);
		\coordinate (t2) at (2,2);
		\coordinate (t3) at (0,2);
		\coordinate (t4) at (0,4);
		\coordinate (t41) at (-2,6);
		\coordinate (t42) at (2,6);
		\coordinate (t43) at (0,8);
		\draw[kernels2,tinydots] (t1) -- (root);
		\draw[kernels2] (t2) -- (root);
		\draw[kernels2] (t3) -- (root);
		\draw[symbols] (root) -- (tri);
		\draw[symbols] (t3) -- (t4);
		\draw[kernels2,tinydots] (t4) -- (t41);
		\draw[kernels2] (t4) -- (t42);
		\draw[kernels2] (t4) -- (t43);
		\node[not] (rootnode) at (root) {};
		\node[not] (rootnode) at (t4) {};
		\node[not] (rootnode) at (t3) {};
		\node[not,label= {[label distance=-0.2em]below: \scriptsize  $  $}] (trinode) at (tri) {};
		\node[var] (rootnode) at (t1) {\tiny{$ k_{\tiny{4}} $}};
		\node[var] (rootnode) at (t41) {\tiny{$ k_{\tiny{1}} $}};
		\node[var] (rootnode) at (t42) {\tiny{$ k_{\tiny{3}} $}};
		\node[var] (rootnode) at (t43) {\tiny{$ k_{\tiny{2}} $}};
		\node[var] (trinode) at (t2) {\tiny{$ k_5 $}};
	\end{tikzpicture}  + \begin{tikzpicture}[scale=0.2,baseline=-5]
	\coordinate (root) at (0,0);
	\coordinate (tri) at (0,-2);
	\coordinate (t1) at (-2,2);
	\coordinate (t2) at (2,2);
	\coordinate (t3) at (0,3);
	\draw[kernels2,tinydots] (t1) -- (root);
	\draw[kernels2] (t2) -- (root);
	\draw[kernels2] (t3) -- (root);
	\draw[symbols] (root) -- (tri);
	\node[not] (rootnode) at (root) {};t
	\node[not,label= {[label distance=-0.2em]below: \scriptsize  $ $}] (trinode) at (tri) {};
	\node[var] (rootnode) at (t1) {\tiny{$ k_{\tiny{1}} $}};
	\node[var] (rootnode) at (t3) {\tiny{$ k_{\tiny{2}} $}};
	\node[var] (trinode) at (t2) {\tiny{$ k_3 $}};
\end{tikzpicture} \otimes
	\begin{tikzpicture}[scale=0.2,baseline=-5]
		\coordinate (root) at (0,0);
		\coordinate (tri) at (0,-2);
		\coordinate (t1) at (-2,2);
		\coordinate (t2) at (2,2);
		\coordinate (t3) at (0,3);
		\draw[kernels2,tinydots] (t1) -- (root);
		\draw[kernels2] (t2) -- (root);
		\draw[kernels2] (t3) -- (root);
		\draw[symbols] (root) -- (tri);
		\node[not] (rootnode) at (root) {};t
		\node[not,label= {[label distance=-0.2em]below: \scriptsize  $  $}] (trinode) at (tri) {};
		\node[var] (rootnode) at (t1) {\tiny{$ k_{\tiny{4}} $}};
		\node[var] (rootnode) at (t3) {\tiny{$ \ell_1 $}};
		\node[var] (trinode) at (t2) {\tiny{$ k_5 $}};
	\end{tikzpicture}  
\end{equs}
where $ \ell_1 = -k_1 + k_2 + k_3 $. For the first line, one has only one blue edge and therefore only two possibilities: Cutting this edge or the empty cut. The second tree has two blue edges that are on the same path to the root, one cannot cut these edges at the same time. We need to introduce another product dual to the cut that is the grafting product $ \curvearrowright $ defined as
\begin{equation}
	\label{grafting_a}
	\sigma \curvearrowright \tau:=\sum_{v\in  L_{\tau} } \sigma \curvearrowright_v  \tau,
\end{equation}
where  $\sigma $ and $\tau$ are two decorated trees, $ L_{\tau} $ is the set of leaves of $ \tau $ and $\sigma \curvearrowright_v \tau$ is obtained by merging  the root of $\sigma$ with the leaf $v$ of $\tau$. The new decorated trees obtained have to satisfy \eqref{frequencies_identity} otherwise the grafting is set to be zero. We extend $ \curvearrowright $ to the empty tree $ \mathbf{1} $ by setting
\begin{equs}
	\sigma \curvearrowright \mathbf{1} = \mathbf{1}  \curvearrowright \sigma.
\end{equs}
Here, we assume that each blue edge not related to the root  is connected to an edge of type $ \mathfrak{t}_1 $. Moreover, one has only one blue edge at each node. This means that they appear under the form $ \CI_{o_1}(\lambda_k \mathcal{I}_{o_2}(\lambda_k \cdot)) $. We make the assumption that all the leaves are of the form $ \mathcal{I}_{o_1}(\lambda_k) $.
We will need to use the adjoint  map of the grafting denoted by $ \curvearrowright^* $ and given by
\begin{equs}
	\left\langle 	\curvearrowright^* \tau, \tau_1 \otimes \tau_2 \right\rangle  := \left\langle 	 \tau, \tau_1 \curvearrowright \tau_2 \right\rangle
\end{equs}
where the inner product is defined by
\begin{equs}
	\left\langle  \sigma, \tau \right\rangle = \delta_{\sigma, \tau} S(\tau).
\end{equs}
Here, $ \delta_{\sigma, \tau} $ is equal to one only if $ \sigma = \tau $. Otherwise, it is equal to zero. One can provide a recursive definition of the adjoint map $ 	\curvearrowright^* $. It is defined as the same as $ \Delta_{\text{\tiny{BCK}}} $ for \eqref{BCK_new}. But the multiplicativity \eqref{mult_Delta} is replaced by
\begin{equs}
	\curvearrowright^* F_1 \cdot F_2 = 
\curvearrowright^* F_1 ( F_2 \otimes \one )  +	(F_1 \otimes \one ) \curvearrowright^* F_2.
\end{equs}
 For example, one has
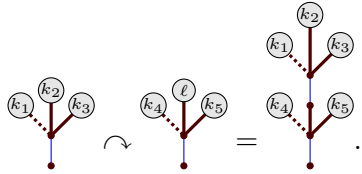
\begin{equs}
\begin{tikzpicture}[scale=0.2,baseline=-5]
	\coordinate (root) at (0,0);
	\coordinate (tri) at (0,-2);
	\coordinate (t1) at (-2,2);
	\coordinate (t2) at (2,2);
	\coordinate (t3) at (0,3);
	\draw[kernels2,tinydots] (t1) -- (root);
	\draw[kernels2] (t2) -- (root);
	\draw[kernels2] (t3) -- (root);
	\draw[symbols] (root) -- (tri);
	\node[not] (rootnode) at (root) {};t
	\node[not,label= {[label distance=-0.2em]below: \scriptsize  $ $}] (trinode) at (tri) {};
	\node[var] (rootnode) at (t1) {\tiny{$ k_{\tiny{1}} $}};
	\node[var] (rootnode) at (t3) {\tiny{$ k_{\tiny{2}} $}};
	\node[var] (trinode) at (t2) {\tiny{$ k_3 $}};
\end{tikzpicture}	\curvearrowright	\begin{tikzpicture}[scale=0.2,baseline=-5]
		\coordinate (root) at (0,0);
		\coordinate (tri) at (0,-2);
		\coordinate (t1) at (-2,2);
		\coordinate (t2) at (2,2);
		\coordinate (t3) at (0,3);
		\draw[kernels2,tinydots] (t1) -- (root);
		\draw[kernels2] (t2) -- (root);
		\draw[kernels2] (t3) -- (root);
		\draw[symbols] (root) -- (tri);
		\node[not] (rootnode) at (root) {};t
		\node[not,label= {[label distance=-0.2em]below: \scriptsize  $  $}] (trinode) at (tri) {};
		\node[var] (rootnode) at (t1) {\tiny{$ k_{\tiny{4}} $}};
		\node[var] (rootnode) at (t3) {\tiny{$ \ell $}};
		\node[var] (trinode) at (t2) {\tiny{$ k_5 $}};
	\end{tikzpicture}  = \begin{tikzpicture}[scale=0.2,baseline=-5]
	\coordinate (root) at (0,0);
	\coordinate (tri) at (0,-2);
	\coordinate (t1) at (-2,2);
	\coordinate (t2) at (2,2);
	\coordinate (t3) at (0,2);
	\coordinate (t4) at (0,4);
	\coordinate (t41) at (-2,6);
	\coordinate (t42) at (2,6);
	\coordinate (t43) at (0,8);
	\draw[kernels2,tinydots] (t1) -- (root);
	\draw[kernels2] (t2) -- (root);
	\draw[kernels2] (t3) -- (root);
	\draw[symbols] (root) -- (tri);
	\draw[symbols] (t3) -- (t4);
	\draw[kernels2,tinydots] (t4) -- (t41);
	\draw[kernels2] (t4) -- (t42);
	\draw[kernels2] (t4) -- (t43);
	\node[not] (rootnode) at (root) {};
	\node[not] (rootnode) at (t4) {};
	\node[not] (rootnode) at (t3) {};
	\node[not,label= {[label distance=-0.2em]below: \scriptsize  $  $}] (trinode) at (tri) {};
	\node[var] (rootnode) at (t1) {\tiny{$ k_{\tiny{4}} $}};
	\node[var] (rootnode) at (t41) {\tiny{$ k_{\tiny{1}} $}};
	\node[var] (rootnode) at (t42) {\tiny{$ k_{\tiny{3}} $}};
	\node[var] (rootnode) at (t43) {\tiny{$ k_{\tiny{2}} $}};
	\node[var] (trinode) at (t2) {\tiny{$ k_5 $}};
\end{tikzpicture}.
\end{equs}
There was one leaf for which the condition \eqref{frequencies_identity} is satisfied. Therefore, one only gets one term.

 We introduce a specific space of words that will be very convenient for encoding the various integration by parts performed in the normal form derivation. 
  We consider $A$ the alphabet whose letters are given by decorated trees of the form:
  \begin{equs}
  \CI_{(\mathfrak{t}_2,0)}( \lambda_{\ell}	\prod_{j=1}^n \CI_{(\Labhom_1,a_j)}( \lambda_{\ell_j} )  ), \quad  \CI_{(\mathfrak{t}_2,1)}( \lambda_\ell	\prod_{j=1}^n \CI_{(\Labhom_1,\bar{a}_j)}( \lambda_{\ell_j} )  ).
  	\end{equs}
  	where the $\ell_i$ are linear combinations of the $k_i$ with coefficients in $ \lbrace-1, 0, 1 \rbrace $.
  	The $a_j$ are the ones that appear in equation \eqref{dis}.
For NLS, one gets the following letters
  
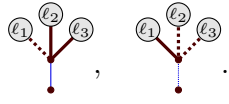
\begin{equs}
  	\begin{tikzpicture}[scale=0.2,baseline=-5]
  		\coordinate (root) at (0,0);
  		\coordinate (tri) at (0,-2);
  		\coordinate (t1) at (-2,2);
  		\coordinate (t2) at (2,2);
  		\coordinate (t3) at (0,3);
  		\draw[kernels2,tinydots] (t1) -- (root);
  		\draw[kernels2] (t2) -- (root);
  		\draw[kernels2] (t3) -- (root);
  		\draw[symbols] (root) -- (tri);
  		\node[not] (rootnode) at (root) {};t
  		\node[not,label= {[label distance=-0.2em]below: \scriptsize  $ $}] (trinode) at (tri) {};
  		\node[var] (rootnode) at (t1) {\tiny{$ \ell_{\tiny{1}} $}};
  		\node[var] (rootnode) at (t3) {\tiny{$ \ell_{\tiny{2}} $}};
  		\node[var] (trinode) at (t2) {\tiny{$ \ell_3 $}};
  	\end{tikzpicture}, \quad \begin{tikzpicture}[scale=0.2,baseline=-5]
  	\coordinate (root) at (0,0);
  	\coordinate (tri) at (0,-2);
  	\coordinate (t1) at (-2,2);
  	\coordinate (t2) at (2,2);
  	\coordinate (t3) at (0,3);
  	\draw[kernels2] (t1) -- (root);
  	\draw[kernels2,tinydots] (t2) -- (root);
  	\draw[kernels2,tinydots] (t3) -- (root);
  	\draw[symbols,tinydots] (root) -- (tri);
  	\node[not] (rootnode) at (root) {};t
  	\node[not,label= {[label distance=-0.2em]below: \scriptsize  $ $}] (trinode) at (tri) {};
  	\node[var] (rootnode) at (t1) {\tiny{$ \ell_{\tiny{1}} $}};
  	\node[var] (rootnode) at (t3) {\tiny{$ \ell_{\tiny{2}} $}};
  	\node[var] (trinode) at (t2) {\tiny{$ \ell_3 $}};
  \end{tikzpicture}.
  \end{equs}
 We set $T(A)$ to be the linear span of the words on this alphabet. The empty word is denoted by $ \varepsilon  $. The length of a word $w$, its number of letters, is denoted by $|w|$. We define on $ T(A)$ the shuffle product $ \shuffle $ given by
\begin{equs}
	\varepsilon \shuffle v=v\shuffle \varepsilon  =v, \quad (au\shuffle bv) = a(u\shuffle bv) + b(au\shuffle v)
\end{equs}
for all $u,v\in T(A)$ and $a,b\in A$. Given a word $ u= u_n...u_1 $ with $ u_i \in A $, we set for $j \leq n$
\begin{equs}
	u_{[j]} = u_j...u_1.
\end{equs}
 We define the arborification map $ \mathfrak{a} : \CH_2 \rightarrow T(A) $, for $\tau \in \CT$ as 
\begin{equs} \label{arbo_new}		
		\mathfrak{a}(  \tau )
	= \mathcal{M}_{\tiny{\text{c}}}\left(  P_{A} \otimes \mathfrak{a}   \right) \curvearrowright^{*} \tau.
\end{equs}
where $ P_A $ is the projections on the letters of $A$.
Then, we set for two forests $F_1, F_2 \in \CH_2$
\begin{equs}
	\mathfrak{a}( F_1 \cdot F_2 ) = \mathfrak{a}(F_1) \shuffle \mathfrak{a}(F_2).
\end{equs}
We define $\mathfrak{a}_n$ to be the restriction of $\mathfrak{a}$ to decorated forest with $n$ nodes. One has
\begin{equs}
		\mathfrak{a}_n(  \tau )
	= \mathcal{M}_{\tiny{\text{c}}}\left(  P_{A} \otimes \mathfrak{a}_{n-1}   \right) \curvearrowright^{*} \tau.
	\end{equs}
and $ 	\mathfrak{a}_n(  \tau ) = 0 $ if $\tau$ have a number of nodes different from $n$.
In the next proposition, we prove a co-associativity identity between $\curvearrowright^*$ and $\Delta_{\text{\tiny{BCK}}}$. 
\begin{proposition} \label{co-asso}
	One has
	\begin{equs}
		\left(\id \otimes P_A \otimes \id \right)  	\left( \Delta_{\normalfont\text{\tiny{BCK}}} \otimes \, \id    \right)		\Delta_{\normalfont\text{\tiny{BCK}}} = 		\left(\id \otimes P_A \otimes \id \right)  \left( \id \,  \otimes \curvearrowright^{*}    \right)		\Delta_{\normalfont \text{\tiny{BCK}}}.
	\end{equs}
\end{proposition}
\begin{proof} We proceed by induction on the size of the forest. One has
	\begin{equs}
		\left( \Delta_{\normalfont\text{\tiny{BCK}}} \otimes \, \id    \right)		\Delta_{\text{\tiny{BCK}}} \CI_{o_2}( \lambda_{k}  F )  & = \left( \Delta_{\normalfont\text{\tiny{BCK}}} \, \otimes \CI_{o_2}( \lambda_{k}  \cdot ) \right) \Delta_{\text{\tiny{BCK}}} F +  \Delta_{\normalfont\text{\tiny{BCK}}} \CI_{o_2}( \lambda_{k}  F ) \otimes \one.
		\\ & = \left( \id \, \otimes \id \, \otimes \CI_{o_2}( \lambda_{k}  \cdot ) \right) 	\left( \Delta_{\normalfont\text{\tiny{BCK}}} \otimes \, \id    \right)	 \Delta_{\text{\tiny{BCK}}} F \\& +  
		\left( \id \, \otimes \CI_{o_2}( \lambda_{k}  \cdot ) \right) \Delta_{\normalfont\text{\tiny{BCK}}} F \otimes \one
		+ \CI_{o_2}( \lambda_{k}  F ) \otimes \one \otimes \one.
	\end{equs}
	Then, 
	\begin{equs}
	& 	\left(\id \otimes P_A \otimes \id \right)	\left( \Delta_{\normalfont\text{\tiny{BCK}}} \otimes \, \id    \right)		\Delta_{\text{\tiny{BCK}}} \CI_{o_2}( \lambda_{k}  F ) 	\\ & = \left( \id \, \otimes \id \, \otimes \CI_{o_2}( \lambda_{k}  \cdot ) \right) \left(\id \otimes P_A \otimes \id \right)	\left( \Delta_{\normalfont\text{\tiny{BCK}}} \otimes \, \id    \right)	 \Delta_{\text{\tiny{BCK}}} F \\&  
	+ \left( \id \otimes P_A \otimes \id  \right) \left( \id \, \otimes \CI_{o_2}( \lambda_{k}  \cdot ) \right) \Delta_{\normalfont\text{\tiny{BCK}}} F \otimes \one.
	\end{equs}
	On the other hand
	\begin{equs}
		\left(\id \, \otimes \curvearrowright^{*}     \right)		\Delta_{\text{\tiny{BCK}}} \CI_{o_2}( \lambda_{k}  F )  & = \left( \id \, \otimes \curvearrowright^{*} \CI_{o_2}( \lambda_{k}  \cdot ) \right) \Delta_{\text{\tiny{BCK}}} F +  \ \CI_{o_2}( \lambda_{k}  F ) \otimes \one \otimes \one.
		\\ & = \left( \id \, \otimes \id \, \otimes \CI_{o_2}( \lambda_{k}  \cdot ) \right) 	\left( \id \, \otimes  \curvearrowright^{*}    \right)	 \Delta_{\text{\tiny{BCK}}} F \\& +  
		\left( \id \, \otimes \CI_{o_2}( \lambda_{k}  \cdot) \right) \Delta_{\text{\tiny{BCK}}} F \otimes \one
		+ \CI_{o_2}( \lambda_{k}  F ) \otimes \one \otimes \one.
	\end{equs}
	Then, 
	\begin{equs}
			& \left(\id \otimes P_A \otimes \id \right)	\left(\id \, \otimes \curvearrowright^{*}     \right)		\Delta_{\text{\tiny{BCK}}} \CI_{o_2}( \lambda_{k}  F )  
		\\ & = \left( \id \, \otimes \id \, \otimes \CI_{o_2}( \lambda_{k}  \cdot ) \right) \left(\id \otimes P_A \otimes \id \right)	\left( \id \, \otimes  \curvearrowright^{*}    \right)	 \Delta_{\text{\tiny{BCK}}} F \\& +  
	\left( \id \otimes P_A \otimes \id  \right) \left( \id \, \otimes \CI_{o_2}( \lambda_{k}  \cdot ) \right) \Delta_{\normalfont\text{\tiny{BCK}}} F \otimes \one.
	\end{equs}
	We conclude by applying the induction hypothesis on $F$. The proof works the same for $\mathcal{I}_{o_1}(\lambda_k F)$. For a product of two forests, one has 
	\begin{equs}
		\left( \id  \otimes \, \curvearrowright^{*}    \right)		\Delta_{\text{\tiny{BCK}}} (F_1 \cdot F_2) &= \left( \left( \id  \otimes \, \curvearrowright^{*}   \right)		\Delta_{\text{\tiny{BCK}}} F_1 \right) \Delta_{\text{\tiny{BCK}}} F_2 \\ & + \Delta_{\text{\tiny{BCK}}} F_2 \left( \left( \id  \otimes \, \curvearrowright^{*}    \right)		\Delta_{\text{\tiny{BCK}}} F_1 \right) 
	\end{equs}
	where we have used the fact that $ \curvearrowright^{*} $ is a derivation. The map $ \left( \Delta_{\normalfont\text{\tiny{BCK}}} \otimes \, \id    \right)		\Delta_{\normalfont\text{\tiny{BCK}}} $ is multiplicative and therefore not a derivation. But in fact, $ \left(\id \otimes P_A \otimes \id \right)  	\left( \Delta_{\normalfont\text{\tiny{BCK}}} \otimes \, \id    \right)		\Delta_{\normalfont\text{\tiny{BCK}}} $ is a derivation due to the projection $P_A$.  Then, one applies the induction hypothesis on $F_1$ and $F_2$.
\end{proof}

Let $ T  $ given by 
\begin{equs} \label{T_dec}
	T =  \CI_{(\mathfrak{t}_2,a)}( \lambda_{\ell}	\prod_{j=1}^n \CI_{(\Labhom_1,a_j)}( \lambda_{\ell_j} T_j) )
\end{equs}
where the $T_j $ are rooted trees and the $\ell_j$ are disjoint.
One has the following decomposition 
\begin{equs} \label{star_decomp}
	T =  \prod_{j=1}^m  T_j  \star T_r
\end{equs}
where $ \star $ is the simultaneous grafting and 
\begin{equs}
	T_r =  \CI_{(\mathfrak{t}_2,a)}( \lambda_{\ell}	\prod_{j=1}^n \CI_{(\Labhom_1,a_j)}( \lambda_{\ell_j} ) ).
\end{equs}
The product $\star$ grafts simultaneously the decorated trees $T_j$ onto the decorated tree $T_r$. This means that one identifies at the same time the distinct roots of the $T_j$ with distinct leaves of $T_r$. The fact that the $ \ell_j $ are disjoint makes the decomposition \eqref{star_decomp} unique.  
We provide an example below:
\begin{equation*}
	\begin{aligned}
 \begin{tikzpicture}[scale=0.2,baseline=-5]
	\coordinate (root) at (0,0);
	\coordinate (tri) at (0,-2);
	\coordinate (t1) at (-2,2);
	\coordinate (t2) at (2,2);
	\coordinate (t3) at (0,3);
	\draw[kernels2] (t1) -- (root);
	\draw[kernels2,tinydots] (t2) -- (root);
	\draw[kernels2,tinydots] (t3) -- (root);
	\draw[symbols,tinydots] (root) -- (tri);
	\node[not] (rootnode) at (root) {};t
	\node[not,label= {[label distance=-0.2em]below: \scriptsize  $ $}] (trinode) at (tri) {};
	\node[var] (rootnode) at (t1) {\tiny{$ k_{\tiny{1}} $}};
	\node[var] (rootnode) at (t3) {\tiny{$ k_{\tiny{2}} $}};
	\node[var] (trinode) at (t2) {\tiny{$ k_3 $}};
\end{tikzpicture} \, \, \begin{tikzpicture}[scale=0.2,baseline=-5]
\coordinate (root) at (0,0);
\coordinate (tri) at (0,-2);
\coordinate (t1) at (-2,2);
\coordinate (t2) at (2,2);
\coordinate (t3) at (0,3);
\draw[kernels2,tinydots] (t1) -- (root);
\draw[kernels2] (t2) -- (root);
\draw[kernels2] (t3) -- (root);
\draw[symbols] (root) -- (tri);
\node[not] (rootnode) at (root) {};t
\node[not,label= {[label distance=-0.2em]below: \scriptsize  $ $}] (trinode) at (tri) {};
\node[var] (rootnode) at (t1) {\tiny{$ k_{\tiny{5}} $}};
\node[var] (rootnode) at (t3) {\tiny{$ k_{\tiny{6}} $}};
\node[var] (trinode) at (t2) {\tiny{$ k_7 $}};
\end{tikzpicture} 	\star	\begin{tikzpicture}[scale=0.2,baseline=-5]
			\coordinate (root) at (0,0);
			\coordinate (tri) at (0,-2);
			\coordinate (t1) at (-2,2);
			\coordinate (t2) at (2,2);
			\coordinate (t3) at (0,3);
			\draw[kernels2,tinydots] (t1) -- (root);
			\draw[kernels2] (t2) -- (root);
			\draw[kernels2] (t3) -- (root);
			\draw[symbols] (root) -- (tri);
			\node[not] (rootnode) at (root) {};t
			\node[not,label= {[label distance=-0.2em]below: \scriptsize  $ $}] (trinode) at (tri) {};
			\node[var] (rootnode) at (t1) {\tiny{$ \ell_{\tiny{1}} $}};
			\node[var] (rootnode) at (t3) {\tiny{$ k_{\tiny{2}} $}};
			\node[var] (trinode) at (t2) {\tiny{$ \ell_2 $}};
		\end{tikzpicture}	= \begin{tikzpicture}[scale=0.2,baseline=-5]
			\coordinate (root) at (0,0);
			\coordinate (tri) at (0,-2);
			\coordinate (t1) at (-2,2);
			\coordinate (t2) at (2,2);
			\coordinate (t3) at (0,3);
			\coordinate (t4) at (4,4);
			\coordinate (t41) at (2,6);
			\coordinate (t42) at (6,6);
			\coordinate (t43) at (4,8);
			\coordinate (t4l) at (-4,4);
			\coordinate (t41l) at (-2,6);
			\coordinate (t42l) at (-6,6);
			\coordinate (t43l) at (-4,8);
			\draw[kernels2,tinydots] (t1) -- (root);
			\draw[kernels2] (t2) -- (root);
			\draw[kernels2] (t3) -- (root);
			\draw[symbols] (root) -- (tri);
			\draw[symbols] (t2) -- (t4);
			\draw[kernels2,tinydots] (t4) -- (t41);
			\draw[kernels2] (t4) -- (t42);
			\draw[kernels2] (t4) -- (t43);\draw[symbols,tinydots] (t1) -- (t4l);
			\draw[kernels2,tinydots] (t4l) -- (t41l);
			\draw[kernels2] (t4l) -- (t42l);
			\draw[kernels2,tinydots] (t4l) -- (t43l);
			
			\node[not] (rootnode) at (root) {};
			\node[not] (rootnode) at (t4) {};
			\node[var] (rootnode) at (t3) {\tiny{$ k_{\tiny{4}} $}};
			\node[not,label= {[label distance=-0.2em]below: \scriptsize  $  $}] (trinode) at (tri) {};
			\node[not] (rootnode) at (t1) {};
			\node[var] (rootnode) at (t41) {\tiny{$ k_{\tiny{5}} $}};
			\node[var] (rootnode) at (t42) {\tiny{$ k_{\tiny{7}} $}};
			\node[var] (rootnode) at (t43) {\tiny{$ k_{\tiny{6}} $}};
			\node[var] (rootnode) at (t41l) {\tiny{$ k_{\tiny{3}} $}};
			\node[var] (rootnode) at (t42l) {\tiny{$ k_{\tiny{1}} $}};
			\node[var] (rootnode) at (t43l) {\tiny{$ k_{\tiny{2}} $}};
			\node[not] (trinode) at (t2) {};
		\end{tikzpicture} 
	\end{aligned}
\end{equation*}
where $ - \ell_1 = k_1 -k_2 - k_3 $ and $ \ell_2 = -k_5 + k_6 + k_7 $.
In the sequel, we make the following  abuse of notation for $T \in A$, $$ |\nabla|^\alpha(T) :=  |\nabla|^\alpha(k)  $$ where $ k$ is the decoration of the node of $T$ connected to its root.

\begin{proposition} \label{int_decomp}
	For every decorated tree $T$ given by \eqref{T_dec}, one has
	\begin{equs}
		(\Pi T)(t) = - i |\nabla|^\alpha(T_r)\int_0^t e^{i s \mathscr{F}(T_r) } \prod_{j=1}^m (\Pi T_j)(s) ds.
	\end{equs}
	\end{proposition}
	\begin{proof} It is just a direct application of the identity \eqref{def_Pi} and the Definition \ref{dom_freq}.
		\end{proof}

\section{Poincaré-Dulac normal form for NLS and discretisation}

\label{Sec::4}
We start by recalling the normal form from \cite{GKO13} providing an explicit expression for the first terms of this decomposition. Then, we introduce the main ideas for the low regularity discretisation. All the derivations are made for the first terms on NLS to illustrate the general formula given in the next section. 

We recall 
the following cubic nonlinear Schr\"odinger equation (NLS)  on the $d$ dimensional torus $\mathbb{T}^d$:
\begin{align}
	\begin{cases}
		i \partial_t u + \Delta  u  = | u |^{2} u \\
		u |_{t = 0} = u_0,
	\end{cases}
	\quad (x, t) \in \mathbb{T}^d \times \mathbb{R}.
	\label{NLS1}
\end{align}
We rewrite the equation using twisting variables:
\begin{equs}
	v(t) =  e^{-it \Delta} {u}(t).
\end{equs}
On the Fourier side, we have
$v_k(t) = e^{i t k^2} u_k(t)$ and
\begin{align}
	\begin{split}
		\partial_t v_k 
		& =  - i 
		\sum_{\substack{k = -k_1 +k_2 + k_3} }
		e^{ i t \Phi(\bar{k})} 
		\bar{v}_{k_1} v_{k_2} v_{k_3}.
	\end{split}
	\label{NLS4}
\end{align}
Here,  the  phase function $\Phi(\bar{k})$ is defined by 
\begin{align}
	\Phi(\bar{k}):& = \Phi(k, k_1, k_2, k_3) = k^2 + k_1^2 - k_2^2- k_3^2 
\end{align}
with $k=-k_1 + k_2 + k_3$. The main idea of the Poincaré-Dulac normal form is to consider the case when the phase $ \Phi(\bar{k}) $ is non-zero and to rewrite the oscillatory factor as
\begin{equs}
	e^{ i t\Phi(\bar{k})}  = \partial_t \left(  \frac{e^{  i t \Phi(\bar{k}) } }{i \Phi(\bar{k})} \right)
\end{equs}
and then to proceed with an integration by parts
\begin{equs}
	e^{ i t \Phi(\bar{k}) } 
	\bar{v}_{k_1} v_{k_2} v_{k_3} & = \partial_t \left(  \frac{e^{ i t \Phi(\bar{k}) } }{i \Phi(\bar{k})} \right) \bar{v}_{k_1} v_{k_2} v_{k_3}
	\\ & = \partial_t \left(  \frac{e^{ i t \Phi(\bar{k}) } }{i \Phi(\bar{k})}  \bar{v}_{k_1} v_{k_2} v_{k_3}  \right) -
	\frac{e^{ i t \Phi(\bar{k}) } }{i \Phi(\bar{k})}  (\partial_t \bar{v}_{k_1}) v_{k_2} v_{k_3} \\ &  -    \frac{e^{ i t \Phi(\bar{k}) } }{i \Phi(\bar{k})}   \bar{v}_{k_1} (\partial_t v_{k_2}) v_{k_3} - \frac{e^{ i t \Phi(\bar{k}) } }{i \Phi(\bar{k})}   \bar{v}_{k_1}  v_{k_2} (\partial_t v_{k_3}).
\end{equs}
One substitutes the $ \partial_t \bar{v}_{k_1}, \partial_t v_{k_2}, \partial_t v_{k_3}  $ by the right-hand side of \eqref{NLS4}. For $ \partial_t v_{k_2}$, one gets
\begin{equs}
	-	\frac{e^{ i t \Phi(\bar{k}) } }{i \Phi(\bar{k})}   \bar{v}_{k_1} (\partial_t v_{k_2}) v_{k_3} =  
	\sum_{\substack{k_2 = -k_4 +k_5 + k_6} } \frac{1 }{ \Phi(\bar{k})}
	e^{ i t(\Phi(\bar{k})+ \Phi(\bar{k}_2)) } 
	\bar{v}_{k_1} v_{k_3}\bar{v}_{k_4} v_{k_5} v_{k_6}.
\end{equs}
Then, one repeats the integration by parts and divides by $  \Phi(\bar{k})+ \Phi(\bar{k}_2)$. This procedure can be encoded by the arborification map $\mathfrak{a}$ and a modification of a character $\tilde{\Psi}$ given for $w= T_{p} \cdots T_1 $ by:
\begin{equs} \label{character_main}
	\tilde{\Psi}\left( w \right)(t) = 	 \frac{e^{i t \sum_{j=1}^p \mathscr{F}(T_j)}	}{\prod_{m=1}^{p} \sum_{j=1}^{m} \mathscr{F}(T_j)}.
\end{equs}
By character, we mean that we have that for every words $ w, \bar{w} $ 
\begin{equs} \label{morphism_prop}
	\tilde{\Psi}\left( w \shuffle \bar{w} \right)(t) = \tilde{\Psi}\left( w  \right)(t) \tilde{\Psi}\left(  \bar{w} \right)(t).
\end{equs}
This character has been introduced in \cite[Sec. 3]{B24} and it is inspired from one character coming from mould calculus (see \cite[Lemme II.8]{Cr09} and \cite[Prop. 6]{FM}).
One cannot use this map for a low regularity scheme. Indeed, one cannot in general map back to physical space 
terms of the form $ \frac{1}{\prod_{m=1}^{p} \sum_{j=1}^{m} \mathscr{F}(T_j)}$ with usual differential operators. The main idea is to replace this factor with something that one can write easily in physical space. 
Before doing the integration by parts, we split the phase $\Phi(\bar{k})$ and get
\begin{equs}
	\Phi(\bar{k}) & = k^2 + k_1^2 - k_2^2- k_3^2 
	\\ &  = (-k_1 + k_2 + k_3)^2 + k_1^2 - k_2^2 - k_3^2
	\\ & = \Phi_{{\tiny{\text{dom}}}}(\bar{k}) + \Phi_{{\tiny{\text{low}}}}(\bar{k})
\end{equs}
where $ \Phi_{{\tiny{\text{\tiny{dom}}}}}(\bar{k}) $ (resp. $ \Phi_{{\tiny{\text{\tiny{low}}}}}(\bar{k}) $) is the dominant (resp. lower) part of $\Phi(\bar{k}) $. They are given by
\begin{equs}
	\Phi_{{\tiny{\text{dom}}}}(\bar{k}) = 	2 k_1^2, \quad \Phi_{{\tiny{\text{low}}}}(\bar{k}) =  - 2 k_1 (k_2 + k_3) + 2 k_2 k_3.
\end{equs}
One observes that the dominant part is of degree $2$ in $k_1$ whereas the lower part has no factor of degree $2$.
Then, one Taylor-expands the lower part before performing the integration by parts:
\begin{equs} \label{Taylor_low}
	\begin{aligned}
		e^{ i t \Phi(\bar{k}) } 
		\bar{v}_{k_1} v_{k_2} v_{k_3} & = 	e^{ i t \Phi_{{\tiny{\text{dom}}}}(\bar{k}) }   e^{ i t \Phi_{{\tiny{\text{low}}}}(\bar{k}) }
		\bar{v}_{k_1} v_{k_2} v_{k_3} 
		\\ & = \sum_{n =0}^r   \frac{t^n}{n!} i^n \Phi^n_{{\tiny{\text{low}}}}(\bar{k}) e^{ i t \Phi_{{\tiny{\text{dom}}}}(\bar{k}) } 
		\bar{v}_{k_1} v_{k_2} v_{k_3} \\ &  + \mathcal{O}( t^{r+1}  \Phi^{r+1}_{{\tiny{\text{low}}}}(\bar{k}) 
		\bar{v}_{k_1} v_{k_2} v_{k_3} ).
	\end{aligned}
\end{equs}
By doing this expansion, we have to ask some regularity on the solution that depends on the lower part via $\Phi^n_{{\tiny{\text{\tiny{low}}}}}(\bar{k}) $. It requires $r+1$ derivatives on the solution as $\Phi^{r+1}_{{\tiny{\text{\tiny{low}}}}}(\bar{k}) $ contains crossed terms ($k_1k_2, k_1 k_3, k_2 k_3$) to the power $r+1$. This is less that the regularity asked by the full Taylor expansion, namely:
\begin{equs}
	e^{ i t\Phi(\bar{k}) } 
	\bar{v}_{k_1} v_{k_2} v_{k_3}  = \sum_{n =0}^r   \frac{t^n}{n!} i^n \Phi^n(\bar{k})  
	\bar{v}_{k_1} v_{k_2} v_{k_3}  + \mathcal{O}( t^{r+1}  \Phi^{r+1}(\bar{k}) 
	\bar{v}_{k_1} v_{k_2} v_{k_3} ).
\end{equs}
Here, the error is governed by $\Phi^{r+1}(\bar{k}) $ that requires $2r+2$ derivatives on the solution.
Once, one has performed the Taylor expansion \eqref{Taylor_low}, we can proceed with the integration by parts:
\begin{equs}
	\, & \frac{t^n}{n!} i^n \Phi^n_{{\tiny{\text{low}}}}(\bar{k}) e^{t i \Phi_{{\tiny{\text{dom}}}}(\bar{k}) } 
	\bar{v}_{k_1} v_{k_2} v_{k_3}  =  \frac{t^n}{n!} i^n \Phi^n_{{\tiny{\text{low}}}}(\bar{k}) \partial_t \left(\frac{e^{ it \Phi_{{\tiny{\text{dom}}}}(\bar{k}) } }{i \Phi_{{\tiny{\text{dom}}}}(\bar{k})} \right)
	\bar{v}_{k_1} v_{k_2} v_{k_3}
	\\ &  = \partial_t \left( \frac{t^n}{n!} i^n \Phi^n_{{\tiny{\text{low}}}}(\bar{k}) \frac{e^{ i t \Phi_{{\tiny{\text{dom}}}}(\bar{k}) } }{i \Phi_{{\tiny{\text{dom}}}}(\bar{k})} 
	\bar{v}_{k_1} v_{k_2} v_{k_3} \right) - \frac{t^{n-1}}{(n-1)!} i^n \Phi^n_{{\tiny{\text{low}}}}(\bar{k}) \frac{e^{ i t \Phi_{{\tiny{\text{dom}}}}(\bar{k}) } }{i \Phi_{{\tiny{\text{dom}}}}(\bar{k})} 
	\bar{v}_{k_1} v_{k_2} v_{k_3} 
	\\ &  -  \frac{t^n}{n!} i^n \Phi^n_{{\tiny{\text{low}}}}(\bar{k}) \frac{e^{ i t \Phi_{{\tiny{\text{dom}}}}(\bar{k}) } }{i \Phi_{{\tiny{\text{dom}}}}(\bar{k})} \left( 
	\partial_t  \bar{v}_{k_1}   v_{k_2} v_{k_3}  +
	\bar{v}_{k_1}  \partial_t  v_{k_2}  v_{k_3} + \bar{v}_{k_1} v_{k_2} \partial_t  v_{k_3}  \right).
\end{equs} 
One can easily prove by recurrence the following proposition:
\begin{proposition} \label{prop_inte_parts}
	One has
	\begin{equs} \label{expression_integration} \begin{aligned}
			\, &	\frac{t^n}{n!} i^n \Phi^n_{{ \normalfont \tiny{\text{low}}}}(\bar{k}) e^{ it \Phi_{{ \normalfont\tiny{\text{dom}}}}(\bar{k}) } 
			\bar{v}_{k_1} v_{k_2} v_{k_3}   =  \sum_{m=0}^n 
			\frac{t^{n-m}}{(n-m)!} i^{n-m-1} \frac{\Phi^n_{{\normalfont \tiny{\text{low}}}}(\bar{k})}{\Phi^{m+1}_{{\normalfont \tiny{\text{dom}}}}(\bar{k})}
			\\ & \left(  \partial_t \left(  e^{ i \Phi_{{ \normalfont\tiny{\text{dom}}}}(\bar{k})t } 
			\bar{v}_{k_1} v_{k_2} v_{k_3} \right) -  e^{ i t \Phi_{{ \normalfont \tiny{\text{dom}}}}(\bar{k}) } 
			\partial_t \left( 	\bar{v}_{k_1} v_{k_2} v_{k_3} \right) \right).
		\end{aligned}
	\end{equs}
\end{proposition}
\begin{proof} We proceed by recurrence on $n$. For $n=0$, one has by Leibniz rule
	\begin{equs}
		e^{ i \Phi_{{ \normalfont\tiny{\text{dom}}}}(\bar{k})t } 
		\bar{v}_{k_1} v_{k_2} v_{k_3} & = \partial_t \left( 	\frac{e^{ i t \Phi_{{ \normalfont\tiny{\text{dom}}}}(\bar{k}) }}{i \Phi_{{ \normalfont\tiny{\text{dom}}}}(\bar{k})} \right)
		\bar{v}_{k_1} v_{k_2} v_{k_3}
		\\ & = \partial_t \left( 	\frac{e^{ it \Phi_{{ \normalfont\tiny{\text{dom}}}}(\bar{k}) }}{i \Phi_{{ \normalfont\tiny{\text{dom}}}}(\bar{k})} 
		\bar{v}_{k_1} v_{k_2} v_{k_3}\right) - 	\frac{e^{ i t\Phi_{{ \normalfont\tiny{\text{dom}}}}(\bar{k}) }}{i \Phi_{{ \normalfont\tiny{\text{dom}}}}(\bar{k})} 
	\partial_t \left( 	\bar{v}_{k_1} v_{k_2} v_{k_3}\right).
	\end{equs}
	We suppose the property true for $n \in \mathbb{N}$. Then, one has
	\begin{equs}
		&	\frac{t^{n+1}}{(n+1)!} i^{n+1} \Phi^{n+1}_{{ \normalfont \tiny{\text{low}}}}(\bar{k}) e^{ it \Phi_{{ \normalfont\tiny{\text{dom}}}}(\bar{k}) } 
		\bar{v}_{k_1} v_{k_2} v_{k_3} \\ & = 	\frac{t^{n+1}}{(n+1)!} i^{n+1} \Phi^{n+1}_{{ \normalfont \tiny{\text{low}}}}(\bar{k}) \partial_t \left( 	\frac{e^{ i t \Phi_{{ \normalfont\tiny{\text{dom}}}}(\bar{k}) }}{i \Phi_{{ \normalfont\tiny{\text{dom}}}}(\bar{k})} \right)
		\bar{v}_{k_1} v_{k_2} v_{k_3}
		\\ &  =   \frac{t^{n+1}}{(n+1)!} i^{n+1} \Phi^{n+1}_{{ \normalfont \tiny{\text{low}}}}(\bar{k}) \partial_t \left( 	\frac{e^{ i t \Phi_{{ \normalfont\tiny{\text{dom}}}}(\bar{k}) }}{i \Phi_{{ \normalfont\tiny{\text{dom}}}}(\bar{k})} 
		\bar{v}_{k_1} v_{k_2} v_{k_3} \right) \\ &  - \frac{t^{n+1}}{(n+1)!} i^{n+1} \Phi^{n+1}_{{ \normalfont \tiny{\text{low}}}}(\bar{k}) 	\frac{e^{ i t \Phi_{{ \normalfont\tiny{\text{dom}}}}(\bar{k}) }}{i \Phi_{{ \normalfont\tiny{\text{dom}}}}(\bar{k})} 
		\partial_t \left( 	\bar{v}_{k_1} v_{k_2} v_{k_3}\right)   
		\\ & - \frac{t^{n}}{n!} i^{n+1} \Phi^{n+1}_{{ \normalfont \tiny{\text{low}}}}(\bar{k}) 	\frac{e^{ i t \Phi_{{ \normalfont\tiny{\text{dom}}}}(\bar{k})}}{i \Phi_{{ \normalfont\tiny{\text{dom}}}}(\bar{k})} 
		\bar{v}_{k_1} v_{k_2} v_{k_3}.
	\end{equs}
	From the induction hypothesis, one has
	\begin{equs}
		 & - \frac{t^{n}}{n!} i^{n+1} \Phi^{n+1}_{{ \normalfont \tiny{\text{low}}}}(\bar{k}) 	\frac{e^{ i \Phi_{{ \normalfont\tiny{\text{dom}}}}(\bar{k})t }}{i \Phi_{{ \normalfont\tiny{\text{dom}}}}(\bar{k})} 
		\bar{v}_{k_1} v_{k_2} v_{k_3}  = \sum_{m=0}^n 
		\frac{t^{n-m}}{(n-m)!} i^{n-m} \frac{\Phi^{n+1}_{{\normalfont \tiny{\text{low}}}}(\bar{k})}{\Phi^{m+2}_{{\normalfont \tiny{\text{dom}}}}(\bar{k})}
		\\ & \left(  \partial_t \left(  e^{ i t \Phi_{{ \normalfont\tiny{\text{dom}}}}(\bar{k})} 
		\bar{v}_{k_1} v_{k_2} v_{k_3} \right) -  e^{ i t\Phi_{{ \normalfont \tiny{\text{dom}}}}(\bar{k})} 
		\partial_t \left( 	\bar{v}_{k_1} v_{k_2} v_{k_3} \right) \right).
		\end{equs}
		Then, in order to conclude, one performs the change of variable $ m' = m+1 $ in the previous term that gives
			\begin{equs}
	\sum_{m=0}^n 
	\frac{t^{n-m}}{(n-m)!} i^{n-m} \frac{\Phi^{n+1}_{{\normalfont \tiny{\text{low}}}}(\bar{k})}{\Phi^{m+2}_{{\normalfont \tiny{\text{dom}}}}(\bar{k})}	 = \sum_{m=1}^{n+1} 
			\frac{t^{n +1 -m}}{(n+1-m)!} i^{n +1-m} \frac{\Phi^{n+1}_{{\normalfont \tiny{\text{low}}}}(\bar{k})}{\Phi^{m+1}_{{\normalfont \tiny{\text{dom}}}}(\bar{k})}.
		\end{equs}
	\end{proof}
From \eqref{expression_integration}, one wants to iterate the procedure as for the Poincaré-Dulac normal form. Normally, one stops on the first term of the second line in \eqref{expression_integration} and continues with substitution in the second term of \eqref{expression_integration}.
In this way, one can produce an approximation of the Poincaré-Dulac normal form up to the order $ \mathcal{O}(t^{r+1}) $ and the regularity asked on the solution will be given by terms of the form $  \Phi^{r+1}_{{ \normalfont \tiny{\text{\tiny{low}}}}}(\bar{k}) $.
We refrain from writing precisely such a statement as we do not see any applications at the present time. 
Its proof should follow the same line as for the resonance scheme.

Let us explain why we cannot use directly this normal form approach. One has to discretise terms of the form
\begin{equs}
	\, & \int_0^t	 \partial_s \left(  e^{ i s \Phi_{{ \normalfont\tiny{\text{dom}}}}(\bar{k})} 
	\bar{v}_{k_1}(s) v_{k_2}(s) v_{k_3}(s) \right) ds \\ &= 
	e^{ i t\Phi_{{ \normalfont\tiny{\text{dom}}}}(\bar{k}) } 
	\bar{v}_{k_1}(t) v_{k_2}(t) v_{k_3}(t)  - 	\bar{v}_{k_1}(0) v_{k_2}(0) v_{k_3}(0).
\end{equs}
Then, one has to perform an approximation  of the
$ \bar{v}_{k_1}(t), v_{k_2}(t), v_{k_3}(t)  $. A Taylor expansion will require too much regularity.
Therefore, we will apply the normal form reduction directly on iterated integrals obtained by iterating Duhamel's formula on the initial data.

\section{Derivation of low regularity schemes}
\label{Sec::5}
In this section, we provide a normal from derivation of new low regularity schemes. We start by recalling a definition (see \cite[Def. 2.2]{BS}) of an important operator that allows us to compute dominant and lower parts of a Fourier operator.
\begin{definition} \label{proj_dom}
	Let $P(k_1, ..., k_m)$ a polynomial in the $k_i$. If the highest-degree
	monomials of $P$ are of the form
	\begin{equs}
		a \sum_{i=1}^m (a_i k_i)^p, \quad a_i \in \lbrace 0,1 \rbrace, \, a \in \mathbb{Z},
	\end{equs}
	then we define $ \mathcal{P}_{\!{\tiny \text{\tiny{dom}} }}(P) $ as
	\begin{equs}
		\mathcal{P}_{\!{\tiny \text{dom} }}(P)  = a \left( \sum_{i=1}^m a_i k_i \right)^p
	\end{equs}
	Otherwise, it is zero.
	\end{definition}

 We introduce lower and dominant parts associated with a word of $T(A)$ in the definition below
\begin{definition} \label{def_dom}
	Let $w   = T_{\ell}...T_1$, we set
	\begin{equs}
	\mathscr{F}_{\!{\tiny \text{dom} }}(w) & =  \mathcal{P}_{\!{\tiny \text{dom} }} \left( 	\mathscr{F}(T_{\ell}) + 	\mathscr{F}_{\!{\tiny \text{dom} }}(T_{\ell-1}...T_1) \right),
	\\ \mathscr{F}_{\!{\tiny \text{low} }}(w) & =  
	(\id -\mathcal{P}_{\!{\tiny \text{dom} }} ) \left( 	\mathscr{F}(T_{\ell}) + 	\mathscr{F}_{\!{\tiny \text{dom} }}(T_{\ell-1}...T_1) \right).
	\end{equs}
	\end{definition}

We need a refinement of the previous definition taking into account the regularity a priori assumed on the solution and the order of the monomial in time multiplying the oscillatory term. Before, we introduce the notion of the degree of a monomial in the $k_i$.
Given $P = \prod_{i} k_i^{m_i}$ to be
\begin{equs}
	\deg(P) = \max_{i} m_i.
\end{equs}
Then, if $P$ is a linear combination of monomials, the degree is the maximum of the degree of its monomials.
As an example, one has
\begin{equs}
	\deg(  k_1 k_2 ) = 1, \quad \deg(k_1^2) = 2. 
\end{equs}
Then, we introduce the following definition.
\begin{definition}
	\label{adaptative_splitting}
		Let $w   = T_{\ell}...T_1$, $ m = m_{\ell}...m_1$, $n, r, m_i \in\mathbb{N}$, one sets
	\begin{equs}
		\mathscr{F}^{n,r}_{\!{\tiny \text{dom} }}(w,m)  = 0,  
		\quad \mathscr{F}_{\!{\tiny \text{low} }}^{n,r}(w,m)  =  
	 	\mathscr{F}(T_\ell) + 	\mathscr{F}^{n,r}_{\!{\tiny \text{dom} }}(w_{[\ell-1]}, m_{[\ell-1]}).
	\end{equs}
	if 
	\begin{equs} \label{cond_split} (r+ \ell+1-m_{\ell})\deg(\mathscr{F}(T_\ell) + 	\mathscr{F}^{n,r}_{\!{\tiny \text{dom} }}(w_{[\ell-1]},m_{[\ell-1]})) + \ell \alpha
		\\ + \sum_{j=1}^{\ell-1} (r+j+1-m_{j})\deg( 	\mathscr{F}^{n,r}_{\!{\tiny \text{dom} }}(w_{[j-1]},m_{[j-1]})) \leq  n.
	\end{equs}
	 Otherwise, one has
\begin{equs}
	\mathscr{F}^{n,r}_{\!{\tiny \text{dom} }}(w,m) & =  \mathcal{P}_{\!{\tiny \text{dom} }} \left( 	\mathscr{F}(T_{\ell}) + 	\mathscr{F}^{n,r}_{\!{\tiny \text{dom} }}(w_{[\ell-1]}, m_{[\ell-1]}) \right),
	\\ \mathscr{F}^{n,r}_{\!{\tiny \text{low} }}(w,m) & =  
	(\id -\mathcal{P}_{\!{\tiny \text{dom} }} ) \left( 	\mathscr{F}(T_{\ell}) + 	\mathscr{F}^{n,r}_{\!{\tiny \text{dom} }}(w_{[\ell-1]}, m_{[\ell-1]})\right).
\end{equs}
	\end{definition}
	The main difference between Definitions \ref{def_dom} and \ref{adaptative_splitting}, is the condition \eqref{cond_split} that says that one has enough regularity and can Taylor-expand the full operator instead of doing a resonance analysis. One has a similar condition in \cite[Def. 3.1]{BS}.
	
	\begin{example} \label{ex_f}
		We illustrate the previous example on cubic NLS by considering the following phase:
			\begin{equs}
			T = \begin{tikzpicture}[scale=0.2,baseline=-5]
				\coordinate (root) at (0,0);
				\coordinate (tri) at (0,-2);
				\coordinate (t1) at (-2,2);
				\coordinate (t2) at (2,2);
				\coordinate (t3) at (0,3);
				\draw[kernels2,tinydots] (t1) -- (root);
				\draw[kernels2] (t2) -- (root);
				\draw[kernels2] (t3) -- (root);
				\draw[symbols] (root) -- (tri);
				\node[not] (rootnode) at (root) {};t
				\node[not,label= {[label distance=-0.2em]below: \scriptsize  $ $}] (trinode) at (tri) {};
				\node[var] (rootnode) at (t1) {\tiny{$ k_{\tiny{1}} $}};
				\node[var] (rootnode) at (t3) {\tiny{$ k_{\tiny{2}} $}};
				\node[var] (trinode) at (t2) {\tiny{$ k_3 $}};
			\end{tikzpicture}, \quad \mathscr{F}(T) = (-k_1+k_2+k_3)^2 + k_1^2 -k_2^2 - k_3^2.
		\end{equs}
		One has
		\begin{equs}
			\mathscr{F}^{2,1}_{\!{\tiny \text{dom} }}(T,0)  =  2 k_1^2,
			\quad \mathscr{F}^{2,1}_{\!{\tiny \text{low} }}(T,0)  =  
			-2 k_1(k_2 + k_3) + 2 k_2 k_3.
		\end{equs}
		which is due to the fact that
		\begin{equs}
			3 \deg(\mathscr{F}(T)) = 6 > 2.
			\end{equs}
		On the other hand, one has
		\begin{equs}
			\mathscr{F}^{6,1}_{\!{\tiny \text{dom} }}(T,0)  =  0,
			\quad \mathscr{F}^{6,1}_{\!{\tiny \text{low} }}(T,0)  =  \mathscr{F}(T)
		.
		\end{equs}
	\end{example}
	
We introduce below the basic blocks for defining low regularity schemes which are an extension of $ \tilde{\Psi} $ given in  \eqref{character_main}.

	\begin{definition} \label{def_Phi}
		For every $w \in T(A)$, $m, r \in \mathbb{N}$, $a \in \lbrace 0,1 \rbrace$, one sets
 \begin{equs} \label{Psi_m_ell}
 	\begin{aligned}
\Psi^{n,r}_{m,a}\left( w \right)(t) & = -	  \frac{t^{m}}{m!} \sum_{ m_{j+1}= \sum_{\ell \leq j } (p_{\ell}-q_{\ell})  } \,\prod_{j=1}^{|w|} i^{p_j + q_j-1}\binom{m_{j} + p_j}{p_j}   \\ & \times i |\nabla|^\alpha(T_j) \frac{\mathscr{F}_{\!{\tiny\tiny \text{low} }}^{n,r}(w_{[j]},m_{[j]})^{p_j}	}{  \mathscr{F}^{n,r}_{\!{\tiny\tiny \text{dom} }}(w_{[j]},m_{[j]})^{q_j+1}}      \left(    e^{i t  \mathscr{F}^{n,r}_{\!{\tiny\tiny \text{dom} }}(w,\hat{m})} - a \delta_{m,0} \right)
\end{aligned}
\end{equs}
where $  \hat{m} = m_{|w|}...m_1$ with $m_1 =0$, $m_{|w|+1} =m$  and
we recall the notation $ w_{[j]} = w_j...w_1 $. One must have 
\begin{equs} \label{bound_Taylor_expansion}
0 \leq q_j \leq m_j + p_j, \quad p_{j} = r+ j-1 -m_j.
\end{equs}
Also if the dominant part $\mathscr{F}^{n,r}_{\!{\tiny\tiny \text{\tiny{dom}} }}(w_{[j]},m_{[j]})$ is zero, we set $q_j = -1$. 
 Then, one also sets
 \begin{equs} \label{itermediate_psi}
	\Psi^{n,r}_{m} \left( w \right)(t) &=  i |\nabla|^\alpha(T_{|w|}) e^{i t  \mathscr{F}(T_{|w|})} \Psi^{n,r}_{m,0}\left( w_{[|w|-1]} \right)(t).
\end{equs}
If $ w $ has more than one letter otherwise
\begin{equs} \label{itermediate_psi_bis}
	\Psi^{n,r}_{m,0} \left( w \right)(t) &=  i |\nabla|^\alpha(w) \one_{m=0} e^{i t  \mathscr{F}(w)}.
\end{equs}
\end{definition}
Let us comment briefly on this definition. Each letter of $w$ induces a Taylor expansion of a lower part given by $ \mathscr{F}^{n,r}_{\!{\tiny\tiny \text{\tiny{low}} }}(w_{[j]},m_{[j]})^{p_j} $. Here, $ r+j-1-m_j$ is the length of the Taylor expansion at the $j$th step. Then, before moving to another resonance analysis, one performs several integrations by parts, which lower  the degree of the monomial obtained from the previous Taylor expansion. This action corresponds to the various $p_j$ and its value in \eqref{bound_Taylor_expansion} reflect the power of the monomial in $t$ before starting the  next integration by parts.

\begin{remark}
	If one only uses  Definition \eqref{def_dom} for the splitting into dominant and lower parts, one gets a simpler expressions given by
 \begin{equs} \label{Psi_m_ell}
	\begin{aligned}
		\Psi^{r}_{m,a}\left( w \right)(t) & = - 	  \frac{t^{m}}{m!} \sum_{ m_{j+1} = \sum_{\ell \leq j } (p_{\ell}-q_{\ell})  } \, \prod_{j=1}^{|w|} i^{p_j + q_j-1}\binom{m_{j} + p_j}{p_j}   \\ & \times i |\nabla|^\alpha(T_j) \frac{\mathscr{F}_{\!{\tiny\tiny \text{low} }}(w_{[j]})^{p_j}	}{  \mathscr{F}_{\!{\tiny\tiny \text{dom} }}(w_{[j]})^{q_j+1}}      \left(    e^{i t  \mathscr{F}_{\!{\tiny\tiny \text{dom} }}(w)} - a \delta_{m,0} \right).
	\end{aligned}
\end{equs}
	\end{remark}
	
	\begin{example} \label{ex_psi}
		Let us illustrate this complicated definition via an example. One considers the decorated tree $w = T$ given by
		\begin{equs}
			T = 
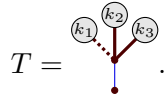
\begin{tikzpicture}[scale=0.2,baseline=-5]
				\coordinate (root) at (0,0);
				\coordinate (tri) at (0,-2);
				\coordinate (t1) at (-2,2);
				\coordinate (t2) at (2,2);
				\coordinate (t3) at (0,3);
				\draw[kernels2,tinydots] (t1) -- (root);
				\draw[kernels2] (t2) -- (root);
				\draw[kernels2] (t3) -- (root);
				\draw[symbols] (root) -- (tri);
				\node[not] (rootnode) at (root) {};t
				\node[not,label= {[label distance=-0.2em]below: \scriptsize  $ $}] (trinode) at (tri) {};
				\node[var] (rootnode) at (t1) {\tiny{$ k_{\tiny{1}} $}};
				\node[var] (rootnode) at (t3) {\tiny{$ k_{\tiny{2}} $}};
				\node[var] (trinode) at (t2) {\tiny{$ k_3 $}};
			\end{tikzpicture}.
		\end{equs}
		One has $ |w|=1 $ and gets 
		\begin{equs} 
			\begin{aligned}
				\Psi^{n,r}_{m,a}\left( w \right)(t) & = 	  -\frac{t^{m}}{m!} \sum_{\substack{m = p - q \leq r +1 }} i^{p + q}     \frac{\mathscr{F}_{\!{\tiny\tiny \text{low} }}^{n,r}(w,0)^{p}	}{  \mathscr{F}^{n,r}_{\!{\tiny\tiny \text{dom} }}(w,0)^{q+1}}      \left(    e^{i t  \mathscr{F}^{n,r}_{\!{\tiny\tiny \text{dom} }}(w,0)} - a \delta_{m,0} \right).
			\end{aligned}
		\end{equs}
		One the other hand, one gets
		\begin{equs}
			(\Pi T)(t) & = - i \int^{t}_0 e^{is \mathscr{F}^{n,1}_{\!{\tiny\tiny \text{dom} }}(w,0)}
			(1 + is \mathscr{F}^{n,1}_{\!{\tiny\tiny \text{low} }}(w,0) + \mathcal{O}(s^2 \mathscr{F}^{n,1}_{\!{\tiny\tiny \text{low} }}(w,0)^2) ds 
			\\ &  = -i \left( \frac{1	}{ i \mathscr{F}^{n,1}_{\!{\tiny\tiny \text{dom} }}(w,0)} - \frac{i \mathscr{F}_{\!{\tiny\tiny \text{low} }}^{n,1}(w,0)	}{ i^2 \mathscr{F}^{n,1}_{\!{\tiny\tiny \text{dom} }}(w,0)^2}  \right)   \left(    e^{i t  \mathscr{F}^{n,1}_{\!{\tiny\tiny \text{dom} }}(w,0)} - 1 \right)
			\\ & - i^2 t \frac{\mathscr{F}_{\!{\tiny\tiny \text{low} }}^{n,1}(w,0)	}{ i \mathscr{F}^{n,1}_{\!{\tiny\tiny \text{dom} }}(w,0)}      \left(    e^{i t  \mathscr{F}^{n,1}_{\!{\tiny\tiny \text{dom} }}(w,0)}  \right)
			+ \mathcal{O}(t^3 \mathscr{F}^{n,1}_{\!{\tiny\tiny \text{low} }}(w,0)^2) 
			\\ & = \Psi^{n,1}_{0,1}\left( w \right)(t) + \Psi^{n,1}_{1,0}\left( w \right)(t)  +  \mathcal{O}(t^3 \mathscr{F}^{n,1}_{\!{\tiny\tiny \text{low} }}(w,0)^2).
			\end{equs}
	\end{example}

	\begin{proposition} \label{partial_t_pi}
		One has
		\begin{equs}
		\partial_t (\Pi \cdot )(t) = - i   \left( (\Pi \cdot)(t) \otimes  e^{i  t \mathscr{F}(\cdot)}|\nabla|^\alpha(\cdot)   \right) \left(\id \otimes P_{A} \right)  \Delta_{\normalfont\text{\tiny{BCK}}}.
		\end{equs}
		\end{proposition}
		\begin{proof}
			Let us consider a forest $F = \prod_{j=1}^p T_j$, then
			\begin{equs}
				\partial_t  \Pi \left( F \right) = \sum_{\ell=1}^p \partial_t \left( \Pi  T_{\ell} \right) \prod_{j \neq \ell} \Pi T_j
			\end{equs}
			on the other side
			\begin{equs}
			\left(\id \otimes P_{A} \right)  \Delta_{\normalfont\text{\tiny{BCK}}}	\prod_{j=1}^p T_j = \sum_{\ell=1}^p (\left(\id \otimes P_{A} \right)  \Delta_{\normalfont\text{\tiny{BCK}}} T_{\ell}) \, ( \one \otimes \prod_{j \neq \ell}  T_j )
			\end{equs}
			which is due to the fact that $P_A$ keeps only letters of $A$, excluding any forest.
			Therefore, we can focus on the case when $F$ is reduced to a decorated tree $T$.
			Then, we can decompose $T$ as in \eqref{star_decomp} 
			\begin{equs}
			T = 	\prod_{j=1}^p T_j 	\star T_r
			\end{equs}
			then from Proposition \ref{int_decomp}, one has
			\begin{equs}
				(\Pi T)(t) = - i |\nabla|^\alpha(T_r) \int_0^t e^{i s  \mathscr{F}(T_r)}
		\prod_{j=1}^p (\Pi T_j)(s)		ds.
				\end{equs}
				Taking the derivative in time, one gets
					\begin{equs}
				\partial_t	(\Pi T)(t) =  - i |\nabla|^\alpha(T_r) e^{i t  \mathscr{F}(T_r)}
					\prod_{j=1}^p (\Pi T_j)(t).
				\end{equs}
				On the other side
				\begin{equs}
				 \left(\id \otimes P_{A} \right)  \Delta_{\normalfont\text{\tiny{BCK}}} T =
				 \prod_{j=1}^p T_j 	\otimes T_r
				\end{equs}
				which is due again to the projection $P_{A}$ that only keeps  elements of $A$. This allows us to conclude.
			\end{proof}
			
			The next proposition makes  a derivative in time appear and performs a Taylor expansion on some lower part, followed by several integrations by parts. This creates a local error which is precisely described via some lower parts of some Fourier operators.
			\begin{proposition} \label{partial_derivative}
				One has for a word $w$ and $r \in \mathbb{N}$
				\begin{equs}
			&	\sum_{m=0}^r	\Psi^{n,r}_{m}\left( Tw \right)(t) = \sum_{m=0}^{r}  \partial_t 	\Psi^{n,r}_{m,0} \left( Tw \right)(t) \\ & +\sum_{\tilde{m}} \mathcal{O}(t^{r+|w| +1} |\nabla|^\alpha(Tw)\prod_{j=1}^{|w|+1}
				\mathscr{F}^{n,r}_{\!{\tiny\tiny \normalfont \text{low} }}((Tw)_{[j]}, \tilde{m}_{[j]})^{r + j-m_j})
				\end{equs}
				where $ \tilde{m} = m_{|w|+1}...m_1 $ with $m_1 =0$ and $m_j \leq r+j$. We have also used the notation
				\begin{equs}
					|\nabla|^\alpha(Tw) =|\nabla|^\alpha(T)  \prod_{j=1}^{|w|} |\nabla|^\alpha(w_j).
				\end{equs}
				\end{proposition}
				\begin{proof}
		From Definition \ref{def_Phi},	one has
					\begin{equs}
					\Psi_{m}^{n,r} \left( Tw \right)(t) & = -  i |\nabla|^\alpha(T) 	 
					\frac{t^{m}}{m!} \sum_{ m_{j+1}= \sum_{\ell \leq j } (p_{\ell}-q_{\ell})} \,
					 \prod_{j=1}^{|w|} i^{p_j + q_j-1}\binom{m_{j} + p_j }{p_j}   \\ & \times  i|\nabla|^\alpha(T_j) \frac{\mathscr{F}_{\!{\tiny\tiny \text{low} }}^{n,r}(w_{[j]},m_{[j]})^{p_j}	}{  \mathscr{F}^{n,r}_{\!{\tiny\tiny \text{dom} }}(w_{[j]},m_{[j]})^{q_j+1}}         e^{i t  (\mathscr{F}^{n,r}_{\!{\tiny\tiny \text{dom} }}(w,\hat{m})  + \mathscr{F}(T) )}.
					\end{equs}
				One proceeds with the Taylor expansion of the lower part:
					\begin{equs}
					\frac{t^{m}}{m!}	e^{i t (  \mathscr{F}^{n,r}_{\!{\tiny\tiny \text{dom} }}(w,\hat{m}) + \mathscr{F}(T) )}
					& = \frac{t^{m}}{m!}	e^{i t (  \mathscr{F}^{n,r}_{\!{\tiny\tiny \text{dom} }}(Tw,\tilde{m}) + \mathscr{F}^{n,r}_{\!{\tiny\tiny \text{dom} }}(Tw, \tilde{m})  )}
					\\  & = \sum_{p \leq  r+ |w| - m} \frac{t^{m+p}}{m! p!}  i^{p} \mathscr{F}^{n,r}_{\!{\tiny\tiny \text{low} }}(Tw, \tilde{m} )^{p} 	e^{i t  \mathscr{F}^{n,r}_{\!{\tiny\tiny \text{dom} }}(Tw, \tilde{m} )} \\ & + \mathcal{O}(t^{r+|w| +1} \mathscr{F}^{n,r}_{\!{\tiny\tiny \text{low} }}(Tw,  \tilde{m} )^{r+|w|+1-m} )
					\end{equs}
					where $ \tilde{m} = m \hat{m} $.
					From Proposition \ref{prop_inte_parts}, by performing several integrations by parts, one has
					\begin{equs}
	&	\frac{t^{m+p}}{m! p!}  i^{p} \mathscr{F}^{n,r}_{\!{\tiny\tiny \text{low} }}(Tw, \tilde{m} )^{p} 	e^{i t  \mathscr{F}^{n,r}_{\!{\tiny\tiny \text{dom} }}(Tw, \tilde{m} )} \\ & = 			\partial_t \left(  \sum_{q \leq m+ p}  	\frac{  (m+p)! t^{m+p-q}}{(m+p -q)!m! p!} i^{p +q-1} \frac{\mathscr{F}^{n,r}_{\!{\tiny\tiny \text{low} }}(Tw, \tilde{m} )^{p}}{\mathscr{F}^{n,r}_{\!{\tiny\tiny \text{low} }}(Tw, \tilde{m} )^{q+1}} e^{i t  \mathscr{F}^{n,r}_{\!{\tiny\tiny \text{dom} }}(Tw, \tilde{m} )} \right)
	\\ & = \partial_t \left(  \sum_{q \leq m+ p}  \binom{m+p}{p}	\frac{   t^{m+p-q}}{(m+p -q)!} i^{p +q-1} \frac{\mathscr{F}^{n,r}_{\!{\tiny\tiny \text{low} }}(Tw, \tilde{m} )^{p}}{\mathscr{F}^{n,r}_{\!{\tiny\tiny \text{low} }}(Tw, \tilde{m} )^{q+1}} e^{i t  \mathscr{F}^{n,r}_{\!{\tiny\tiny \text{dom} }}(Tw, \tilde{m} )} \right). 
						\end{equs}
One gets
				\begin{equs}
				&	\Psi^{n,r}_m \left(Tw  \right)(t)  = -	\partial_t \sum_{ m_{j+1}= \sum_{\ell \leq j } (p_{\ell}-q_{\ell}) } \, \prod_{j=1}^{|w|} i^{p_j + q_j-1}\binom{m_j + p_j }{p_j}   \\ & \times i |\nabla|^\alpha(w_j) \frac{\mathscr{F}^{n,r}_{\!{\tiny\tiny \text{low} }}(w_{[j]},m_{[j]})^{p_j}	}{  \mathscr{F}^{n,r}_{\!{\tiny\tiny \text{dom} }}(w_{[j]},m_{[j]})^{q_j+1}} \\ &  \sum_{p \leq  r+ |w|-m} \sum_{q \leq m+ p}  \binom{m+p}{p}	\frac{   t^{m+p-q}}{(m+p -q)!} i^{p +q-1} \\ & \times  i|\nabla|^\alpha(T)	  \frac{\mathscr{F}^{n,r}_{\!{\tiny\tiny \text{low} }}(Tw, \tilde{m} )^{p}}{\mathscr{F}^{n,r}_{\!{\tiny\tiny \text{dom} }}(Tw,\tilde{m} )^{q+1}} e^{i t  \mathscr{F}^{n,r}_{\!{\tiny\tiny \text{dom} }}(Tw, \tilde{m})}
				\\ & +  \sum_{\tilde{m}} \mathcal{O}(t^{r+|w| +1} |\nabla|^\alpha(Tw) \prod_{j=1}^{|w|+1}\mathscr{F}^{n,r}_{\!{\tiny\tiny \text{low} }}((Tw)_{[j]}, \tilde{m}_{[j]})^{r + j-m_j}) .
				\end{equs}
				Then, one fixes $ p_{|w|+1} = p $ and $ q_{|w|+1}  = q $. By summing over $m$, one gets
					\begin{equs}
						\sum_{m=0}^r \Psi^{n,r}_m \left(T w \right)(t)  & = \sum_{m=0}^r \partial_t \sum_{ m_{j+1}= \sum_{\ell \leq j } (p_{\ell}-q_{\ell}) } \, \prod_{j=1}^{|Tw|} i^{p_j + q_j-1}\binom{m_j + p_j }{p_j}   \\ & \times (-i |\nabla|^\alpha(w_j))\frac{\mathscr{F}^{n,r}_{\!{\tiny\tiny \text{low} }}(w_{[j]},m_{[j]})^{p_j}	}{  \mathscr{F}^{n,r}_{\!{\tiny\tiny \text{dom} }}(w_{[j]}, m_{[j]})^{q_j+1}}  	\frac{   t^{m}}{m!}  e^{i t  \mathscr{F}^{n,r}_{\!{\tiny\tiny \text{dom} }}(Tw, \tilde{m} )}
				\\ & 	  + \sum_{\tilde{m}} \mathcal{O}(t^{r+|w| +1} |\nabla|^\alpha(Tw) \prod_{j=1}^{|w|+1}\mathscr{F}^{n,r}_{\!{\tiny\tiny \text{low} }}((Tw)_{[j]}, \tilde{m}_{[j]})^{r + j-m_j}) 
					 \\ & = \sum_{m=0}^{r}  \partial_t 	\Psi^{n,r}_{m,0} \left( Tw \right)(t) \\ &+ \sum_{\tilde{m}} \mathcal{O}(t^{r+|w| +1} |\nabla|^\alpha(Tw) \prod_{j=1}^{|w|+1}\mathscr{F}^{n,r}_{\!{\tiny\tiny \text{low} }}((Tw)_{[j]}, \tilde{m}_{[j]})^{r + j-m_j}) 
				\end{equs}
				which allows us to conclude.
					\end{proof}
The next proposition makes appear the inductive construction of the low regularity scheme via the arborification map $ \mathfrak{a} $ and the Butcher-Connes-Kreimer type coproduct $ \Delta_{\normalfont\text{\tiny{BCK}}} $. The idea is to discretise the oscillatory integral associated with a subtree containing the root of the decorated tree we started with. Then, one wants to iterate the low regularity approach to the rest of the oscillatory integral which is disconnected from the first discretisation.
	\begin{proposition} \label{Taylor_phi} One has for $ \ell \in \mathbb{N}^*$
		\begin{equs}
		&	\Pi (T)(t) =  \sum_{p < \ell} ( \Psi^{n,r}_{0,1}( \mathfrak{a}_{p}(T))(t) - 	\Psi^{n,r}_{0,0}( \mathfrak{a}_{p}(T))(t) \\ &  + \sum_{m =0}^r
		  (\Pi \cdot)(t) \otimes\Psi^{n,r}_{m,0}\left( \mathfrak{a}_{p}(\cdot) \right)(t)) 	\Delta_{\normalfont\text{\tiny{BCK}}} T ) \\ &  + \sum_{m= 0}^r \int_0^t \left(  (\Pi \cdot)(s) \otimes \Psi_{m}^{n,r}\left( \mathfrak{a}_{\ell}(\cdot) \right)(s) \right)	\Delta_{\normalfont \text{\tiny{BCK}}} T ds) \\& +  \sum_{p < \ell} \mathcal{O}( (    (\Pi \cdot)(t) \otimes t^{r+p+ 1} \sum_{|\tilde{m}| =p }|\nabla|^\alpha(\cdot)\prod_{j=1}^{p}\mathscr{F}^{n,r}_{\normalfont \!{\tiny\tiny \text{low} }}((\mathfrak{a}_{p}(\cdot))_{[j]},\tilde{m}_{[j]})^{r+j-m_j}  ) 	\Delta_{\normalfont\text{\tiny{BCK}}} T )
		\end{equs}
		where we have made the following abuse of notations
		\begin{equs}
			\mathscr{F}^{n,r}_{\normalfont \!{\tiny\tiny \text{low} }}(\sum_{i} \lambda_i w_i,\tilde{m})^{r} & =\sum_{i} \lambda_i \mathscr{F}^{n,r}_{\normalfont \!{\tiny\tiny \text{low} }}( w_i,\tilde{m})^{r},
			\\
	(\mathfrak{a}_p(T))_{[j]} &= \sum_{i} \lambda_i (w_i)_{[j]}.
		\end{equs}
		Here, $w_i$ are words on the alphabet $A$ with the same length.
		\end{proposition}
		\begin{proof} One proceeds by recurrence on $\ell$. We start by $\ell=1$.
			For $T = \prod_{j=1}^p T_{j} \star T_{r} $ as in \eqref{star_decomp}, one gets
			\begin{equs}
				\int_0^t \left(  (\Pi \cdot)(s) \otimes \Psi^{n,r}_{0}\left( \mathfrak{a}_{1}(\cdot) \right)(s) \right)	\Delta_{\text{\tiny{BCK}}} T ds = \int_0^t \prod_{j=1}^p \left(  (\Pi T_{j})(s)  \Psi^{n,r}_{0}\left( T_{r} \right)(s) \right)  ds.
			\end{equs}
			On the other hand, one has from Proposition \ref{int_decomp}
			\begin{equs}
				\Pi (T)(t) = - i |\nabla|^\alpha(T_r)
				\int_0^t e^{i s \mathscr{F}(T_{r}) } \prod_{j=1}^p (\Pi T_{j})(s) ds.
			\end{equs}
			We conclude by using \eqref{itermediate_psi_bis} which gives
			\begin{equs}
			- i |\nabla|^\alpha(T_r)	e^{i s \mathscr{F}(T_{r}) } = \Psi^{n,r}_{0}\left( T_{r} \right)(s).
			\end{equs}
			Now, we suppose the property true for some  $\ell$. Then, using Proposition \ref{partial_derivative}, one has
			\begin{equs}
			\sum_{m =0}^r	\int_0^t & \left(  (\Pi \cdot)(s) \otimes \Psi^{n,r}_{m}\left( \mathfrak{a}_{\ell}(\cdot) \right)(s) \right)	\Delta_{\text{\tiny{BCK}}} T ds \\ &  = \sum_{m =0}^r\int_0^t \left(  (\Pi \cdot)(s) \otimes \partial_t \Psi^{n,r}_{m,0}\left( \mathfrak{a}_{\ell}(\cdot) \right)(s) \right)	\Delta_{\text{\tiny{BCK}}} T ds 
			\\ &	  +  \mathcal{O}( (    (\Pi \cdot)(t) \otimes t^{r+\ell+ 1} \sum_{|\tilde{m}| =\ell }|\nabla|^\alpha(\cdot)\prod_{j=1}^{\ell}\mathscr{F}^{n,r}_{\!{\tiny\tiny \text{low} }}((\mathfrak{a}_{\ell}(\cdot))_{[j]},\tilde{m}_{[j]})^{r+j-m_j}  ) 	\Delta_{\text{\tiny{BCK}}} T ).	\end{equs}
			Via an integration by parts, one gets
			\begin{equs}
				&\int_0^t  \left(  (\Pi \cdot)(s) \otimes \partial_t \Psi^{n,r}_{m,0}\left( \mathfrak{a}_{\ell}(\cdot) \right)(s) \right)	\Delta_{\text{\tiny{BCK}}} T ds \\ & = - \int_0^t \left(  \partial_t  (\Pi \cdot)(s) \otimes\Psi^{n,r}_{m,0}\left( \mathfrak{a}_{\ell}(\cdot) \right)(s) \right)	\Delta_{\text{\tiny{BCK}}} T ds 
				\\ &+    \left(    (\Pi \cdot)(t) \otimes\Psi^{n,r}_{m,0}\left( \mathfrak{a}_{\ell}(\cdot) \right)(t) \right)	\Delta_{\text{\tiny{BCK}}} T
				\\ & + \one_{\lbrace \ell =|T|_{\text{\tiny{ord}}} \rbrace \cap \, \lbrace  m = 0 \rbrace} \left(  \Psi^{n,r}_{0,1}( \mathfrak{a}_{\ell}(T))(t) - 	\Psi^{n,r}_{0,0}( \mathfrak{a}_{\ell}(T))(t) \right).
			\end{equs}		
			This is due to the fact that 
			\begin{equs}
				\int_0^t \partial_t \left(   \Psi^{n,r}_{0,0}\left( \mathfrak{a}_{\ell}(T) \right)(s) \right)  ds  =  \Psi^{n,r}_{0,1}\left( \mathfrak{a}_{\ell}(T) \right)(t).
			\end{equs}
			Then, by using Proposition \ref{partial_t_pi}, one has
			\begin{equs}
				\, & - \left(  \partial_t  (\Pi \cdot)(s) \otimes\Psi^{n,r}_{m,0}\left( \mathfrak{a}_{\ell}(\cdot) \right)(s) \right)	\Delta_{\text{\tiny{BCK}}}
				\\ & =  \left( (\Pi \cdot)(s) \otimes i | \nabla|^\alpha(\cdot)    e^{i  s \mathscr{F}(\cdot)} \otimes \Psi^{n,r}_{m,0}\left( \cdot \right)(s) \right) \left((\id \otimes P_{A}) \Delta_{\text{\tiny{BCK}}}  \otimes \, \mathfrak{a}_{\ell} \right) \Delta_{\text{\tiny{BCK}}} 
				\\ & = \left( (\Pi \cdot)(s) \otimes i | \nabla|^\alpha(\cdot)    e^{i  s \mathscr{F}(\cdot)} \otimes \Psi^{n,r}_{m,0}\left( \cdot \right)(s) \right) \left(\id \otimes P_{A}  \otimes \, \mathfrak{a}_{\ell} \right) (\id \, \otimes  \curvearrowright^{*}) \Delta_{\text{\tiny{BCK}}}
				\\ & = \left( (\Pi \cdot)(s) \otimes   \Psi_{m}^{n,r}\left( \cdot \right)(s) \right) \left(\id \otimes \mathcal{M}_c (P_{A}  \otimes \, \mathfrak{a}_{\ell}) \curvearrowright^{*} \right)  \Delta_{\text{\tiny{BCK}}} 
				\\  & = \left( (\Pi \cdot)(s) \otimes   \Psi_{m}^{n,r}\left( \cdot \right)(s) \right) \left(\id \otimes \mathfrak{a}_{\ell+1} \right)  \Delta_{\text{\tiny{BCK}}} 
			\end{equs}
			where 
			we have used for the third line Proposition \ref{co-asso}, for the fourth line that for $a \in A$ and $ w \in T(A) $
			\begin{equs}
				\, 	\left( i | \nabla|^\alpha(\cdot)  e^{i  s \mathscr{F}(\cdot)} \otimes \Psi^{n,r}_{m,0}\left( \cdot \right)(s) \right) \left( a  \otimes w \, \right)  & =   i | \nabla|^\alpha(a) e^{i  s \mathscr{F}(a)}  \Psi^{n,r}_{m,0}\left( w \right)(s) 
				\\ &  =  \Psi^{n,r}_{m}\left( a w \right)(s)
				\\ &  = \Psi^{n,r}_{m}\left( \cdot \right)(s) \circ  \mathcal{M}_c( a \otimes w ) 
			\end{equs}
			 and for the fifth line
			 \begin{equs}
			 	\mathfrak{a}_{\ell+1} = \mathcal{M}_c (P_{A}  \otimes \, \mathfrak{a}_{\ell}) \curvearrowright^{*}.
			 \end{equs}
			One gets from the previous computations
			\begin{equs}
			\int_0^t & \left(  \partial_t  (\Pi \cdot)(s) \otimes\Psi^{n,r}_{m,0}\left( \mathfrak{a}_{\ell}(\cdot) \right)(s) \right)	\Delta_{\text{\tiny{BCK}}} T ds \\&  = \int_0^t \left(    (\Pi \cdot)(s) \otimes\Psi_{m}^{n,r}\left( \mathfrak{a}_{\ell+1}(\cdot) \right)(s) \right)	\Delta_{\text{\tiny{BCK}}} T ds 
			\end{equs}
			which allows us to conclude.
			\end{proof}
			The next corollary is an immediate consequence of the previous proposition.
		\begin{corollary} \label{main_corollary} One gets
					\begin{equs}
				&	\Pi (T)(t) =   \Psi^{n,r}_{0,1}( \mathfrak{a}_{}(T))(t) - 	\Psi^{n,r}_{0,0}( \mathfrak{a}_{}(T))(t) \\ &   + \sum_{m =0}^r
					\left((\Pi \cdot)(t) \otimes\Psi^{n,r}_{m,0}\left( \mathfrak{a}_{}(\cdot) \right)(t)\right) 	\tilde{\Delta}_{\normalfont\text{\tiny{BCK}}} T   \\& +  \sum_{p \leq |T|_{\normalfont \text{\tiny{ord}}}} \mathcal{O}( (    (\Pi \cdot)(t) \otimes t^{r+p+ 1} \sum_{|\tilde{m}| =p }|\nabla|^\alpha(\cdot)\prod_{j=1}^{p}\mathscr{F}^{n,r}_{\normalfont\!{\tiny\tiny \text{low} }}((\mathfrak{a}_{p}(\cdot))_{[j]},\tilde{m}_{[j]})^{r+j-m_j}  ) 	\Delta_{\normalfont\text{\tiny{BCK}}} T )
				\end{equs}
				where 
				\begin{equs}
					\tilde{\Delta}_{\normalfont\text{\tiny{BCK}}} T =
						\Delta_{\normalfont\text{\tiny{BCK}}} T
						- \one \otimes T. 
					\end{equs}
		\end{corollary}
		\begin{proof}
				By taking $\ell$ strictly larger than $|T|_{\text{\tiny{ord}}}$, one has
			\begin{equs}
				\int_0^t \left(  (\Pi \cdot)(s) \otimes \Psi^{n,r}_{m}\left( \mathfrak{a}_{\ell}(\cdot) \right)(s) \right)	\Delta_{\text{\tiny{BCK}}} T ds = 0.
			\end{equs}
		 One  also notices that
			\begin{equs}
				\sum_{m =1}^{  |T|_{\text{\tiny{ord}}}} \left( \id \otimes	\mathfrak{a}_{m}(\cdot) \right)\Delta_{\text{\tiny{BCK}}} T = \left( \id \otimes 	\mathfrak{a}(\cdot) \right)\tilde{\Delta}_{\text{\tiny{BCK}}} T.
			\end{equs}
			Then, we conclude by applying  Proposition \ref{Taylor_phi},
			\end{proof}

	\begin{theorem} \label{main_theorem} One sets the following approximation of $ (\Pi T)(t) $
		\begin{equs}\label{main_formula_theorem}
			\begin{aligned}
			(\Pi^{n,r} T)(t) &=   \Psi^{n,r_T}_{0,1}( \mathfrak{a}_{}(T))(t) - 	\Psi^{n,r_T}_{0,0}( \mathfrak{a}_{}(T))(t) \\ &   + \sum_{m =0}^{r}
		\left( \mathcal{P}_{r-m} \left( 	(\Pi^{n,r-m} \cdot)(t) \right) \otimes\Psi^{n,r_T}_{m,0}\left( \mathfrak{a}_{}(\cdot) \right)(t) \right) 	\tilde{\Delta}_{\normalfont\text{\tiny{BCK}}} T 
		\end{aligned}
		\end{equs}
		where $r_T = r - |T|_{\text{\tiny{ord}}}$ and $\mathcal{P}_r$ removes all  the monomials of the form $ t^{p}$ with $p > r$. 
	Then,
		\begin{equs} \label{local_error_scheme}
			(\Pi T)(t) - (\Pi^{n,r} T)(t) = \mathcal{O}( t^{r+1}E^{n,r}(T) )
		\end{equs}
		where 
		\begin{equs} \label{local_error}
			\begin{aligned}
		&	E^{n,r}(T) =  \sum_{p < \ell} (    D^{\alpha}(\cdot)\otimes  \sum_{|\tilde{m}| =p }|\nabla|^\alpha(\cdot)\prod_{j=1}^{p}\mathscr{F}^{n,r_T}_{\normalfont \!{\tiny\tiny \text{low} }}((\mathfrak{a}_{p}(\cdot))_{[j]},\tilde{m}_{[j]})^{r_T+j-m_j}  ) 	\Delta_{\normalfont\text{\tiny{BCK}}} T 
		\\ & +  \sum_{m =0}^{r}
		\left( E^{n,r-m}_+(\cdot) \otimes \frac{1}{t^m}\Psi^{n,r_T}_{m,0}\left( \mathfrak{a}_{}(\cdot) \right)(t) \right) 	\tilde{\Delta}_{\normalfont\text{\tiny{BCK}}} T 
		\end{aligned}
		\end{equs}
		with
		\begin{equs}
			E^{n,r-m}_+(\prod_{j=1}^q T_j) = \sum_{j=1}^q	E^{n,r-m}(T_j).
		\end{equs}
		and 
		\begin{equs}
			 D^{\alpha}(\prod_{j=1}^q T_j) = \prod_{j=1}^q D^{\alpha}(T_j), \quad D^{\alpha}( 	\prod_{j=1}^p T_j 	\star T_r) =|\nabla|^\alpha(T_r) D^{\alpha}( 	\prod_{j=1}^p T_j ). 
		\end{equs}
		\end{theorem}
		\begin{proof}
	This is a consequence of Corollary \ref{main_corollary}. The first error term comes from this Corollary together with the bound:
	\begin{equs}
		(\Pi T)(t) = \mathcal{O}(t^{|T|_{\text{\tiny{ord}}}} D^{\alpha}(T)).
	\end{equs}
	For the other error terms, one notices that
	\begin{equs}
	(\Pi - 	\Pi^{n,r})(T \cdot \bar{T})(t)  &  = (\Pi - 	\Pi^{n,r})(T )(t)  (\Pi^{n,r} \bar{T} )(t)
	+ (\Pi T )(t)(\Pi - 	\Pi^{n,r})(\bar{T} )(t)
	\\ &=   \mathcal{O}(t^{r+1} E^{n,r}(T)  ) +  \mathcal{O}(t^{r+1} E^{n,r}(\bar{T})  )
	\end{equs}
which gives us the additive structure on the local error.	 	
	\end{proof}
	
	The previous theorem gives a low regularity approximation of oscillatory integrals of the form $(\Pi T)(t)$. Equipped with this approximation, one gives in the next corollary a low regularity scheme.
	
	\begin{corollary} \label{main_low_schemes}
	Using the same notations as in Theorem \ref{main_theorem}, one derives the following low regularity scheme:
	\begin{equs}\label{genscheme_low}
		U_{k}^{n,r}(t, u_0) =   \sum_{T \in \CT^{\leq r,k}_{0}} \frac{\Upsilon( T)(u_0)}{S(T)} (\Pi^{n,r}   T )(t)
	\end{equs}
	with the local error
	\begin{equs} \label{local_low}
	u_k(t) -	U_{k}^{n,r}(t, u_0)  = \sum_{T \in \CT^{\leq r,k}_{0}}  \mathcal{O}( t^{r+1}E^{n,r}(T) \Upsilon( T)(u_0) ).
	\end{equs}
		\end{corollary}
		\begin{proof}
			One has the following decomposition
			\begin{equs}
					u_k(t) -	U_{k}^{n,r}(t, u_0)  =  	u_k(t) -	U_{k}^{r}(t, u_0) + 		U_{k}^{r}(t, u_0) - U_{k}^{n,r}(t, u_0). 
			\end{equs}
			For the first difference, one applies Proposition \ref{tree_series}. Then for the second term, one observes:
			\begin{equs}
				U_{k}^{r}(t, u_0) - U_{k}^{n,r}(t, u_0) = \sum_{T \in \CT^{\leq r,k}_{0}} \frac{\Upsilon( T)(u_0)}{S(T)} \left( (\Pi T)(t)- (\Pi^{n,r}   T )(t) \right).
			\end{equs}
			We conclude by applying Theorem \ref{main_theorem}, in particular identity \eqref{local_error_scheme}.
			\end{proof}
	
	\begin{remark}
		In \cite{BS}, one also added the extra parameter $n \in \mathbb{N}^*$ for performing the discretisation  $ (\Pi^{n,r} \cdot) $. Here $n$ is the regularity a priori of the solution, meaning that we assume that the solution is  $ \mathcal{C}^{n}$ in space. With this information, one is able to write a simple scheme and perform a full Taylor expansion without integrating some dominant part and performing a resonance analysis. 
		\end{remark}
		
		\begin{remark}
			 Corollary \ref{main_low_schemes} provides a local error estimate for the low regularity scheme \eqref{genscheme_low}. For stability estimates, one relies  on the algebraic structure
			of the underlying space. In the stability analysis of dispersive PDEs set in Sobolev
			spaces $H^r$ the following bilinear estimates of type
			\begin{equs}
				\Vert vw \Vert_r \leq c_{r,d} \Vert v \Vert_r \Vert w \Vert_r;
			\end{equs}
			The latter only hold for $r > d/2$. To obtain $L^2$ global error estimates one
			needs to exploit discrete Strichartz estimates and discrete Bourgain spaces which is beyond the scope of this paper. These limitations appear also in the low regularity schemes given in \cite{BS} (see Remark 4.10).
		\end{remark}
		
		\begin{remark} \label{local_error_physical}The local error is written in Fourier space in  \eqref{local_low}. It can also be written in Physical space.
		Let $\Phi_\tau(u_0) =u^1 \approx u(\tau)$ denote the numerical solution at time $t=\tau$. We write 
		\begin{equation}\label{eq:deflocalErr}
			u(\tau) -\Phi_\tau(u_0) = \CO_{\| \cdot\|}(\tau^m  \tilde{\mathcal{L}} u_0)
		\end{equation}
		if in a suitable norm $\| \cdot\|$, it holds that 
		\begin{equation}\label{eq:deflocalErrBound}
			\| u(\tau) -\Phi_\tau(u_0)\| \le C(T,d) \tau^m \sup_{0\le t\le\tau} \|q\left( \tilde{\mathcal{L}} u(t)\right) \|,
		\end{equation}
		for some polynomial $q$, differential operator $\tilde{\mathcal{L}} $, and constant $C$ independent of $\tau$.
		The operator $q\left( \tilde{\mathcal{L}} \cdot\right)$ in Fourier space is associated with  $\sum_{T \in \CT^{\leq r,k}_{0}}  E^{n,r}(T) \Upsilon( T)(\cdot) )$.
		\end{remark}
		
		\begin{example} \label{discretisation_int}
			Let us apply our normal form schemes on iterated integrals coming from NLS with $ n=2 $ and $r=2$.
			We start with the following decorated tree 
			\begin{equs}
			T_1 = 	\begin{tikzpicture}[scale=0.2,baseline=-5]
					\coordinate (root) at (0,0);
					\coordinate (tri) at (0,-2);
					\coordinate (t1) at (-2,2);
					\coordinate (t2) at (2,2);
					\coordinate (t3) at (0,3);
					\draw[kernels2,tinydots] (t1) -- (root);
					\draw[kernels2] (t2) -- (root);
					\draw[kernels2] (t3) -- (root);
					\draw[symbols] (root) -- (tri);
					\node[not] (rootnode) at (root) {};t
					\node[not,label= {[label distance=-0.2em]below: \scriptsize  $ $}] (trinode) at (tri) {};
					\node[var] (rootnode) at (t1) {\tiny{$ k_{\tiny{1}} $}};
					\node[var] (rootnode) at (t3) {\tiny{$ k_{\tiny{2}} $}};
					\node[var] (trinode) at (t2) {\tiny{$ k_3 $}};
				\end{tikzpicture}, \quad \mathfrak{a}(T_1) = T_1.
				\end{equs}
				If we apply \eqref{main_formula_theorem}, one gets 
				\begin{equs}
					r_{T_1} = r - |T_1|_{\text{\tiny{ord}}} = 2 -1 =  1
					\end{equs}
					 and
				\begin{equs}
					(\Pi^{2,2} T_1)(t) & = \Psi^{2,1}_{0,1}( \mathfrak{a}_{}(T_1))(t) - 	\Psi^{2,1}_{0,0}( \mathfrak{a}_{}(T_1))(t) \\ &   + \sum_{m =0}^2
					\left( \mathcal{P}_{2-m} \left( 	(\Pi^{2,2-m} \cdot)(t) \right) \otimes\Psi^{2,1}_{m,0}\left( \mathfrak{a}_{}(\cdot) \right)(t) \right) 	\tilde{\Delta}_{\normalfont\text{\tiny{BCK}}} T_1. 
				\end{equs}
				From the fact that $T_1$ is primitive that is
				\begin{equs}
			 \Delta_{\normalfont\text{\tiny{BCK}}} T_1  = T_1 \otimes \one + \one \otimes T_1, \quad  \tilde{\Delta}_{\normalfont\text{\tiny{BCK}}} T_1  = T_1 \otimes \one,
				\end{equs}
				one gets 
				\begin{equs}
				(\Pi^{2,2} T_1)(t) & = \Psi^{2,1}_{0,1}( \mathfrak{a}_{}(T_1))(t) - 	\Psi^{2,1}_{0,0}( \mathfrak{a}_{}(T_1))(t)   + \sum_{m =0}^2
				\Psi^{2,1}_{m,0}\left( \mathfrak{a}_{}(T_1) \right)(t)
				\\	 & =\Psi^{2,1}_{0,1}( T_1)(t) + 
				\Psi^{2,1}_{1,0}\left( T_1 \right)(t) + \Psi^{2,1}_{2,0}\left( T_1 \right)(t).
				\end{equs}
				Combining the computations of Examples \ref{ex_f} and \ref{ex_psi}, one gets
					\begin{equs}
					\mathscr{F}^{2,1}_{\!{\tiny \text{dom} }}(T_1,0)  =  2 k_1^2,
					\quad \mathscr{F}^{2,1}_{\!{\tiny \text{low} }}(T_1,0)  =  
					-2 k_1(k_2 + k_3) + 2 k_2 k_3.
				\end{equs}
				and 
				\begin{equs}
			(\Pi^{2,2} T_1)(t)	&	= -i \left( \frac{1	}{ i \mathscr{F}^{2,1}_{\!{\tiny\tiny \text{dom} }}(T_1,0)} - \frac{i \mathscr{F}_{\!{\tiny\tiny \text{low} }}^{2,1}(T_1,0)	}{ i^2 \mathscr{F}^{2,1}_{\!{\tiny\tiny \text{dom} }}(T_1,0)^2}  \right)   \left(    e^{i t  \mathscr{F}^{2,1}_{\!{\tiny\tiny \text{dom} }}(T_1,0)} - 1 \right)
				\\ &	 - i^2 t \frac{\mathscr{F}_{\!{\tiny\tiny \text{low} }}^{2,1}(T_1,0)	}{ i \mathscr{F}^{2,1}_{\!{\tiny\tiny \text{dom} }}(T_1,0)}      \left(    e^{i t  \mathscr{F}^{2,1}_{\!{\tiny\tiny \text{dom} }}(T_1,0)}  \right).
				\end{equs}
				where we have used the fact that 
				\begin{equs}
				\Psi^{2,1}_{2,0}\left( T_1 \right)(t)  =0.	
				\end{equs}
				This due to the fact that 
				\begin{equs}
					\Psi^{m,1}_{2,0}\left( T_1 \right)(t) &=-	  \frac{t^{m}}{m!} \sum_{ m=  p-q  } \, i^{p + q-1} \times i \frac{\mathscr{F}_{\!{\tiny\tiny \text{low} }}^{2,1}(T_1,0)^{p}	}{  \mathscr{F}^{2,1}_{\!{\tiny\tiny \text{dom} }}(T_1,0)^{q+1}}          e^{i t  \mathscr{F}^{n,r}_{\!{\tiny\tiny \text{dom} }}(w,m0)} 
					\end{equs}
					with $ p = r+1-1 -0 = 1 $ and $q \geq 0$ which is not compatible with $2=  m=p-q$. Therefore, this term is equal to zero.
				The local error is then given by
				\begin{equs}
					E^{2,2}(T_1) =  (\mathscr{F}^{2,1}_{\normalfont \!{\tiny\tiny \text{low} }}(T_1,0))^{2} .
				\end{equs}
				Both low regularity approximation and local error are exactly the same as obtained in \cite{OS18,BS}.
			 One considers the following decorated trees
		\begin{equs}
			T_3  = 
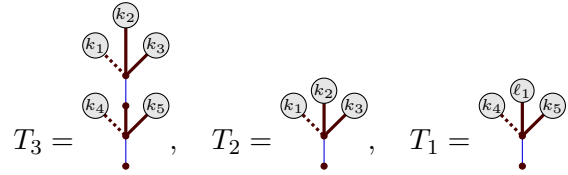
\begin{tikzpicture}[scale=0.2,baseline=-5]
				\coordinate (root) at (0,0);
				\coordinate (tri) at (0,-2);
				\coordinate (t1) at (-2,2);
				\coordinate (t2) at (2,2);
				\coordinate (t3) at (0,2);
				\coordinate (t4) at (0,4);
				\coordinate (t41) at (-2,6);
				\coordinate (t42) at (2,6);
				\coordinate (t43) at (0,8);
				\draw[kernels2,tinydots] (t1) -- (root);
				\draw[kernels2] (t2) -- (root);
				\draw[kernels2] (t3) -- (root);
				\draw[symbols] (root) -- (tri);
				\draw[symbols] (t3) -- (t4);
				\draw[kernels2,tinydots] (t4) -- (t41);
				\draw[kernels2] (t4) -- (t42);
				\draw[kernels2] (t4) -- (t43);
				\node[not] (rootnode) at (root) {};
				\node[not] (rootnode) at (t4) {};
				\node[not] (rootnode) at (t3) {};
				\node[not,label= {[label distance=-0.2em]below: \scriptsize  $  $}] (trinode) at (tri) {};
				\node[var] (rootnode) at (t1) {\tiny{$ k_{\tiny{4}} $}};
				\node[var] (rootnode) at (t41) {\tiny{$ k_{\tiny{1}} $}};
				\node[var] (rootnode) at (t42) {\tiny{$ k_{\tiny{3}} $}};
				\node[var] (rootnode) at (t43) {\tiny{$ k_{\tiny{2}} $}};
				\node[var] (trinode) at (t2) {\tiny{$ k_5 $}};
			\end{tikzpicture},  \quad T_2 =\begin{tikzpicture}[scale=0.2,baseline=-5]
				\coordinate (root) at (0,0);
				\coordinate (tri) at (0,-2);
				\coordinate (t1) at (-2,2);
				\coordinate (t2) at (2,2);
				\coordinate (t3) at (0,3);
				\draw[kernels2,tinydots] (t1) -- (root);
				\draw[kernels2] (t2) -- (root);
				\draw[kernels2] (t3) -- (root);
				\draw[symbols] (root) -- (tri);
				\node[not] (rootnode) at (root) {};t
				\node[not,label= {[label distance=-0.2em]below: \scriptsize  $ $}] (trinode) at (tri) {};
				\node[var] (rootnode) at (t1) {\tiny{$ k_{\tiny{1}} $}};
				\node[var] (rootnode) at (t3) {\tiny{$ k_{\tiny{2}} $}};
				\node[var] (trinode) at (t2) {\tiny{$ k_3 $}};
			\end{tikzpicture},
			\quad T_1 = 	\begin{tikzpicture}[scale=0.2,baseline=-5]
				\coordinate (root) at (0,0);
				\coordinate (tri) at (0,-2);
				\coordinate (t1) at (-2,2);
				\coordinate (t2) at (2,2);
				\coordinate (t3) at (0,3);
				\draw[kernels2,tinydots] (t1) -- (root);
				\draw[kernels2] (t2) -- (root);
				\draw[kernels2] (t3) -- (root);
				\draw[symbols] (root) -- (tri);
				\node[not] (rootnode) at (root) {};t
				\node[not,label= {[label distance=-0.2em]below: \scriptsize  $  $}] (trinode) at (tri) {};
				\node[var] (rootnode) at (t1) {\tiny{$ k_{\tiny{4}} $}};
				\node[var] (rootnode) at (t3) {\tiny{$ \ell_1 $}};
				\node[var] (trinode) at (t2) {\tiny{$ k_5 $}};
			\end{tikzpicture}
		\end{equs}
		where $ \ell_1 = -k_1 + k_2 + k_3 $. Then,
		one has
		\begin{equs}
			\mathfrak{a}(T) = T_2 T_1,
			\quad
			\tilde{\Delta}_{\normalfont\text{\tiny{BCK}}} T_3 = \one \otimes T_3 +
		T_2 \otimes T_1.
			\end{equs}
			One has 
			\begin{equs}
				r_{T_3} = r - |T_3|_{\text{\tiny{ord}}} = 2 -2 =  0
			\end{equs}
			and
			\begin{equs}
				(\Pi^{2,2} T_3)(t) & = \Psi^{2,0}_{0,1}( \mathfrak{a}_{}(T_3))(t) - 	\Psi^{2,0}_{0,0}( \mathfrak{a}_{}(T_3))(t)  \\ &  + 
			\sum_{m=0}^2	\left(  	(\Pi^{2,2-m} \cdot)(t)  \otimes\Psi^{2,0}_{m,0}\left( \mathfrak{a}_{}(\cdot) \right)(t) \right) 	\tilde{\Delta}_{\normalfont\text{\tiny{BCK}}} T_3 
				\\ & = \Psi^{2,0}_{0,1}( T_2 T_1)(t) + 	\Psi^{2,0}_{1,0}( T_2 T_1)(t) + \Psi^{2,0}_{2,0}( T_2 T_1)(t) \\ &+ \sum_{m=0}^2	(\Pi^{2,2-m} T_2)(t) \, \Psi^{2,0}_{m,0}\left( T_1 \right)(t). 	
			\end{equs}
		 One has for $m \in \lbrace 0,1,2 \rbrace$,
			\begin{equs}
				\mathscr{F}^{2,0}_{\!{\tiny \text{dom} }}( T_1,  0)  = \mathscr{F}^{2,0}_{\!{\tiny \text{dom} }}(& T_2,  0) = 0,
				\quad \mathscr{F}^{2,0}_{\!{\tiny \text{low} }}( T_1,  0)  =  \mathscr{F}( T_1), \quad  \mathscr{F}^{2,0}_{\!{\tiny \text{low} }}( T_2,  0)  =  \mathscr{F}( T_2),
				\\
					\mathscr{F}^{2,0}_{\!{\tiny \text{dom} }}(T_2 T_1, m 0)  &=  0,
					\quad \mathscr{F}^{2,0}_{\!{\tiny \text{low} }}(T_2 T_1, m 0)   =  \mathscr{F}(T_2 ) +   \mathscr{F}^{2,0}_{\!{\tiny \text{low} }}( T_1,0) = \mathscr{F}(T_2 )
					.
			\end{equs}
			This is due to the fact that 
			\begin{equs}
				\deg(\mathscr{F}( T_1)) =  2 \leq 2, \quad \deg(\mathscr{F}( T_2)) =  2 \leq 2.
			\end{equs}
			Then, one has for $ m \in \lbrace 0,1,2 \rbrace $
		\begin{equs}
			\Psi^{2,0}_{m,0}\left( T_1 \right)(t) & = - it\one_{\lbrace m=1 \rbrace}, \quad (\Pi^{2,1} T_2)(t) = - it,
			\\ 
			\Psi^{2,0}_{0,1}( T_2 T_1)(t)  & = 0, \quad  	\Psi^{2,0}_{1,0}( T_2 T_1)(t) ) = 0, \quad  \Psi^{2,0}_{2,0}( T_2 T_1)(t)  = - i^2 \frac{t^2}{2}.
		\end{equs}
		In the end, one gets
		\begin{equs}
				(\Pi^{2,2} T_3)(t) & = i^2 t^2 - i^2 \frac{t^2}{2} = - \frac{t^2}{2}.
		\end{equs}
		which is the approximation obtained in \cite{BS}. As a conclusion for these computations, one can observe that the low regularity schemes defined in \eqref{genscheme_low} provides a similar local error as the one defined in \cite{BS}. Moreover in this simple case, the schemes are the same obtained via two different derivations. 
		\end{example}
		
		One important question left is to know if the scheme proposed via normal forms coincides with the general low regularity scheme given in \cite{BS}. Without specific assumptions, one does not expect this to be true. Indeed, in \cite{BS}, one starts the decomposition of the lower and dominant parts on the edges close to the leaves, then one proceeds with a Taylor expansion of the lower part and continues toward the root. The scheme proposed here is different as it starts from the root of the tree and progresses in a certain order given by the arborification map toward the leaves. 
		
		A generalisation will be to give a certain order on the blue edges that will dictate the order in which one performs the discretisation, the scheme in \cite{BS} and the normal form scheme being only two specific examples but also the easiest to implement and derive explicit formulae.

		 Below, we provide the assumption that could guarantee that all these schemes have the same local error and that the order of the Taylor expansion does not matter.
		 The proof of such a statement is rather difficult as the combinatorial description of the low regularity schemes given in \cite{BS} differs from the normal form approach. One first has to rewrite the schemes of \cite{BS} with Definition \ref{adaptative_splitting}, which needs to be adapted as one does not work with words in this setting. Even with this, it is not clear that the statement will be true in full generality.

		 Before our next proposition, we recall the way dominant and lower parts are computed on decorated trees. Given a decorated tree $T$, one has
		 for $ T = \prod_{j=1}^n T_j \star T_r $,
		 \begin{equs} \label{def_dom_T}
		 	\mathscr{F}_{\!{\tiny \text{dom} }}(T) & =  \mathcal{P}_{\!{\tiny \text{dom} }} ( 	\mathscr{F}(T_{r}) + 	\sum_{j=1}^n\mathscr{F}_{\!{\tiny \text{dom} }}(T_{j}) )
		 \end{equs}
		 This is	equivalent to \cite[Def. 2.6]{BS}. Our assumption is the following

		\begin{assumption} \label{Assumption_1} Let $w = T_n...T_1$ on the alphabet $A$ such that $w$ appears in the decomposition of  $\mathfrak{a}(T)$ where $ T $ is a decorated tree, one has 
			\begin{equs} \label{condition_dominant_parts}
				\mathscr{F}_{\normalfont \!{\tiny \text{dom} }}(w) = 
			\mathscr{F}_{\normalfont \!{\tiny \text{dom} }}(T).
			\end{equs}
			\end{assumption}

			\begin{remark} Assumption \eqref{Assumption_1}  guarantees that the dominant parts in the various schemes are the same; the order of the letters coming from the arborification does not matter in the computation. This is also a strong result for the local error analysis, as one can expect the lower parts to have the same degree (lower parts being inhomogeneous monomials) and therefore, one expects the local error analysis to be the same.
				One can potentially argue that the schemes look pretty similar in the end up to the point of being exactly the same in some particular cases like NLS (see Example \ref{discretisation_int}). One can potentially write a similar assumption for the resonance analysis given in Definition \ref{adaptative_splitting}.
				\end{remark}

			 Below, we prove that this Assumption is satisfied for equation of type \eqref{dis}. 
			 
			 \begin{proposition} \label{proof_dis}
			 	One has for $w = T_n...T_1$ where $w$ is  a word in the decomposition of $ \mathfrak{a}(T)$ with $T$ a decorated tree,
			 	\begin{equs}
			 	\mathscr{F}_{\normalfont \!{\tiny \text{dom} }}(T) = 	\mathscr{F}_{\normalfont \!{\tiny \text{dom} }}(w) =
			 		2^{1-\sigma}( (-1)^{b_r} k_r + \sum_{j=1}^m (-1)^{b_j}\ell_j)^{\sigma}
			 	\end{equs}
			 	where $ k_r $ is the  decoration of the node connected to the root of $T$ and the $ \ell_j  $ are the leaves decorations that do not appear in any inner nodes of the decorated trees $T_p, p \in \lbrace 1,...,n \rbrace$. They correspond to the decorations of the leaves of $T$.  One also has
			 	\begin{equs}
			 	 (1+ \sigma) a_r = \sigma b_r \mod 2, \quad (1+ \sigma) a_j +1 = \sigma b_j \mod 2
			 	 \quad 
			 	 \end{equs}
			 	 where $ a_r $ is the decoration of the edge connected to the root and $a_j$ is the decoration of the edge connected to the leaf with decoration $\ell_j$.
			 	\end{proposition}
			 	\begin{proof}
			 		We proceed by recurrence on the size of the word $w$. One has
			 		\begin{equs}
			 			\mathscr{F}_{\!{\tiny \text{dom} }}(w) & =  \mathcal{P}_{\!{\tiny \text{dom} }} \left( 	\mathscr{F}(T_{n}) + 	\mathscr{F}_{\!{\tiny \text{dom} }}(T_{n-1}...T_1) \right).
			 		\end{equs}
			 		By the recurrence hypothesis, one has from Definition \ref{def_dom}
			 		\begin{equs}
			 			\mathscr{F}_{\!{\tiny \text{dom} }}(T_{n-1}...T_1) = 2^{1-\sigma}((-1)^{b_r} k_r + \sum_{j=1}^m (-1)^{b_j}\ell_j)^{\sigma}.
			 		\end{equs} 
			 		The root of $T_n$ is associated with one of the leaves of the $T_j$ as $ T_n $ has been obtained via the arborification procedure. Then, the decoration of the edge connected to the root of $T_n$ is the same as the edge connected to some leaf $\ell_p$ belonging to some tree $T_j$.
			 		Then, one has $(1 + \sigma) a_p = \sigma b_p \mod 2 $ and 
			 		\begin{equs}
			 			\mathscr{F}(T_{n}) = (-1)^{a_p} ((-1)^{a_p}\ell_p)^{\sigma} - \sum_{j=1}^{q} (-1)^{c_j} ((-1)^{c_j}\tilde{\ell}_j)^{\sigma} + \cdots
			 		\end{equs}
			 		where we have considered only the leading terms in $\sigma$. One also has from \eqref{frequencies_identity}
			 		\begin{equs}
			 			(-1)^{a_p}\ell_p = \sum_{j=1}^{q} (-1)^{c_j}\tilde{\ell}_j.
			 		\end{equs}
			 		Then, 
			 		\begin{equs}
			 \,	&	\mathscr{F}(T_{n}) + 	\mathscr{F}_{\!{\tiny \text{dom} }}(T_{n-1}...T_1) \\&= 2^{1-\sigma}((-1)^{b_r} k_r + \sum_{j=1}^m (-1)^{b_j}\ell_j)^{\sigma}  +(-1)^{a_p} ((-1)^{a_p}\ell_p)^{\sigma} \\&  - \sum_{j=1}^{q} (-1)^{c_j} ((-1)^{c_j}\tilde{\ell}_j)^{\sigma} + \cdots
			 		\\ &  = 2^{1-\sigma}((-1)^{b_r} k_r + \sum_{j \neq p} (-1)^{b_j}\ell_j
			 		+ \sum_{j=1}^{q} (-1)^{\hat{c}_j} \tilde{\ell}_j
			 		)^{\sigma} 
			 		\end{equs}
			 		where $ (1+\sigma) c_j +1 = \sigma \hat{c}_j \mod 2$ and  we have neglected the terms of lower order in the $\cdots$. Moreover, we have used the fact that $k_r$ is a linear combination of the $ \ell_j $ with coefficients belonging to $ \lbrace 1,-1 \rbrace $.  One has
			 		\begin{equs}
			 		(-1)^{b_r} k_r + \sum_{j=1}^m (-1)^{b_j}\ell_j = \sum_{j=1}^m d_j (-1)^{b_j}\ell_j
			 		\end{equs}
			 		with $ d_j \in \lbrace 0,2 \rbrace $.
		This allows us to conclude for the explicit formula of $ \mathscr{F}_{\normalfont \!{\tiny \text{\tiny{dom}} }}(w) $.  The proof for  $ \mathscr{F}_{\normalfont \!{\tiny \text{\tiny{dom}} }}(w) $ works a bit the same as for the words. We proceed by induction on the construction of $T$. From \eqref{def_dom_T} one has $ T = \prod_{j=1}^n \hat{T}_j \star T_r $,
		\begin{equs}
			\mathscr{F}_{\!{\tiny \text{dom} }}(T) & =  \mathcal{P}_{\!{\tiny \text{dom} }} ( 	\mathscr{F}(T_{r}) + 	\sum_{j=1}^n\mathscr{F}_{\!{\tiny \text{dom} }}(\hat{T}_{j}) ).
		\end{equs}
		One applies the induction hypothesis to each of the $\hat{T}_j$ and gets
		\begin{equs}
				\mathscr{F}_{\normalfont \!{\tiny \text{dom} }}(\hat{T}_j)	=
			2^{1-\sigma}( (-1)^{b_j} k_j + \sum_{p=1}^{m_j} (-1)^{b_{j,p}}\ell_{j,p})^{\sigma}.
		\end{equs}
		One concludes from the fact that 
		\begin{equs}
				\mathscr{F}(T_{r}) = (-1)^{a_r} ((-1)^{a_r}k_r)^{\sigma} - \sum_{j=1}^{n} (-1)^{c_j} ((-1)^{c_j}k_j)^{\sigma} + \cdots
		\end{equs}
		with $ (1+ \sigma) c_j +1 = \sigma b_j \mod 2 $.
			 		\end{proof}
			 		
			 		\begin{example}
			 			One considers the following decorated trees from cubic NLS
			 			\begin{equs}	T_3  = 
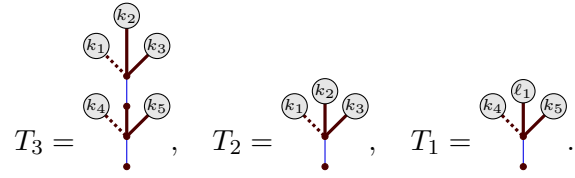
\begin{tikzpicture}[scale=0.2,baseline=-5]
			 				\coordinate (root) at (0,0);
			 				\coordinate (tri) at (0,-2);
			 				\coordinate (t1) at (-2,2);
			 				\coordinate (t2) at (2,2);
			 				\coordinate (t3) at (0,2);
			 				\coordinate (t4) at (0,4);
			 				\coordinate (t41) at (-2,6);
			 				\coordinate (t42) at (2,6);
			 				\coordinate (t43) at (0,8);
			 				\draw[kernels2,tinydots] (t1) -- (root);
			 				\draw[kernels2] (t2) -- (root);
			 				\draw[kernels2] (t3) -- (root);
			 				\draw[symbols] (root) -- (tri);
			 				\draw[symbols] (t3) -- (t4);
			 				\draw[kernels2,tinydots] (t4) -- (t41);
			 				\draw[kernels2] (t4) -- (t42);
			 				\draw[kernels2] (t4) -- (t43);
			 				\node[not] (rootnode) at (root) {};
			 				\node[not] (rootnode) at (t4) {};
			 				\node[not] (rootnode) at (t3) {};
			 				\node[not,label= {[label distance=-0.2em]below: \scriptsize  $  $}] (trinode) at (tri) {};
			 				\node[var] (rootnode) at (t1) {\tiny{$ k_{\tiny{4}} $}};
			 				\node[var] (rootnode) at (t41) {\tiny{$ k_{\tiny{1}} $}};
			 				\node[var] (rootnode) at (t42) {\tiny{$ k_{\tiny{3}} $}};
			 				\node[var] (rootnode) at (t43) {\tiny{$ k_{\tiny{2}} $}};
			 				\node[var] (trinode) at (t2) {\tiny{$ k_5 $}};
			 			\end{tikzpicture},  \quad T_2 =\begin{tikzpicture}[scale=0.2,baseline=-5]
			 				\coordinate (root) at (0,0);
			 				\coordinate (tri) at (0,-2);
			 				\coordinate (t1) at (-2,2);
			 				\coordinate (t2) at (2,2);
			 				\coordinate (t3) at (0,3);
			 				\draw[kernels2,tinydots] (t1) -- (root);
			 				\draw[kernels2] (t2) -- (root);
			 				\draw[kernels2] (t3) -- (root);
			 				\draw[symbols] (root) -- (tri);
			 				\node[not] (rootnode) at (root) {};t
			 				\node[not,label= {[label distance=-0.2em]below: \scriptsize  $ $}] (trinode) at (tri) {};
			 				\node[var] (rootnode) at (t1) {\tiny{$ k_{\tiny{1}} $}};
			 				\node[var] (rootnode) at (t3) {\tiny{$ k_{\tiny{2}} $}};
			 				\node[var] (trinode) at (t2) {\tiny{$ k_3 $}};
			 			\end{tikzpicture},
			 			\quad T_1 = 	\begin{tikzpicture}[scale=0.2,baseline=-5]
			 				\coordinate (root) at (0,0);
			 				\coordinate (tri) at (0,-2);
			 				\coordinate (t1) at (-2,2);
			 				\coordinate (t2) at (2,2);
			 				\coordinate (t3) at (0,3);
			 				\draw[kernels2,tinydots] (t1) -- (root);
			 				\draw[kernels2] (t2) -- (root);
			 				\draw[kernels2] (t3) -- (root);
			 				\draw[symbols] (root) -- (tri);
			 				\node[not] (rootnode) at (root) {};t
			 				\node[not,label= {[label distance=-0.2em]below: \scriptsize  $  $}] (trinode) at (tri) {};
			 				\node[var] (rootnode) at (t1) {\tiny{$ k_{\tiny{4}} $}};
			 				\node[var] (rootnode) at (t3) {\tiny{$ \ell_1 $}};
			 				\node[var] (trinode) at (t2) {\tiny{$ k_5 $}};
			 			\end{tikzpicture}.
			 			\end{equs}
			 			One has
			 			\begin{equs}
			 				\mathscr{F}_{\!{\tiny \text{dom} }}(T_2) = 2 k_1^2 = 2^{1-2} ( -k_1 + k_2 + k_3 -k_1 -k_2 -k_3  )^2.
			 			\end{equs}
			 			Then
			 			\begin{equs}
			 			\mathscr{F}_{\!{\tiny \text{dom} }}(T_3) & = 2^{1-2} ( -k_4 -k_1 + k_2 +k_5 -k_4 -k_1 - k_2 - k_3 -k_5)^2 \\ &= 2 (-k_1 - k_4)^2.
			 			\end{equs}
			 			On the other hand, one has
			 			\begin{equs}
			 		\mathscr{F}_{\!{\tiny \text{dom} }}(T_3) & = \mathcal{P}_{\!{\tiny \text{dom} }} ( 	\mathscr{F}(T_{1}) + 	\mathscr{F}_{\!{\tiny \text{dom} }}(T_{2}) )
			 		\\ & = \mathcal{P}_{\!{\tiny \text{dom} }} ( 	
			 		(-k_4 + \ell_1 + k_5)^2 + k_4^2 - \ell_1^2 - k_5^2 + 	2 k_1^2 )
			 		\\ & =2 (k_1 + k_4)^2.
			 			\end{equs}
			Then, 
				\begin{equs}
				\mathscr{F}_{\!{\tiny \text{dom} }}(T_2 T_1) & = \mathcal{P}_{\!{\tiny \text{dom} }} ( 	\mathscr{F}(T_{2}) + 	\mathscr{F}_{\!{\tiny \text{dom} }}(T_{1}) )
				\\ & = \mathcal{P}_{\!{\tiny \text{dom} }} ( 	
				(-k_1 + k_2 + k_3)^2 + k_1^2 - k_2^2 - k_3^2 + 	2 k_4^2 )
				\\ & =2 (k_1 + k_4)^2.
			\end{equs}
			In the end, one has
			\begin{equs}
				\mathscr{F}_{\!{\tiny \text{dom} }}(T_3) = 	\mathscr{F}_{\!{\tiny \text{dom} }}(T_2 T_1) = 2 (k_1 + k_4)^2.
			\end{equs}
			With the previous identity, we have checked Proposition \ref{proof_dis} for the decorated tree $T_3$ and the word $ w = T_2 T_1 $.
			 		\end{example}

\end{document}